\pdfoutput=1
\documentclass[10pt]{article} 

\usepackage{amsfonts}
\usepackage{amsmath}
\usepackage{amsthm}
\usepackage{amssymb}

\usepackage{mdframed} 
\usepackage{thmtools}

\definecolor{shadecolor}{gray}{0.9}
\declaretheoremstyle[
headfont=\normalfont\bfseries,
notefont=\mdseries, notebraces={(}{)},
bodyfont=\normalfont,
postheadspace=0.5em,
spaceabove=1pt,
mdframed={
  skipabove=8pt,
  skipbelow=8pt,
  hidealllines=true,
  backgroundcolor={shadecolor},
  innerleftmargin=4pt,
  innerrightmargin=4pt}
]{shaded}

\declaretheorem[style=shaded,within=section]{definition}
\declaretheorem[style=shaded,sibling=definition]{theorem}

\declaretheorem[style=shaded,sibling=definition]{assumption}

\declaretheorem[sibling=definition]{remark}
\declaretheorem[sibling=definition]{example}
\declaretheorem[style=shaded,sibling=definition]{lemma}

\usepackage{colordvi,epsfig}
\usepackage[table]{xcolor}   
\usepackage{algorithm}
\usepackage[noend]{algpseudocode}

\usepackage{verbatim,tikz}
\usepackage{caption,subcaption}   
\usepackage[maxbibnames=99, firstinits=true,doi=false,isbn=false,url=false,backend=bibtex]{biblatex}	

\bibliography{./jabref/jabref_library}

\graphicspath{{./figures/}}  
\usepackage{fullpage}

\newcommand{\ignore}[1]{}
\newcommand{\R}{\mathbb{R}}

\newcommand{\iu}{{i\mkern1mu}}

\newcommand{\eqdef}{\overset{\text{def}}{=}} 
\newcommand{\diag}{{\rm Diag}} 
\newcommand{\support}{{\rm supp}} 

\newcommand{\cD}{{\cal D}}

\newcommand{\cG}{{\cal G}}

\newcommand{\cL}{{\cal L}}

\newcommand{\cN}{{\cal N}}
\newcommand{\cO}{{\cal O}}

\newcommand{\mA}{{\bf A}}
\newcommand{\mB}{{\bf B}}

\newcommand{\mD}{{\bf D}}

\newcommand{\mH}{{\bf H}}
\newcommand{\mI}{{\bf I}}
\newcommand{\mJ}{{\bf J}}

\newcommand{\mM}{{\bf M}}
\newcommand{\mN}{{\bf N}}

\newcommand{\mP}{{\bf P}}

\newcommand{\mR}{{\bf R}}
\newcommand{\mS}{\mathbf{S}}

\newcommand{\mW}{{\bf W}}
\newcommand{\mX}{{\bf X}}
\newcommand{\mY}{{\bf Y}}
\newcommand{\mZ}{{\bf Z}}

\newcommand{\Proj}{{\bf \Pi}} 
\newcommand{\Jac}{{ \bf \nabla F}} 
\newcommand{\ones}{e}

\newcommand{\EE}[2]{\mathbb{E}_{#1}\left[#2\right] }
\newcommand{\E}[1]{\mathbb{E}\left[#1\right] } 
\newcommand{\ED}[1]{\mathbb{E}_{\cD}\left[#1\right] } 

\newcommand{\Prb}[1]{\mathbb{P}\left[#1\right] }

\newcommand{\norm}[1]{\left \| #1 \right\|}
\newcommand{\dotprod}[1]{\left< #1\right>}
\newcommand{\Tr}[1]{\mbox{Tr}\left( #1\right)}

\providecommand{\Range}[1]{\mbox{Range}\left( #1\right)}

\usepackage[colorinlistoftodos,bordercolor=orange,backgroundcolor=orange!20,linecolor=orange,textsize=scriptsize]{todonotes}

\newcommand{\LGconst}{L^{\cG}_{\max}}
\newcommand{\LGi}{L^{\cG}_i}

\usepackage{hyperref}

\title{\bf Stochastic Quasi-Gradient Methods: \\ Variance Reduction  via Jacobian Sketching\footnote{
The first  results of this  paper were obtained in Fall 2015 and most key results were obtained by Fall 2016. All key results were obtained by Fall 2017. The first author gave a series of talks on the results (before the paper was released online) in November 2016 (Machine learning seminar at T\'{e}l\'{e}com ParisTech), December 2016 (CORE seminar, Universit\'{e} catholique de Louvain), March 2017 (Optimization, machine learning, and pluri-disciplinarity workshop,  Inria Grenoble - Rhone-Alpes), May 2017 (SIAM Conference on Optimization, Vancouver), September 2017 (Optimization 2017, Faculdade de Ciencias of the Universidade de Lisboa), and November 2017 (PGMO Days 2017, session on Continuous Optimization for Machine Learning, EDF'Lab Paris-Saclay). 
}}

\author{Robert M. Gower\footnote{T\'{e}l\'{e}com ParisTech, France.}  $\qquad$ Peter Richt\'arik\footnote{King Abdullah University of Science and Technology (KAUST), Saudi Arabia --- University of Edinburgh, United Kingdom --- Moscow Institute of Physics and Technology (MIPT), Russia.} $\qquad$  Francis Bach\footnote{INRIA - ENS - PSL Research University, France. 
}}

\date{April 25, 2018}

\begin{document}
\maketitle

\begin{abstract}
We develop a new family of  variance reduced stochastic gradient descent methods for minimizing the average of a very large number of smooth functions. Our method---JacSketch---is motivated by novel developments in randomized numerical linear algebra, and operates by maintaining a stochastic estimate of a Jacobian matrix composed of the gradients of individual functions. In each iteration, JacSketch  efficiently updates the Jacobian matrix by first obtaining a random linear measurement of the true Jacobian through (cheap) sketching, and then projecting the previous estimate onto the solution space of a linear matrix equation whose solutions are consistent with the measurement. The Jacobian estimate is then used to compute a variance-reduced unbiased estimator of the gradient, followed by  a stochastic gradient descent step.  Our strategy is analogous to the way quasi-Newton methods maintain an estimate of the Hessian, and hence our method can be seen as a {\em stochastic quasi-gradient method}.  Indeed, quasi-Newton methods project the current Hessian estimate onto a solution space of a linear equation consistent with a certain linear (but non-random) measurement of the true Hessian.  Our method can also be seen as stochastic gradient descent applied to a  {\em controlled stochastic optimization reformulation} of the original problem, where the control comes from the Jacobian estimates.

We prove that for smooth and strongly convex functions, JacSketch converges linearly with a meaningful rate dictated by a single  convergence theorem which applies to general sketches. We also provide a refined convergence theorem which applies to a smaller class of sketches, featuring a novel proof technique  based on a  {\em stochastic Lyapunov function}. This  enables us to obtain sharper complexity results for variants of JacSketch with importance sampling. By specializing our general approach to specific  sketching strategies, JacSketch reduces to  the celebrated stochastic average gradient (SAGA) method, and its several existing and many new minibatch, reduced memory, and importance sampling variants. Our rate for SAGA with importance sampling is the current best-known rate for this method, resolving a conjecture by Schmidt et al (2015).  The rates we obtain for minibatch SAGA are also superior to  existing rates. Moreover, we obtain the first  minibatch SAGA method with importance sampling.
\end{abstract}

\newpage

\tableofcontents

\newpage
\section{Introduction}\label{sec:intro}

We consider the problem of minimizing the average of a large number of differentiable functions
\begin{equation} \label{eq:prob}
x^*=\arg\min_{x \in \R^d} \left[f(x) \eqdef \frac{1}{n}\sum_{i=1}^n f_i(x) \right],
\end{equation}
where $f$ is $\mu$--strongly convex and $L$--smooth. 
In solving~\eqref{eq:prob}, we restrict our attention to 
first-order methods that use a (variance-reduced) stochastic estimate of the gradient $g^k \approx \nabla f(x^k)$  to take a step towards minimizing~\eqref{eq:prob} by iterating
\begin{equation} \label{eq:xupdate}x^{k+1} = x^k - \alpha g^{k},\end{equation}
where $\alpha>0$ is a  stepsize.

In the context of machine learning, \eqref{eq:prob} is an abstraction of the {\em empirical risk minimization} problem; $x$ encodes the parameters/features of a (statistical) model,  and $f_i$ is the loss of example/data point $i$ incurred by model $x$. The goal is to find the model $x$ which minimizes the average loss on the $n$ observations.

Typically, $n$ is so large that algorithms which rely on  scanning through all $n$ functions in each iteration are too costly. The need for incremental methods for the training phase of machine learning models has revived the interest in the  stochastic gradient descent (SGD) method~\cite{RobbinsMonro:1951}. SGD sets  $g^k=\nabla f_i(x^k)$, where $i$ is an index chosen from $[n]\eqdef \{1,2,\dots,n\}$ uniformly at random. SGD therefore requires only a single data sample to complete a step and make progress towards the solution. Thus SGD scales well in the number of data samples, which is important in several machine learning applications since there many be a large number of data samples. On the downside,  the variance of the stochastic estimates of the gradient produced by SGD does not vanish during the iterative process, which suggests that a decreasing stepsize regime needs to be put into place if SGD is to converge. Furthermore, for SGD to work efficiently, this decreasing stepsize regime needs to be tuned for each application area, which is costly.

\subsection{Variance-reduced methods}

 Stochastic variance-reduced versions of SGD offer a solution to this high variance issue, which improves the theoretical convergence rate and solves the issue with ad hoc stepsize regimes. The first variance reduced method for empirical risk minimization is the stochastic average gradient (SAG) method of Schmidt, Le Roux and Bach~\cite{SAG}. The analysis of SAG is notoriously difficult, which is perhaps due to the estimator of gradient  being biased. Soon afterwards,  the SAG gradient estimator was modified into an unbiased one, which resulted in the  SAGA method  \cite{SAGA_Nips}.  SAGA maintains a  matrix of the latest gradients computed for each datapoint $i$, and uses this matrix to construct a stochastic estimate of the gradient. The analysis of SAGA is dramatically simpler than that of SAG. Another  popular method is SVRG of Johnson and Zhang \cite{Johnson2013} (see also S2GD \cite{S2GD}). SVRG enjoys the same theoretical complexity bound as SAGA, but has a much smaller memory footprint. It is based on an inner-outer loop procedure. In the outer loop, a full pass over data is performed to compute the gradient of $f$ at the current point. In the inner loop, this gradient is modified with the use of cheap stochastic gradients, and steps are taken in the direction of the modified gradients.  A   notable recent addition to the family of variance reduced methods, developed by Nguyen et al~\cite{SARAH},  is known as SARAH. Unlike other methods, SARAH does not use an estimator that is unbiased in the last step. Instead, it is unbiased over a long history of the method.

A fundamentally different way of designing variance reduced methods is to use coordinate descent \cite{Richtarik2014a, PCDM2} to solve the dual. This is what the SDCA  method \cite{SDCA} and its various extensions \cite{dmSDCA} do. The key advantage of this approach is that 
the dual often has a seperable structure in the coordinate space, which in turn means that each iteration of coordinate descent is cheap. Furthermore, SDCA is a variance-reduced method by design since the coordinates of the gradient tend to zero as one approaches the solution.  One of the downsides of SDCA is that it requires calculating Fenchel duals and their derivatives. This issue was later solved by introducing approximations and mapping the dual iterates to the primal space as pointed out in~\cite{SAGA_Nips}. This resulted in primal variants of SDCA such as dual-free SDCA~\cite{dfSDCA}. A primal-dual variant which enables the use of arbitrary minibatch strategies was developped by Qu et al~\cite{Qu2015b}, and is known as QUARTZ.

\subsection{Gaps in our understanding of SAGA}

Despite  significant research into variance-reduced stochastic gradient descent methods for solving \eqref{eq:prob}, there are still big gaps in our understanding of variance reduction. For instance, the current theory supporting the  SAGA algorithm is far from complete. 

 SAGA with uniform probabilities  enjoys the iteration complexity $\cO ( (n+ \tfrac{L_{\max}}{\mu} ) \log \tfrac{1}{\epsilon})$, where $L_{\max} \eqdef \max_i L_i$ and $L_i$ is the smoothness constant of $f_i$. While importance sampling versions of SAGA have proved in practice to produce a speed-up over uniform SAGA~\cite{Schmidtnonuni},  a proof of this speed-up has been elusive. It was conjectured by Schmidt et al.\ \cite{Schmidtnonuni} that a properly designed importance sampling strategy for SAGA should lead to the rate $\cO( (n+ \tfrac{\bar{L}}{\mu} ) \log \tfrac{1}{\epsilon})$, where $\bar{L}=\tfrac{1}{n}\sum_i L_i$. However, no such result was proved. This rate is achieved by, for instance, importance sampling variants of SDCA and QUARTZ \cite{Qu2015b}. However, the analysis only applies to a more specialized version of problem \eqref{eq:prob} (e.g., one needs  an explicit strongly convex regularizer). 
 Second, existing minibatch variants of SAGA do not enjoy the same rate as that offered by methods such as SDCA and QUARTZ. Are the above  issues with  SAGA unavoidable, or  is it the case that our understanding of the method is far from complete? Lastly, no minibatch variant of SAGA with importance sampling is known. 
 
 One of the contributions of this paper is giving positive answers to all of the above questions.

\subsection{Jacobian sketching: a new approach to variance reduction} \label{sec:intro_jacsketch}

Our key contribution in this paper is the introduction of a   novel approach---which we call {\em Jacobian sketching}---to designing and understanding variance-reduced stochastic gradient descent methods for solving  \eqref{eq:prob}. We refer to our method by the name {\em JacSketch}.  We shall now briefly introduce some of the key insights motivating our our approach. 

Let $F:\R^d\to \R^n$ be defined by \begin{equation} \label{eq:F(x)}F(x) \eqdef (f_1(x), \ldots, f_n(x)) \in \R^n,\end{equation} and further let \begin{equation}\label{eq:jac_def}\Jac(x) \eqdef [\nabla f_1(x), \ldots, \nabla f_n(x)] \in \R^{d\times n},\end{equation} be the Jacobian of $F$ at $x$. 

The starting point of our new approach is the following trivial observation: the gradient of $f$ at $x$ can be computed from the Jacobian $\Jac(x)$ by a simple {\em linear transformation:}
\begin{equation}\label{eq:ubif98gf8}
\frac{1}{n}\Jac(x) \ones = \nabla f(x),
\end{equation}
where $\ones$ is the vector of all ones in $\R^n$. This alone is not useful to come up with a better way of estimating the gradient. Indeed, 
formula \eqref{eq:ubif98gf8} has two issues. First, the Jacobian is {\em not} available. If we wanted to compute it, we would need to pay the cost of one pass through the data. Second, even if the Jacobian was available, merely multiplying it by the vector of all ones would cost $\cO(nd)$ operations, which is again a cost equivalent to one pass over data.

Now, let us replace the vector of all ones in \eqref{eq:ubif98gf8} by $e_i \in \R^n$, the unit coordinate/basis vector in $\R^n$. If the index $i$ is chosen randomly from $[n]$, then \begin{equation} \label{eq:bu9d9f}\Jac(x)e_i = \nabla f_i(x),\end{equation} which is a stochastic gradient of $f$ at $x$. In other words, by performing a {\em random linear transformation} of the Jacobian, we have arrived at the classical stochastic estimate of the gradient. This approach does not suffer from the first issue mentioned above as the Jacobian is {\em not needed} at all in order to compute $\nabla f_i(x)$. Likewise, it does not suffer from the second issue; namely, the cost of computing the stochastic gradient is merely $\cO(d)$, and we can avoid a costly pass through the data.\footnote{For the purposes of this narrative it suffices to assume that stochastic gradients can be sampled at cost $\cO(d)$. } 

However, this approach suffers from a new issue: by constructing the estimate this way, we {\em do not learn} from the (random) information collected about the Jacobian in prior iterations, through having access to random linear transformations thereof.   In this paper we take the point of view  that this  is the reason why SGD suffers from large variance.  Our approach towards alleviating this problem is to maintain and update  an estimate $\mJ \in \R^{d \times n}$ of the Jacobian  $\Jac(x).$

Given $x^k \in \R^d$, ideally we would like $\mJ$ to satisfy 
\begin{equation}\label{eq:trivial_system} \mJ = \Jac(x^k),\end{equation}
that is, we would like it to be equal to the true Jacobian.
However, at the same time we do not wish to pay the price of computing it.  Hence, assuming we have an estimate $\mJ^k\in \R^{d\times n}$ of the Jacobian available, 
we instead pick a random matrix $\mS_k\in \R^{n\times \tau}$ from some distribution $\cD$ of matrices\footnote{We will not bother about the distribution from which it is picked at the moment. It suffices to say that virtually all distributions are supported by our theory. However, if we wish to obtain  a practical method, some distributions will make much more sense than others.} and consider the following {\em sketched} version of the linear system \eqref{eq:trivial_system}, with unknown  $\mJ $:
\begin{equation}\label{eq:sketching_equation} \mJ \mS_k = \Jac(x^k) \mS_k \in \R^{d\times \tau}.\end{equation}
This equation generalizes both \eqref{eq:ubif98gf8} and \eqref{eq:bu9d9f}. The left hand side contains the sketched system matrix $\mS_k$ and the unknown matrix $\mJ$, and the right hand side contains a quantity we can measure (through a random linear measurement of the Jacobian, which we assume is cheap). Of course, the true Jacobian solves \eqref{eq:sketching_equation}. However, in general, and in particular when $\tau \ll n$ which is the regime we want to be in for practical reasons, the system \eqref{eq:sketching_equation} will have infinite  $\mJ$ solutions.

We pick a unique solution $\mJ^{k+1}$ as the closest solution of~\eqref{eq:sketching_equation} to our previous estimate $\mJ^k$,  with respect to a weighted Frobenius norm with a positive definite weight matrix $\mW\in \R^{n\times n}$:
\begin{eqnarray} \mJ^{k+1}  = &&  \arg \min_{\mJ\in \R^{d\times n}} \|\mJ - \mJ^k\|_{\mW^{-1}} \label{eq:sketch_n_project}\\
 && \text{subject to} \quad \mJ \mS_k  =  \Jac(x^k) \mS_k, \nonumber
\end{eqnarray}
where 
\begin{equation} \label{eq:fro_norm}\norm{\mX}_{\mW^{-1}} \eqdef \sqrt{\Tr{ \mX \mW^{-1} \mX^\top}}.\end{equation}

In doing so, we have built  a learning mechanism whose goal is to maintain good estimates of the Jacobian throughout the run of method \eqref{eq:xupdate}. These  estimates can  be used to efficiently estimate the gradient by performing a linear transformation similar to \eqref{eq:ubif98gf8}, but with $\Jac(x)$ replaced by the latest estimate of the Jacobian.   In practice, it is important to design sketching matrices   so that the Jacobian sketch $\Jac(x)\mS_k$ can be calculated efficiently. 

The ``sketch-and-project'' strategy \eqref{eq:sketch_n_project}  for updating our Jacobian estimate is analogous to the way quasi-Newton methods update the estimate of the Hessian (or inverse Hessian) \cite{Goldfarb1970, Goldfarb1972, Gower2016, ASMI}. From this perspective, our method can be viewed as  a  {\em stochastic quasi-gradient  method}.\footnote{The term ``quasi-gradient methods'' was popular in the 1980s  \cite{Ermoliev1983}, and refers to  algorithms for solving certain stochastic optimization problems which rely on stochastic estimates of function values and their derivatives. In this paper  we give the term a different meaning by drawing a direct link with quasi-Newton methods.} 


Problem \eqref{eq:sketch_n_project} admits the explicit closed-form solution (see Lemma~\ref{lem:sol}): 
\begin{equation}\label{eq:jacobsolWintro}
\mJ^{k+1} = \mJ^{k} +(\Jac(x^k)-\mJ^{k}) \Proj_{\mS_k},
\end{equation}
where
\begin{equation}\label{eq:PSdef} \Proj_\mS \eqdef  \mS (\mS^\top \mW \mS)^{\dagger} \mS^\top \mW,\end{equation}
is  a projection matrix, and $\dagger$ denotes the Moore-Penrose pseudoinverse. 

\begin{quote} \em The key insight of our work is to propose an efficient Jacobian learning mechanism based on ideas borrowed from recent results in randomized numerical linear algebra. \end{quote}

Having established our update of the Jacobian estimate, we now need to use this to form an estimate of the gradient. Unfortunately, using $\mJ^{k+1}$ in place of $\Jac(x^k)$ in~\eqref{eq:ubif98gf8} leads to a biased gradient estimate (something we explore later in Section~\ref{sec:sketchproject}). To obtain an unbiased estimator of the gradient, we introduce a {\em stochastic relaxation parameter} $\theta_{\mS_k}$ and use
\begin{equation} \label{eq:g^k=intro}g^k \eqdef \frac{1-\theta_{\mS_k}}{n} \mJ^k \ones + \frac{\theta_{\mS_k}}{n} \mJ^{k+1} \ones = \frac{1}{n}\mJ^k \ones + \frac{1}{n}(\Jac(x^k) - \mJ^k) \theta_{\mS_k} \Proj_{\mS_k} \ones,\end{equation}
as an approximation of the gradient.
%
 Taking expectations in \eqref{eq:g^k=intro} over $\mS^k\sim \cD$ (for this we use the notation $\ED{\cdot} \equiv \EE{\mS_k\sim \cD}{\cdot}$), we get
\begin{equation}\label{eq:g9dg98ddkhjk}\ED{g^k} =  \frac{1}{n}\mJ^k \ones + \frac{1}{n}(\Jac(x^k) - \mJ^k) \ED{\theta_{\mS_k} \Proj_{\mS_k} \ones}.\end{equation}
Provided that \begin{equation}\label{eq:bias-corr-eq-intro}\ED{\theta_{\mS_k} \Proj_{\mS_k} \ones} = \ones,\end{equation} we have 
$\ED{g^k} \overset{\eqref{eq:g9dg98ddkhjk}}{=} \frac{1}{n}\Jac(x^k) \ones \overset{\eqref{eq:ubif98gf8}}{=} \nabla f(x^k)$, and hence, $g^k$ is a unbiased estimate of the gradient.
If \eqref{eq:bias-corr-eq-intro} holds, we say that $\theta_{\mS_k}$ is a  {\em bias-correcting random variable} and $\mS^k$ is an {\em unbiased sketch.}  Our new {\em JacSketch} method is method \eqref{eq:xupdate} with $g^k$ computed as via \eqref{eq:g^k=intro} and Jacobian estimate updated via \eqref{eq:jacobsolWintro}. This method is formalized in Section~\ref{sec:controlled} as Algorithm~\ref{alg:SketchJac}.

This strategy  indeed works, as we show in detail in this paper. Under appropriate conditions (on the stepsize $\alpha$, properties of $f$ and randomness behind the sketch matrices $\mS_k$ and so on), the variance of $g^k$ diminishes to zero (e.g., see Lemma~\ref{lem:gradient_bounddeltaXX}), which means that JacSketch is a variance-reduced method. We perform an analysis for smooth and strongly convex functions $f$, and obtain a linear convergence result (Theorem~\ref{theo:convgen}). We summarize our complexity results in detail in Section~\ref{sec:intro-summary-of-results}.

\subsection{SAGA as a special case of JacSketch} \label{sec:SAGA-intro}

Of particular importance in this paper are {\em minibatch sketches}, which are sketches of the form $\mS_k = \mI_{S_k}$, where $S_k$ is a random subset of $[n]$, and $\mI_{S_k}$ is a random column submatrix of the $n\times n$ identity matrix with columns indexed by $S_k$. For minibatch sketches, JacSketch corresponds to minibatch variants of SAGA. Indeed, in this case, and if $\mW=\diag(w_1,\dots,w_n)$, we have $\Proj_{\mS_k} \ones = e_{S_k}$, where
$e_{S} = \sum_{i\in S} e_i$ (see Lemma~\ref{prop:bus98g90s09}). Therefore, 
\begin{equation} \label{eq:g^k-SAGA} g^{k} = \frac{1}{n} \mJ^k \ones + \frac{\theta_{\mS_k}}{n} \sum_{i\in S_k} (\nabla f_i(x^k)- \mJ^k_{:i}) . \end{equation}
In view of \eqref{eq:jacobsolWintro}, and since $\Proj_{\mS_k} = \mI_{S_k} \mI_{S_k}^\top$ (see Lemma~\ref{prop:bus98g90s09}), the Jacobian estimate gets updated as follows
\begin{equation} \label{eq:SAGA-Jacobian_update} \mJ^{k+1}_{:i} = \begin{cases} \mJ^k_{:i} & \quad i\notin S_k,\\
 \nabla f_i(x^k)&  \quad i\in S_k. \end{cases}
\end{equation}

Standard uniform SAGA is obtained by setting $S_k = \{i\}$ with probability $1/n$ for each $i\in [n]$, and letting $\theta_{\mS_k} \equiv n$. SAGA with arbitrary probabilities is obtained by instead choosing  $S_k = \{i\}$ with probability $p_i>0$ for each $i\in [n]$, and letting $\theta_{\mS_k} \equiv \tfrac{1}{p_i}$. However, virtually all minibatching and importance sampling strategies can be treated as special cases of our general approach.

The theory we develop answers the open questions raised earlier. In particular, we answer the conjecture of Schmidt et al.\  \cite{Schmidtnonuni}  about the rate of SAGA with importance sampling in the affirmative. In particular, we establish the iteration complexity $(n+ \frac{4\bar{L}}{ \mu} ) \log \tfrac{1}{\epsilon}.$ This complexity is obtained for {\em different} importance sampling distributions than that currently proposed in the literature for SAGA.
In order to achieve this, we develop a new analysis technique which makes use of a {\em stochastic Lyapunov function} (see Section~\ref{sec:theopart}). That is, our Lyapunov function has a random element which is independent of the randomness inherited from the iterates of the method. This is unlike any other Lyapunov function used in the analysis of stochastic methods we are aware of.  Further,  we prove that SAGA converges with any initial matrix $\mJ^0$ in place of the matrix of gradients of functions $f_i$ at the starting point. We also show that our results give better rates for minibatch SAGA than are currently known, even for uniform minibatch strategies. We also allow for a family of completely new uniform minibatching strategies which were not considered  in connection with SAGA before, and consider also SAGA with importance sampling for minibatches\footnote{For some prior results on importance sampling for minibatches, in the context of QUARTZ, see \cite{Csiba_imp_minibatch}.} (based on a partition of $[n]$). 
  Lastly,  as a special case, our method recovers standard gradient descent, together with the sharp iteration complexity of  $\frac{4 L}{\mu}\log \tfrac{1}{\epsilon}$.


Our general approach also enables a novel {\em reduced memory} variant of SAGA as a special case. Let $\mS_k=e_{S_k}$, and choose $\mW = \mI.$ Since $\Proj_{\mS_k} \ones = e_{S_k}$, the formula for $g^k$ is the same as in the case of SAGA, and is given by~\eqref{eq:g^k-SAGA}. What is notably different about this sketch (compared to $\mI_{S_k}$) is that, since $\Proj_{\mI_{S_k}} = \frac{1}{|S_k|}e_{S_k} e_{S_k}^\top,$ the update of the Jacobian estimate~\eqref{eq:jacobsol} is given by
\[\mJ^{k+1} \overset{\eqref{eq:jacobsolWintro}}{=} \mJ^{k} -\frac{1}{|S_k|} \sum_{i \in S_k}\left(
\mJ^{k}_{:i} -\nabla f_i(x^k)\right) e_{S_k}^\top. \]

Thus, {\em the same update is applied to all the columns of $\mJ^k$ that belong to $S_k$.} Equivalently, this update can be written as
\begin{equation}\label{eq:jacreducedmem} \mJ^{k+1}_{:j} =
\begin{cases}
 \frac{1}{|S_k|} \sum_{i\in S_k} \nabla f_{i}(x^k) & \mbox{if } j \in S_k,\\
\mJ^{k}_{:j} &  \mbox{if } j \notin S_k.
\end{cases} \end{equation}

In particular, if $S_k$ only ever picks sets which correspond to a partition of $[n]$, and we initialize $\mJ^0$ so that all the columns belonging to the same partition are the same, then they will  be the same within in each partition for all $k$. In such a case, we do not need to maintain all the identical copies. Instead, we can update and use a condensed/compressed version  of the Jacobian, with one column per partition set only, to reduce the total memory usage. This method, with non-uniform probabilities, is analyzed in our framework in Section~\ref{subsec:calitercomspec}.

\subsection{Sketch and project}

In the case when $\Jac(x)$ is a constant matrix (i.e., does not depend on $x$), randomized iterative methods of the form \eqref{eq:sketch_n_project} for solving linear systems such as \eqref{eq:trivial_system} were recently proposed and analyzed by Gower and Richt\'{a}rik \cite{Gower2015}. For obvious reasons, an iteration of the form \eqref{eq:sketch_n_project} was therein called {\em sketch and project}. In the present context, they show that under weak assumptions on the distribution $\cD$ from which the matrices are sampled (in an i.i.d.\ fashion), the sketch and project method converges linearly to the unique solution of \eqref{eq:trivial_system}. For instance, if $\mS_k$ are unit coordinate vectors  in $\R^n$ chosen uniformly at random, then the theory in \cite{Gower2015, Gower2015c} suggests that the method converges in just $\cO(n \log \tfrac{1}{\epsilon})$ steps in expectation. This rate is to be expected as this choice of  sketching vectors enables us to learn an entire column of the true Jacobian in each iteration. Under this particular choice of the sketching matrix, the sketch and project method for solving  \eqref{eq:trivial_system} is closely related to the   randomized Kaczmarz method of Strohmer and Vershynin \cite{Strohmer2009}.

It has long been known, and was explored in detail by Needell, Srebro and Ward \cite{Needell2014}, that the randomized Kaczmarz method is a specific instantiation of SGD, applied to a suitable least-squares type function. In the context of sketch and project methods with arbitrary sketching matrices $\mS_k$, this was explored by Richt\'{a}rik and Tak\'{a}\v{c} \cite{ASDA}, who also demonstrated that the sketch and project method, and hence also our Jacobian learning iteration \eqref{eq:sketch_n_project}, can be interpreted 
as stochastic gradient descent applied to a suitable stochastic optimization problem.
Therefore, and quite surprisingly:

\begin{quote} \em In our Jacobian sketching framework, variance reduction is obtained by applying SGD to the problem of learning the Jacobian. So, our method uses SGD  in two different ways: as a method for performing the  step toward minimizing the loss (this is standard), and as a method for learning the Jacobian which is then used to lower the variance of the search direction (this is our new insight).
\end{quote}

As a follow up to \cite{Gower2015}, Gower and Richt\'{a}rik  further extended their analysis in \cite{Gower2015c} to arbitrary consistent linear systems (i.e., beyond systems with a single solution, such as \eqref{eq:trivial_system}). Therein they show that the sketch and project method converges linearly to the projection of the starting iterate onto the solution space of the system, and also uncover a dual interpretation of the method as stochastic dual subspace ascent. Related ideas were later used to design stochastic algorithms for inverting matrices \cite{Gower2016} and computing the pseudoinverse of a rectangular matrix~\cite{Gower2016end}. For a  compendium of some of the above papers on sketch and project, see also \cite{GowerThesis}.

An accelerated (in the sense of Nesterov) sketch and project method   was proposed and analyzed in \cite{ASDA}. However, the analysis was restricted to a  weak type of convergence. This was remedied  by Tu et al.\ \cite{TuVWGJR17} for positive definite systems and a special class of sketchings,  by Richt\'{a}rik and Tak\'{a}\v{c} \cite{ASDA2} for general linear systems and general sketchings, and further extended to Euclidean setting and applied to matrix inversion and quasi-Newton updates by Gower et al.\ \cite{ASMI}. A sketch and project method with the heavy ball momentum was studied in \cite{BasicMomentum}.


\subsection{Controlled stochastic reformulation}

Loosely motivated by \cite{ASDA}, we shall explore an alternative narrative to the sketch-and-project motivation described above. In particular,  the development of JacSketch can instead be motivated through the lens of {\em controlled stochastic reformulations} of \eqref{eq:prob}. 

Let us now very briefly outline the main idea. First, we will use the distribution $\cD$ from which the sketching matrices are drawn to define a stochastic optimization reformulation of  problem~\eqref{eq:prob}. That is, we write $f$ as an expectation over some carefully constructed functions $f_\mS(x)$ instead, where the expectation is taken over $\mS\sim \cD$. We then add a ``smart'' zero function, also of the form of an expectation of some functions over $\cD$, to this reformulation. However, this zero perturbation depends on $\mJ^k$. While this does not change the objective function, it affects the stochastic gradients in a positive way: it reduces their variance. We then apply an SGD step to this perturbed (or ``controlled'') reformulation, followed by an update of the Jacobian (through sketch and project). This is iterated until convergence, and results in JacSketch. This alternative narrative is provided in Section~\ref{sec:controlled}.


\subsection{Summary of complexity results}\label{sec:intro-summary-of-results}
  
   \begin{table}
 { \footnotesize
\begin{center}
\begin{tabular}{|c|c|c|c|c|}
\hline
ID & \bf Method & \begin{tabular}{c} \bf Sketch  $\mS\in \R^{n\times \tau}$ \\ \bf $\mW \succ 0$ \end{tabular} &   \bf Iteration complexity ($\times \log \tfrac{1}{\epsilon}$) & \bf Reference \\
 \hline
 \hline
1 & JacSketch & \begin{tabular}{c}any unbiased \\ any \end{tabular} & $ \max \left\{ \frac{4\cL_1}{\mu}, \, \frac{1}{\kappa} + \frac{4\rho \cL_2}{\kappa \mu n^2 }  \right\}$ & Thm~\ref{theo:convgen} \\ 
\hline
2& \begin{tabular}{c} JacSketch \\ 
(with any probabilities \\
for $\tau$--partition)
\end{tabular}& \begin{tabular}{c} $\mI_S$ \\ $\mI$ \end{tabular} & $ \max_{C \in \support(S)}  \left( \frac{1}{p_C}+ \frac{\tau}{n p_C }\frac{ 4 L_{C}}{\mu } \right) $ & Thm~\ref{theo:convpart} \\  
\hline
\hline
3 & Gradient descent & \begin{tabular}{c} $\mI $ \\ $\mI$\end{tabular} &  $\tfrac{4 L}{\mu} $ &  \begin{tabular}{c}Thm~\ref{theo:convgen} \\ \eqref{eq:bu98dg8sbh*YH} \end{tabular} \\
 \hline
4& Gradient descent & \begin{tabular}{c} $\mI $ \\ $\mI$ \end{tabular} &  $\frac{4L}{\mu}$ &  \begin{tabular}{c}Thm~\ref{theo:convpart} \\ \eqref{eq:ih80089h0fh89s8*^b} \end{tabular} \\
 \hline
 \hline 
 5&\begin{tabular}{c} SAGA \\ (with uniform sampling) \end{tabular} & \begin{tabular}{c}$\mI_S$  \\ $\mI$ \end{tabular} & $n+ \frac{4 L_{\max}}{\mu}$ & \begin{tabular}{c} Thm~\ref{theo:convgen}  \\ \eqref{eq:bu98dg8sbh*YH097809} \end{tabular}\\
 \hline
6&\begin{tabular}{c} SAGA \\ (with uniform sampling) \end{tabular} &  \begin{tabular}{c} $\mI_S$  \\ $\mI$ \end{tabular} & $n+ \frac{4 L_{\max}}{\mu}$ & \begin{tabular}{c} Thm~\ref{theo:convpart}  \\ \eqref{eq:ihb98sg9s08gb} \end{tabular}\\
 \hline
7& \begin{tabular}{c} SAGA \\ (with importance sampling) \end{tabular} & \begin{tabular}{c} $\mI_S$  \\  --- \end{tabular} & no improvement on uniform sampling & \begin{tabular}{c} Thm~\ref{theo:convgen} \\  \end{tabular}\\
\hline
8&\begin{tabular}{c}  SAGA \\  (with importance sampling) \end{tabular} & \begin{tabular}{c} $\mI_S$  \\ $\mI$ \end{tabular} & {$n+ \frac{4 \bar{L}}{\mu}$} & \begin{tabular}{c} Thm~\ref{theo:convpart} \\ \eqref{eq:itercomplexopt} \end{tabular}\\
\hline
\hline
9 & \begin{tabular}{c}  Minibatch SAGA \\  ($\tau$--uniform sampling)  \end{tabular}  &  \begin{tabular}{c}   $\mI_S$  \\ $\diag(w_i)$ \end{tabular} &  $ \max \left\{\frac{4 \LGconst }{\mu}, \,
\frac{n}{\tau} + \frac{4 \rho}{\mu n}\max_{i} \left( \frac{L_i}{w_i} \right)     \right\} $ & \begin{tabular}{c} Thm~\ref{theo:convgen} \\ \eqref{eq:itercompluni} \end{tabular} \\ 
\hline
10& \begin{tabular}{c}   Minibatch SAGA \\  ($\tau$--nice sampling)  \end{tabular}  &  \begin{tabular}{c}  $\mI_S$  \\ $\mI$ \end{tabular} & {$\max \left\{ \frac{4 \LGconst}{\mu}, \,
 \frac{n}{\tau} + \frac{n-\tau}{(n-1)\tau} \frac{4 L_{\max}}{\mu}      \right\} $ } & \begin{tabular}{c} Thm~\ref{theo:convgen} \\ \eqref{eq:no8sh09hs9h0(((} \end{tabular} \\
 \hline
11& \begin{tabular}{c}   Minibatch SAGA \\   ($\tau$--nice sampling)  \end{tabular}  &   \begin{tabular}{c} $\mI_S$ \\ $ \diag(L_i)$ \end{tabular} & {$\max \left\{ \frac{4 \LGconst}{\mu}, \, \frac{n}{\tau} + 
 \frac{n-\tau}{n \tau} \frac{4 (\bar{L}+L_{\max})}{\mu}      \right\}$} & \begin{tabular}{c} Thm~\ref{theo:convgen} \\ \eqref{eq:nh09s8h09sh9JJ} \end{tabular}\\
 \hline 
12 & \begin{tabular}{c}  Minibatch SAGA \\  ($\tau$--partition sampling)  \end{tabular}  &  \begin{tabular}{c}  $\mI_S$ \\ $\mI$ \end{tabular} & $ \frac{n}{\tau} + \frac{4 L_{\max}}{\mu}  $ & \begin{tabular}{c} Thm~\ref{theo:convgen}\\ \eqref{eq:nbo9s80s9h} \end{tabular}\\
 \hline
13 & \begin{tabular}{c}  Minibatch SAGA \\  ($\tau$--partition sampling)  \end{tabular}  &  \begin{tabular}{c}  $\mI_S$ \\ $ \diag(L_i)$ \end{tabular} & $
 \frac{n}{\tau} + \frac{4 \max_{C\in \support(S)} \frac{1}{ \tau}\sum_{i\in C} L_i}{\mu}       $ & \begin{tabular}{c} Thm~\ref{theo:convgen} \\ \eqref{eq:bs98g9ofgebd} \end{tabular}\\
  \hline
14 & \begin{tabular}{c} Minibatch SAGA \\ (importance $\tau$--partition \\  sampling) 
\end{tabular} &  \begin{tabular}{c}  $\mI_S$ \\ $\mI$ \end{tabular} &{ $\frac{n}{\tau}+ \frac{4 \frac{1}{|\support(S)|}\sum_{C\in \support(S)} L_C}{\mu}$} & \begin{tabular}{c} Thm~\ref{theo:convpart} \\ \eqref{eq:itercomplexoptmini} \end{tabular}\\
 \hline 
\end{tabular}

\end{center}
}
\caption{Special cases of our JacSketch method, and the  associated iteration complexity. All methods converge linearly. In the iteration complexity column we list the number of iterations sufficient to obtain an $\epsilon$ accurate solution, ignoring a  $\log \tfrac{1}{\epsilon}$ factor.}
\label{tbl:complexity_summary}
\end{table}  

All  convergence results obtained in this paper are summarized in Table~\ref{tbl:complexity_summary}. 

Our convergence results depend on several constants which we will now briefly introduce. The precise definitions can be found in the main text. For $\emptyset \neq C\subseteq [n]=\{1,2,\dots,n\}$, define $f_C(x) \eqdef \frac{1}{|C|}\sum_{i\in C} f_i(x)$. We assume $f_C$ is $L_C$--smooth\footnote{A formal definition can be found in Assumption~\ref{ass:L_C-smoothness}.}. We let $L_i=L_{\{i\}}$,   $L = L_{[n]}$, $L_{\max} = \max_i L_i$ and $\bar{L}=\tfrac{1}{n}\sum_i L_i$. Note that $L_i \leq L_{\max}$, $\bar{L} \leq L_{\max} \leq n \bar{L}$, $L_C \leq \tfrac{1}{|C|}\sum_{i\in C} L_i$ and $L \leq \bar{L}$. For a sampling\footnote{In this paper, a sampling is a random set-valued mapping with the sets being subsets of $[n]$.} $S\subseteq [n]$, we let $\support(S) = \{C \subseteq [n] \;:\; \Prb{S = C} >0\}$. That is, the support of a sampling is the number of sets which are selected by this sampling with positive probability.  Finally, $\LGconst = \max_i \tfrac{1}{c_1}\sum_{C \in \support(S), i \in C} L_C$, where $c_1$ is the cardinality of the set $\{ C \;:\; C \in \support(S), i \in C\}$ (which is assumed to be the same for all $i$). So, $\LGconst$ is the maximum over $i$ of averages of values $L_C$ for those sets  $C$ which are picked by $S$ with positive probability and which contain $i$. Clearly, $\LGconst \leq L_{\max}$ (see Theorem~\ref{lem:two_ineqXXX}).

\paragraph{General theorem.} Theorem~\ref{theo:convgen} is our most general result, allowing for any(unbiased) sketch $\mS$ (see  \eqref{eq:bias-corr-eq-intro}), and any weight matrix $\mW\succ 0$. The resulting iteration complexity given by this theorem is
\[  \max \left\{ \frac{4\cL_1}{\mu}, \, \frac{1}{\kappa} + \frac{4\rho \cL_2}{\kappa \mu n^2 }  \right\} \times  \log \left(\frac{1}{\epsilon}\right) , \]
and is also presented in the first row of Table~\ref{tbl:complexity_summary}. This result depends on two {\em expected smoothness} constants $\cL_1$ (measuring the expected smoothness of the stochastic gradient of our stochastic reformulation; see Assumption~\ref{ass:ES1}) and $\cL_2$ (measuring the expected smoothness of the Jacobian; see Assumption~\ref{ass:ES2}). The complexity also depends on the {\em stochastic condition number} $\kappa$ (see \eqref{eq:kappa}) and the {\em sketch residual} $\rho$ (see \eqref{eq:rhointro} and \eqref{eq:itercomplexgen}). We devote considerable effort to give simple formulas for these constants under some specialized settings (for special combinations of sketches $\mS$ and weight matrices $\mW$). In fact, the entire Section~\ref{sec:minibatch_sketches} is devoted to this. In particular, all rows of Table~\ref{tbl:complexity_summary} where the last column mentions Theorem~\ref{theo:convgen} arise as special cases of the general iteration complexity in the first row.

\begin{itemize}
\item {\bf Gradient descent.} As a starting point, in row 3 we highlight that one can recover gradient descent as a special case of JacSketch with the choice $\mS = \mI$ (with probability 1) and $\mW=\mI$. We get the rate $\tfrac{4L}{\mu} \log \tfrac{1}{\epsilon}$, which is tight. 

\item {\bf SAGA with uniform sampling.} Let us now focus on a slightly more interesting special case: row 5. We see that SAGA with uniform probabilities appears as a special case, and enjoys the rate $(n+ \tfrac{4L_{\max}}{\mu}) \log \tfrac{1}{\epsilon}$, recovering an existing result. 

\item {\bf SAGA with importance sampling.} Unfortunately, the generality of Theorem~\ref{theo:convgen} comes at a cost: we are not able to obtain an importance sampling version of SAGA as a special case which would have a better iteration complexity than uniform SAGA. This will be remedied by our second complexity theorem, which we shall discuss later below. 

\item {\bf Minibatch SAGA. } Rows 9-13 correspond to minibatch versions of SAGA. In particular, row 9 contains a general statement (albeit still a special case of the statement in row 1), covering virtually all minibatch strategies. Rows 10-13 specialize this result to two particular minibatch sketches (i.e., $\mS=\mI_S$), each with two choices of $\mW$. The first sketch corresponds to samplings $S$ which choose from among all subsets of $[n]$ uniformly at random. This sampling is known in the literature as $\tau$-nice sampling \cite{PCDM, ESO}. The second sketch corresponds to $S$ being a $\tau$--partition sampling. This sampling picks uniformly at random subsets of $[n]$ which form a partition of $[n]$, and are all of cardinality $\tau$. Notice that the complexities in rows 10 and 11 are comparable (each can be slightly better than the other, depending on the values of the smoothness constants $\{L_i\}$). On the other hand, in the case of $\tau$--partition, the choice $\mW=\diag(L_i)$ is better than $\mW = \mI$: the complexity in row 13 is better than that in row 12. This is because
$\max_{C\in \support(S)} \frac{1}{\tau} \sum_{i\in C} L_i \leq L_{\max}.$ 
  
\item {\bf Optimal minibatch size for SAGA. } Our analysis for mini-batch SAGA also gives  the first iteration complexities that interpolate between the $(n+\frac{4L_{\max}}{\mu})\log \tfrac{1}{\epsilon}$  complexity of SAGA and the $\frac{4L}{\mu}\log \tfrac{1}{\epsilon}$  complexity  of gradient descent, as $\tau$ increases from $1$ to $n$. Indeed, consider the complexity in rows~10, 11 and 13 for $\tau =1$ and $\tau =n.$  Our iteration complexity of mini-batch SAGA is the first result that is precise enough to inform an optimal mini-batch size (see Section~\ref{subsec:optimalminibatch}).  In contrast, the previous best complexity result for mini-batch SAGA~\cite{HofmannLLM152015} interpolates between $(n+\frac{4L_{\max}}{\mu}) \log \tfrac{1}{\epsilon}$ and $\frac{4L_{\max}}{\mu} \log \tfrac{1}{\epsilon}$ as $\tau$ increases from $1$ to $n$, and thus is not precise enough as to inform the best minibatch size. We make a more detailed comparison between our results and~\cite{HofmannLLM152015} in Section~\ref{sec:hofmann}. 
\end{itemize}

\paragraph{Specialized theorem.} We now move to the second main complexity result of our paper: Theorem~\ref{theo:convpart}. The general complexity statement is listed in row 2 of Table~\ref{tbl:complexity_summary}: \begin{equation} \label{eq:nbo089hf09} \max_{C\in \support(S)} \left( \frac{1}{p_C} + \frac{\tau}{np_C} \frac{4 L_C}{\mu}\right) \times \log \left(\frac{1}{\epsilon}\right),\end{equation} where $p_C = \Prb{S=C}$. This theorem is a refined result specialized to minibatch sketches ($\mS=\mI_S$) with $\tau$--partition samplings $S$. This is a sampling which picks subsets of $[n]$ of size $\tau$  forming a partition of $[n]$, uniformly at random. Our focus on $\tau$--partition samplings  enables us to provide stronger iteration complexity guarantees for non-uniform probabilities.  

\begin{itemize}
\item {\bf Gradient descent.} As a starting point, we point out that just like Theorem~\ref{theo:convgen},  Theorem~\ref{theo:convpart} also recovers the correct complexity of gradient descent as a special case (this is when $S=[n]$ with probability 1); this can be seen in row 4.  Indeed, in this case we have $S=[n]$ with probability 1 (hence, $p_{[n]}=1$), $\support(S) = \{[n]\}$, $\tau=n$ and $L_{[n]} = L$. Hence, \eqref{eq:nbo089hf09} specializes to $\frac{4 L}{\mu} \log \frac{1}{\epsilon}$.

\item {\bf SAGA with importance sampling.} The first remarkable special case of \eqref{eq:nbo089hf09}  is summarized in row~8, and corresponds to SAGA with importance sampling. The complexity obtained, $(n+ \tfrac{4\bar{L}}{\mu}) \log \tfrac{1}{\epsilon}$, answers a conjecture of Schmidt et al.\  \cite{Schmidtnonuni} in the affirmative. In this case, the support of $S$ are the singletons $\{1\}$, $\{2\}, \dots, \{n\}$, $p_{\{i\}} =  p_i$ for all $i$, $\tau=1$ and $L_{\{i\}} = L_i$. Optimizing the complexity bound over the probabilities $p_1,\dots,p_n$, we obtain the importance sampling
$p_i = \frac{\mu n + 4\tau L_i}{\sum_j \mu n + 4\tau L_j}.$

\item {\bf Minibatch SAGA with importance sampling.} In row 14 we state the complexity for a minibatch SAGA method with importance sampling. This is the first result for this method in the literature. Note that by comparing rows 13 and 14, we can conclude that the complexity of minibatch SAGA with importance sampling is better than for minibatch SAGA with uniform probabilities. Indeed, this is because\footnote{We prove inequality \eqref{eq:inew_intro} in the appendix; see Lemma~\ref{lem:two_ineq}.}
\begin{equation} \label{eq:inew_intro}\frac{1}{|\support(S)|}\sum_{C\in \support(S)} L_C \leq \bar{L} \leq \max_{C\in \support(S)} \frac{1}{\tau} \sum_{i\in C} L_i.\end{equation}

\end{itemize}

\subsection{Outline of the paper}

We present an alternative narrative  motivating the development of JacSketch in Section~\ref{sec:controlled}. This narrative is based on a novel technical tool which we call {\em controlled stochastic optimization reformulations} of problem~\eqref{eq:prob}.  We then develop a general convergence theory of JacSketch in Section~\ref{sec:theogen}. This theory admits practically any sketches $\mS$ (including minibatch sketches mentioned in the introduction) and weight matrices $\mW$. The main result in this section is Theorem~\ref{theo:convgen}.  In Section~\ref{sec:minibatch_sketches} we specialize the general results to minibatch sketches. Here we also compute the various constants appearing in the general complexity result for JacSketch for specific classes of minibatch samplings. In Section~\ref{sec:theopart} we develop an alternative theory for JacSketch, one based on a novel {\em stochastic Lyapunov function}. The main result in this section is Theorem~\ref{theo:convpart}. Computational experiments are included in Section~\ref{sec:experiments}.  

 \subsection{Notation}

We will introduce notation when and as needed. If the reader would like to recall any notation, for ease of reference we have a notation glossary in Section~\ref{sec:notation_glossary}. As a general rule, all matrices are written in upper-case bold letters. By $\log t$ we refer to the natural logarithm of $t$.
\section{Controlled Stochastic Reformulations} \label{sec:controlled}

In this section we provide an alternative narrative behind the development of JacSketch; one through the lens of what we call {\em controlled stochastic reformulations}. These reformulations are a novel technical tool enabling us to view JacSketch from a novel perspective.

We design our family of methods so that two keys properties are satisfied, namely {\em unbiasedness},
$
\E{g^k}  =  \nabla f(x^k),
$
 and {\em diminishing variance:}
$
\E{\norm{g^k - \nabla f(x^k) }_2^2} \longrightarrow 0
$ as $x^k \rightarrow x^*$. These are  both favoured statistical properties. Moreover, currently only methods that have diminishing variance exhibt fast linear convergence (exponential decay of the error) on strongly convex problems. On the other hand, unbiasedness is not necessary for a fast method in practice since several biased stochastic gradient methods such as SAG~\cite{SAG} perform well in practice. Still, the absence of bias greatly facilitates the analysis of JacSketch.

\subsection{Stochastic reformulation using sketching} \label{sec:sgdreform}

It will be useful to formalize the condition mentioned in Section~\ref{sec:intro_jacsketch} which leads to  $g^k$ being an unbiased estimator of the gradient. 

\begin{assumption}[Unbiased sketch]\label{def:sketch}
Let $\mW\succ 0$ be a weighting matrix and let $\mD$ be the distribution from which the sketch matrices $\mS$ are drawn.
	 There exists a random variable $\theta_{\mS}$ such that
\begin{equation}\label{eq:unbiased}
\ED{\theta_{\mS} \Proj_\mS} \ones  =  \ones.
\end{equation}

When this assumption is satisfied, we say that $(\mS, \theta_\mS, \mW)$ constitutes an ``unbiased sketch'',  and we call $\theta_{\mS}$ the  bias-correcting random variable. When the triple is obvious from the context, sometimes we shall simply say that $\mS$ is an unbiased sketch.
\end{assumption}

The first key insight of this section is that  besides producing unbiased estimators of the gradient, unbiased sketches produce {\em unbiased estimators of the loss function} as well.   Indeed, by simply observing that $f(x) = \frac{1}{n}\dotprod{F(x),\ones}$, we get 
\[f(x) \overset{\eqref{eq:prob}}{=}   \frac{1}{n}\sum_{i=1}^n f_i(x)    = \frac{1}{n}\dotprod{F(x),\ones}  
\overset{\eqref{eq:unbiased}}{=}   \frac{1}{n}\dotprod{F(x),\ED{\theta_{\mS} \Proj_\mS \ones}} = \ED{\frac{1}{n}\dotprod{F(x),\theta_{\mS} \Proj_\mS \ones}}.
\]
 In other words, we can rewrite the finite-sum optimization problem \eqref{eq:prob} as an equivalent stochastic optimization problem where the randomness comes from $\cD$ rather than from the representation-specific uniform distribution over the $n$ loss functions:
 \begin{equation}\label{eq:probstochvarh}
 \min_{x \in \R^d} f(x) = \ED{f_{\mS}(x)}, \qquad \text{where} \qquad f_{\mS}(x) \eqdef \frac{\theta_{\mS}}{n}\dotprod{F(x), \Proj_\mS \ones}.
 \end{equation}
The stochastic optimization problem \eqref{eq:probstochvarh} is  a 
 {\em stochastic reformulation} of the original problem \eqref{eq:prob}.
 Further, the stochastic gradient of this reformulation is given by
\begin{equation}
\label{eq:stochgradplain}
\nabla f_{\mS}(x)   = \frac{\theta_{\mS}}{n}\Jac(x) \Proj_\mS \ones.
\end{equation}
With these simple observations, our options  at designing stochastic gradient-type algorithms for  \eqref{eq:prob} have suddenly broadened dramatically. Indeed, we can now solve the problem, at least in principle, by applying SGD to any stochastic reformulation:
\begin{equation}\label{eq:SGD_stoch_reform}x^{k+1} = x^k - \alpha \nabla f_{\mS_k} (x^k).  \end{equation}
 But now we have a parameter to play with, namely, the distribution of $\mS$. The choice of this parameter will influence both the iteration complexity of the resulting method as well as the cost of each iteration. We now give a few examples of possible choices of $\cD$ to illustrate this.
 
\begin{example}[gradient descent] \label{ex:1} Let $\mS$ be equal to $\mI$  (or any other $n\times n$ invertible matrix) with probability 1 and let $\mW \succ 0$ be chosen arbitrarily.  Then  $\theta_{\mS} \equiv 1$ is bias-correcting since
\[\ED{ \theta_\mS \Proj_\mS \ones } =\Proj_{\mS} \ones \overset{\eqref{eq:PSdef}}{=}  \mS(\mS^\top \mW \mS)^{\dagger} \mS^\top \mW \ones = \mS \mS^{-1}\mW^{-1} (\mS^{\top})^{-1} \mS^\top \mW \ones = \mI \ones = \ones.\] 
With this setup, the SGD method \eqref{eq:SGD_stoch_reform} becomes {\em gradient descent:}
\begin{equation} \label{eq:GD098098} x^{k+1} = x^k -  \alpha \nabla f_{\mS_k} (x^k)  \overset{\eqref{eq:ubif98gf8}+\eqref{eq:stochgradplain}}{=} x^k -  \alpha \nabla f (x^k).  \end{equation}
\end{example}

\begin{example}[SGD with non-uniform sampling] \label{ex:2}
Let $\mS=e_i$ (unit basis vector in $\R^n$) with probability  $p_i>0$   and let $\mW = \mI$.  Then  $\theta_{e_i} = 1/p_i$   is bias-correcting since
\[\ED{ \theta_\mS \Proj_\mS \ones } \overset{\eqref{eq:PSdef}}{=}  \sum_{i=1}^n p_i \frac{1}{p_i} e_i(e_i^\top e_i)^{-1} e_i^\top \ones   =\sum_{i=1}^n  e_i  e_i^\top \ones = \mI \ones = \ones.\] 

Let  $S_k= \{i_k\}$ be picked at iteration $k$. Then the SGD method \eqref{eq:SGD_stoch_reform} becomes {\em SGD with non-uniform sampling:}
\begin{equation} \label{eq:GD0jbs99}  x^{k+1} = x^k -  \alpha \nabla f_{\mS_k} (x^k)  \overset{\eqref{eq:stochgradplain}}{=} x^k -  \frac{\alpha}{n p_{i_k}} \nabla f_{i_k} (x^k).  \end{equation}
Note that with this setup, and when $p_i=1/n$ for all $i$, the stochastic reformulation is identical to the original finite-sum problem. This is the case because $f_{e_i}(x) = f_i(x)$.
\end{example}

\begin{example}[minibatch SGD]
Let $\mS=e_{S} = \sum_{i\in S}e_i$, where $S=C\subseteq [n]$ with probability $p_C$. Let $\mW = \mI$.  Assume that the cardinality of the set $\{C \subseteq [n]\;:\;C\in \support(S), \; i\in C  \}$ does not depend on $i$  (and is equal to $c_1>0$). Then  $\theta_{e_S} = 1/(c_1 p_S)$  is bias-correcting since
\[\ED{ \theta_\mS \Proj_\mS \ones } \overset{\eqref{eq:PSdef}}{=}  \sum_{C\in \support(S)} p_C \frac{1}{c_1 p_C} e_C(\underbrace{e_C^\top e_C}_{|C|})^{-1} \underbrace{e_C^\top \ones}_{|C|}   = \sum_{C\in \support(S)}  \frac{1}{c_1} e_C  = \ones.\] 
Note that $\Proj_{e_S}\ones = e_S$. Assume that set $S_k$ is picked in iteration $k$.Then the SGD method \eqref{eq:SGD_stoch_reform} becomes {\em minibatch SGD with non-uniform sampling:}
\begin{equation} \label{eq:ug9d082j0nd} x^{k+1} = x^k -  \alpha \nabla f_{\mS_k} (x^k)  \overset{\eqref{eq:stochgradplain}}{=} x^k -  \frac{\alpha}{n c_1}\sum_{i\in S_k}\frac{1}{p_{S_k}} \nabla f_i (x^k).  \end{equation}
Finally, note that gradient descent \eqref{eq:GD098098} is a special case of \eqref{eq:ug9d082j0nd} if we set $p_{[n]} = 1$ and $p_{C}=0$ for all other subsets $C$ of $[n]$. Likewise, SGD with non-uniform probabilities \eqref{eq:GD0jbs99} is a  special case of \eqref{eq:ug9d082j0nd} if we set $p_{ \{i\} } = p_i>0$ for all $i$ and $p_{C}=0$ for all other subsets $C$ of $[n]$.
\end{example}

\subsection{The controlled stochastic reformulation}
\label{sec:sgdcontrolreform}

Though SGD applied to the stochastic reformulation can generate several known algorithms in special cases,  there is no reason to believe that the gradient estimates $g^k$ will have diminishing variance (excluding the extreme case such as gradient descent). Here we handle this issue using {\em control variates}, a commonly used tool to reduce variance in Monte Carlo methods~\cite{hickernell2005}.

Given a random function $z_{\mS}(x)$, we introduce the \emph{controlled stochastic reformulation}:
  \begin{equation}\label{eq:probstochcon}
  \min_{x \in \R^d} f(x) = \ED{f_{\mS,z}(x)}, \qquad \text{where} \qquad f_{\mS,z}(x) \eqdef f_{\mS}(x) -z_{\mS}(x) + \ED{z_{\mS}(x)}.
  \end{equation}

Since
\begin{equation}\label{eq:controlgrad} \nabla f_{\mS,z}(x) \eqdef  \nabla f_{\mS}(x) - \nabla z_{\mS}(x) + \ED{\nabla z_{\mS}(x)}\end{equation}
is an unbiased estimator of the gradient $\nabla f(x)$, we can  apply SGD to the controlled stochastic reformulation instead, which leads to the method
\[x^{k+1} = x^k -\alpha ( \nabla f_{\mS_k}(x) - \nabla z_{\mS_k}(x) + \ED{\nabla z_{\mS}(x)}).\]
Reformulation  \eqref{eq:probstochvarh} and method \eqref{eq:SGD_stoch_reform} is recovered as a special case with the choice $z_\mS(x) \equiv 0$. However, we now have the extra freedom to choose $z_{\mS}(x)$ so as to control the variance of this stochastic gradient. In particular, if $\nabla z_{\mS}(x)$ and $ \nabla f_{\mS}(x)$ are sufficiently correlated, then~\eqref{eq:controlgrad} will have a smaller variance than $  \nabla f_{\mS}(x).$ For this reason, we choose a linear model for $z_{\mS}(x)$ that mimicks the stochastic function $f_{\mS}(x).$ 

Let $\mJ \in \R^{d \times n}$ be a matrix of parameters of the following linear model
\begin{equation}\label{eq:covariate}
z_{\mS}(x) \eqdef \frac{\theta_{\mS}}{n} \dotprod{\mJ^\top x,   \Proj_\mS\ones}, \qquad \nabla z_{\mS}(x) =  \frac{\theta_{\mS}}{n} \mJ\, \Proj_\mS\ones. \end{equation}
 Note that this linear model has the same structure as $f_{\mS}(x)$ in~\eqref{eq:probstochvarh} except that $F(x)$ has been replaced by the linear function $\mJ^\top x.$ 
If  $\mS$ is an unbiased sketch (see \eqref{eq:unbiased}), we get  $\ED{\nabla z_{\mS}(x)} = \frac{1}{n} \mJ \ones $, which plugged into~\eqref{eq:probstochcon} and~\eqref{eq:controlgrad} together with the definition \eqref{eq:probstochvarh} of $f_\mS$  gives the following unbiased estimate of $f(x)$ and $\nabla f(x)$:
  \begin{equation}\label{eq:controlfuncJ} 
  f_{\mS,\mJ}(x) \eqdef f_{\mS,z}(x) = \frac{\theta_{\mS}}{n} \dotprod{F(x)- \mJ^\top x,  \Proj_\mS\ones}  + \frac{1}{n}\dotprod{\mJ^\top x, \ones},\end{equation}
 and
 \begin{equation}\label{eq:controlgradJ} 
 \nabla f_{\mS,\mJ}(x) \eqdef \nabla f_{\mS,z}(x) = \frac{\theta_{\mS}}{n} (\Jac(x)-\mJ)   \Proj_\mS\ones  + \frac{1}{n} \mJ \ones.\end{equation}

We collect this observation that~\eqref{eq:controlgradJ} is unbiased in the following lemma for future reference.
\begin{lemma}
	If $\mS$ is an unbiased sketch (see Definition~\ref{def:sketch}), then 
	\begin{equation}\label{eq:unbiasedgrad}
	\ED{ \nabla f_{\mS,\mJ}(x)} = \nabla f(x), \end{equation}
	for every $\mJ \in \R^{d \times n}$ and $x\in \R^d$.
	That is,  \eqref{eq:controlgradJ} is an unbiased estimate of the gradient~\eqref{eq:prob}. 
\end{lemma}
Now it remains to choose the matrix $\mJ$, which  we do by minimizing the variance of our gradient estimate.

\subsection{The Jacobian estimate, variance reduction and the sketch residual}
\label{sec:jacupdate}

Since~\eqref{eq:controlgradJ} gives an unbiased estimator of $\nabla f(x)$  for all $\mJ \in \R^{d \times n}$, we can attempt to choose $\mJ$ that minimizes its variance. Minimizing the variance of~\eqref{eq:controlgradJ} in terms of $\mJ$ will, for all sketching matrices of interest, lead to $\mJ = \Jac(x).$ This follows because
 \begin{eqnarray}
 \ED{\norm{\nabla f_{\mS,\mJ}(x) - \nabla f(x)}_2^2} &\overset{\eqref{eq:controlgradJ}}{=} & 
 \ED{\norm{\frac{1}{n} \mJ( \mI-\theta_{\mS} \Proj_\mS)\ones -\frac{1}{n}\Jac(x)( \mI -\theta_{\mS} \Proj_\mS)\ones  }_2^2} \notag\\
 &= &  \frac{1}{n^2}\ED{\norm{(\mJ-\Jac(x))(\mI -\theta_{\mS} \Proj_\mS)\ones  }_2^2} \nonumber\\
 & =& \frac{1}{n^2}\Tr{(\mJ-\Jac(x))^\top (\mJ-\Jac(x)) \mB},\nonumber\\
 &= & \frac{1}{n^2} \|\mJ-\Jac(x)\|_{\mB}^2,\label{eq:minofvariance}
 \end{eqnarray} 
 where
 \begin{eqnarray} \mB &\eqdef & \ED{( \mI -\theta_{\mS} \Proj_\mS)\ones \ones^\top ( \mI-\theta_{\mS} \Proj_\mS^\top)} \overset{\eqref{eq:unbiased}}{=} \ED{\theta_{\mS}^2 \Proj_{\mS} \ones \ones^\top \Proj_{\mS}^\top} - \ones \ones^\top \succeq 0,\label{eq:Bmat}\end{eqnarray}
 and we have used the weighted Frobenius norm with weight matrix $\mB$ (see \eqref{eq:fro_norm}).  
 
 For most distributions $\cD$ of interest, the matrix $\mB$ is positive definite\footnote{Excluding such trivial cases as when $\mS$ is an invertible matrix and $\theta_{\mS} =1$ with probability one, in which case $\mB =0$}.  Letting $v_\mS \eqdef ( \mI -\theta_{\mS} \Proj_\mS)\ones$, we can bound the largest eigenvalue of matrix $\mB$ via Jensen's inequality as follows: \[\lambda_{\max}(\mB) \overset{\eqref{eq:Bmat}}{=} \lambda_{\max}(\ED{v_\mS v_\mS^\top}) \leq  \ED{\lambda_{\max}(v_\mS v_\mS^\top)} =  \ED{\|v_\mS\|_2^2}.  \]
 Combined with \eqref{eq:minofvariance}, we get the the following bound on the variance of $\nabla f_{\mS,\mJ}$:
 \[\ED{\norm{\nabla f_{\mS,\mJ}(x) - \nabla f(x)}_2^2} \leq \frac{\ED{\|v_\mS\|_2^2}}{n^2} \|\mJ - \Jac(x)\|_\mI^2.\]
This suggests that the variance is low when $\mJ$ is close to the true Jacobian $\Jac(x)$, and when the second moment of $v_\mS$ is small. If $\mS$ is an unbiased sketch, then $\ED{v_\mS}=0$, and hence $\ED{\|v_\mS\|_2^2}$ is the variance of $v_\mS$. So, the lower the variance of $\tfrac{1}{n}\theta_\mS \Proj_\mS e$ as an estimator of  $\tfrac{1}{n}e$, the lower the variance of $\nabla f_{\mS,\mJ}(x)$ as an estimator of $\nabla f(x)$.

Let us now return to the identity  \eqref{eq:minofvariance} and its role in choosing $\mJ$. Minimizing the variance in a single step is overly ambitious, since it requires setting $\mJ = \Jac(x)$, which is costly. So instead, we propose to minimize~\eqref{eq:minofvariance} iteratively.  But first, to make~\eqref{eq:minofvariance} more manageable, we upper-bound it using a norm defined by the weight matrix $\mW$ as follows
 \begin{equation}\label{eq:BupperWbnd}
 \|\mJ-\Jac(x)\|_{\mB}^2 \quad \leq \quad \rho \,\|\mJ-\Jac(x) \|_{\mW^{-1}}^2,
 \end{equation}
 where 
	\begin{equation}\label{eq:rhointro}
	\rho \eqdef \lambda_{\max}\left(\mW^{1/2} \mB \mW^{1/2}\right) \geq 0
	\end{equation}
is the largest eigenvalue of $\mW^{1/2} \mB \mW^{1/2}$.		We refer to the constant $\rho$ as the \emph{sketch residual}, and it is a key constant affecting the convergence rate of JacSketch as captured by Theorem~\ref{theo:convgen}. 	The sketch residual $\rho$ represents how much information is ``lost'' on average due to sketching and due to how well $\mW^{-1}$ approximates $\mB$.  We develop formulae and estimates of the sketch residual for several specific sketches of interest in Section~\ref{sec:computeconsts}.

\begin{example}[Zero sketch residual] Consider the setup from Example~\ref{ex:1} (gradient descent). That is, let $\mS$ be invertible with probability one and let $\theta_{\mS}=1$ be the bias-reducing variable. Then $\Proj_\mS \ones = \ones$ and hence
$\mB=0$, which means that $\rho =0$. 
\end{example}

\begin{example}[Large sketch residual] Consider the setup from Example~\ref{ex:2} (SAGA with non-uniform probabilities).  That is, let $\mS=e_i$ (unit basis vector in $\R^n$) with probability  $p_i>0$   and let $\mW = \mI$.  Then  $\theta_{e_i} = 1/p_i$  is a bias-reducing variable, and it is easy to show that $\mB  = \diag(1/p_1,\dots,1/p_n) - \ones \ones^\top$. If we choose $p_i=1/n$ for all $i$, then 
$\rho =n$. 
\end{example}
		
	We have switched from the $\mB$ norm to a user-controlled $\mW^{-1}$ norm because minimizing under the $\mB$ norm will prove to be impractical because $\mB$ is a dense matrix for most all practical sketches. With this norm change we now have the option to set $\mW$ as a sparse matrix (e.g., the identity, or a diagonal matrix), as we explain in Remark~\ref{rem:W} further down. However,  the theory we develop allows for any symmetric positive definite matrix $\mW$. 
	    

 We can now minimize~\eqref{eq:BupperWbnd} iteratively by only using  a single sketch of the true Jacobian at each iteration. Suppose we have a current estimate $\mJ^k$ of the true Jacobian and a sketch of the true Jacobian $\Jac(x^k) \mS_k$. With this we can calculate an improved Jacobian estimate using a projection step
 \begin{equation}
 \label{eq:jacupdateFdualintro} 
 \mJ^{k+1}= 
 \underset{\mJ \in \R^{d \times n}}{\arg} \underset{\mY \in \R^{ m \times \tau}}{\min\phantom{g}} \frac{1}{2}\norm{\mJ - \Jac(x^k)}_{\mW^{-1}}^2 \quad \mbox{subject to} \quad \mJ = \mJ^{k} + \mY \mS_k^\top \mW,
 \end{equation}
the solution of which, as it turns out,  depends on $\Jac(x^k)$  through its sketch  $\Jac(x^k) \mS_k$ only.  That is, we choose the next Jacobian estimate $\mJ^{k+1}$ as close as possible to the true Jacobian $\Jac(x^k)$ while restricted  to a matrix subspace that passes through $\mJ^k$. Thus in light of~\eqref{eq:BupperWbnd}, the variance is decreasing.   The explicit solution to~\eqref{eq:jacupdateFdualintro} is given by
 \begin{equation}\label{eq:jacobsol} 
 \mJ^{k+1} = \mJ^{k} -(\mJ^{k}-\Jac(x^k)) \Proj_{\mS_k}.
  \end{equation}
 See Lemma~\ref{lem:sol} in the appendix for the proof. Note that, as alluded to before, $\mJ^{k+1}$ depends on $\Jac(x^k)$ through its sketch only. Note that~\eqref{eq:jacobsol} updates the Jacobian estimate  by re-using the  sketch  $\Jac(x^k) \mS_k$ which we also use when calculating the stochastic gradient~\eqref{eq:controlgradJ}.

 Note that \eqref{eq:jacobsol} gives the same formula for $\mJ^{k+1}$ as \eqref{eq:jacobsolWintro} which we obtained by solving \eqref{eq:sketch_n_project}; i.e., by projecting $\mJ^k$ onto the solution set of \eqref{eq:sketching_equation}. This is not a coincidence. In fact, the optimization problems \eqref{eq:sketch_n_project} and \eqref{eq:jacupdateFdualintro}  are mutually dual. This is formally stated in Lemma~\ref{lem:sol} which can be found in the appendix.  In the context of solving linear systems, this was observed in \cite{Gower2015}. Therein,  \eqref{eq:sketch_n_project} is called the sketch-and-project method, whereas \eqref{eq:jacupdateFdualintro}  is called the {\em constrain-and-approximate} problem. In this sense, the Jacobian sketching narrative we followed in Section~\ref{sec:intro_jacsketch} is dual to the Jacobian sketching narrative we are pursuing here.

 
 \begin{remark}[On the weight matrix and the cost] \label{rem:W}
 Loosely speaking, the denser the weighting matrix $\mW$, the higher the computational cost  for updating the Jacobian using~\eqref{eq:jacobsol}. Indeed, the sparsity pattern of $\mW$ controls how many elements of the previous Jacobian estimate $\mJ^k$ need to be updated. This can be seen by re-arranging~\eqref{eq:jacobsol} as
\begin{equation}\label{eq:jacuprearrangY}
 \mJ^{k+1} = \mJ^{k} + \mY_k \mS_k^\top  \mW, \end{equation}
where
$ \mY_k = (\Jac(x^k) \mS_k-\mJ^{k} \mS_k) (\mS_k^\top \mW \mS_k)^{\dagger} \in \R^{d \times \tau}.$
Although we have no control over the sparsity of $\mY_k$, the matrix $\mS_k^\top \mW$ can be sparse when both $\mS_k$ and $\mW$ are sparse. This will be key in keeping the update~\eqref{eq:jacuprearrangY} at a cost propotional to $d \times \tau$, as oppossed to $n\times d$ when $\mW$ is dense.
This is why we consider a diagonal matrix $\mW = \diag(w_1,\ldots, w_n)$ in all of the special complexity results in Table~\ref{tbl:complexity_summary}. While it is clear that some   non-diagonal sparse matrices $\mW$ could also be used,  we leave such considerations to future work.
 \end{remark}

\subsection{JacSketch Algorithm}

   Combining formula~\eqref{eq:controlgradJ} for the stochastic gradient of the controlled stochastic reformulation with formula~\eqref{eq:jacobsol} for the update of the Jacobian estimate, we arrive at our JacSketch algorithm (Algorithm~\ref{alg:SketchJac}).

\begin{algorithm}
	\begin{algorithmic}[1]
		\State \textbf{Input:} $\left(\cD, \mW, \theta_{\mS} \right)$
		
		\State \textbf{Initialize:}  $x^0\in \R^d$, Jacobian estimate $\mJ^0 \in \R^{d \times n}$, stepsize $\alpha>0$ 		
		
		\For {$k =  0, 1, 2, \dots$}
		\State Sample a fresh copy $\mS_k\sim \cD$
		
		\State Calculate $\Jac(x^k) \mS_k$ \Comment{Sketch the Jacobian}

		\State $ \mJ^{k+1} = \mJ^{k} + (\Jac(x^k) - \mJ^{k}) \Proj_{\mS_k} = \mJ^k(\mI - \Proj_{\mS_k}) + \Jac(x^k)\Proj_{\mS_k}$
		 \label{ln:jacupdate}		
		\Comment Update Jacobian estimate

		\State $g^{k} =   \frac{1}{n}\mJ^k \ones +\frac{\theta_{\mS_k}}{n} (\Jac(x^k)-\mJ^{k} ) \Proj_{\mS_k} \ones = \frac{1-\theta_{\mS_k}}{n} \mJ^k \ones + \frac{\theta_{\mS_k}}{n} \mJ^{k+1} \ones$       \label{ln:gradupdate} \Comment Update gradient estimate

		\State $x^{k+1} = x^k - \alpha g^{k}$ \label{ln:xupdate}		
		\Comment Take a step
		 
		\EndFor
	\end{algorithmic}
	\caption{JacSketch: Variance Reduced Gradient Method via Jacobian Sketching}
	\label{alg:SketchJac}
\end{algorithm}
   
Typically, one should not implement the algorithm as presented above.  That is, we do not suggest that, say, in step 6, one explicitly multiplies $ \Jac(x^k)$ by $\Proj_{\mS_k}$ and $\mJ^k$ by $\Proj_{\mS_k}$ and then subtracts the latter from the former. Nor do we suggest that the result is then multiplied by $\ones$ and $\theta_{\mS_k}$ in step 7, and so on. The most efficient implementation of JacSketch will depend heavily on the the structure of $\mW$, distribution $\cD$ and so on. For instance, in the special case of minibatch SAGA, as presented in Section~\ref{sec:SAGA-intro}, the update of the Jacobian \eqref{eq:SAGA-Jacobian_update} has a particularly simple form. That is, we maintain a single matrix $\mJ\in \R^{d\times n}$ and keep replacing its columns by the appropriate stochastic gradients, as computed. Moreover, in the case of linear predictors, as is well known, a much more memory-efficient implementation is possible. In particular, if $f_i(x) = \phi_i(a_i^\top x)$ for some loss function $\phi_i$ and a data vector $a_i\in \R^d$ and all $i$, then $\nabla f_i(x) = \phi_i'(a_i^\top x) a_i$, which means that the gradient always points in the same direction. In such a situation, it is sufficient to keep track of the loss derivatives $ \phi_i'(a_i^\top x)$ only. Similar comments can be made about the step \eqref{eq:g^k-SAGA}  for computing the gradient estimate $g^k$.

From the point of view of the controlled stochastic reformulation, JacSketch can also be written in the form of Algorithm~\ref{alg:SketchJac2}.

\begin{algorithm}
	\begin{algorithmic}[1]
		\State \textbf{Input:} $\left(\cD, \mW, \theta_{\mS} \right)$
		
		\State \textbf{Initialize:}  $x^0\in \R^d$, Jacobian estimate $\mJ^0 \in \R^{d \times n}$, stepsize $\alpha>0$ 		
		
		\For {$k =  0, 1, 2, \dots$}
		\State Sample a fresh copy $\mS_k\sim \cD$

		\State $ \mJ^{k+1}  = \mJ^k(\mI - \Proj_{\mS_k}) + \Jac(x^k)\Proj_{\mS_k}$
		 \label{ln:jacupdate}		
		\Comment Update linear model

		\State $g^{k} =  \nabla f_{\mS_k, \mJ^k}(x^k)$       \label{ln:gradupdate} \Comment compute stochastic gradient for the controlled stochastic reformulation		
					
		\State $x^{k+1} = x^k - \alpha g^{k}$ \label{ln:xupdate}		
		\Comment Take a step
		 
		\EndFor
	\end{algorithmic}
	\caption{JacSketch: Variance Reduced Gradient Method via Jacobian Sketching}
	\label{alg:SketchJac2}
\end{algorithm}

\subsection{A window into biased estimates and SAG}
\label{sec:sketchproject}
We will now take a small detour from the main flow of the paper to develop an alternative viewpoint of Algorithm~\ref{alg:SketchJac} and also make a bridge to biased methods such as SAG~\cite{SAG}. 

The simple observation that 
\begin{equation} \label{eq:Jac1grad}
\nabla f(x^k) = \frac{1}{n}\Jac(x^k)\ones,\end{equation}
suggests that $\hat{g}^k = \frac{1}{n} \mJ^{k+1} \ones$, where $ \mJ^{k+1} \approx \Jac(x^k)$ would give a good estimate of the gradient. To decrease the variance of $\hat{g}^k$, we can also use the same update of the Jacobian estimate~\eqref{eq:jacobsol}  since
 \begin{eqnarray}
 \E{\norm{\hat{g}^k - \nabla f(x^k)}_2^2} &= & 
\frac{1}{n^2} \E{\norm{( \mJ^{k+1} -\Jac(x^k)) \ones }_2^2} \notag \\
 &= &  \frac{1}{n^2} \E{\norm{( \mJ^{k+1} -\Jac(x^{k})) \mW^{-1/2} \mW^{1/2} \ones }_2^2}\nonumber \\
 &\leq  & \frac{\ones^\top \mW \ones}{ n^2} \E{\norm{ \mJ^{k+1} -\Jac(x^{k})}_{\mW^{-1}}^2}.\nonumber
 \end{eqnarray}
 Thus, if $\E{\norm{ \mJ^{k+1} -\Jac(x^{k})}_{\mW^{-1}}^2}$ converges to zero, so will $ \E{\norm{\hat{g}^k - \nabla f(x^k)}_2^2}.$
Though unfortunately, the combination of the gradient estimate  $\hat{g}^k = \frac{1}{n} \mJ^{k+1} \ones$ and a Jacobian estimate updated via~\eqref{eq:jacobsol} will almost always  give a biased estimator.
For example, if we define $\cD$ by setting $\mS= e_i$ with probability $\frac{1}{n}$ and let $\mW = \mI$, then we recover the celebrated SAG method~\cite{SAG} and its biased estimator of the gradient.

The issue with using $\frac{1}{n} \mJ^{k+1} \ones$  as an estimator of the gradient is that it decreases the variance too aggressively,  neglecting the bias. However, this can be fixed  by trading off variance for bias. One way to do this is to introduce  the random variable $\theta_{\mS}$ as a {\em stochastic relaxation parameter}
\begin{equation}\label{eq:gk}
\hat{g}^{k} =   \frac{1-\theta_{\mS_k}}{n}\mJ^{k}\ones +\frac{\theta_{\mS_k}}{n}  \mJ^{k+1} \ones.
\end{equation}

If $\theta_{\mS}$ is bias correcting, we recover the unbiased SAGA estimator \eqref{eq:g^k=intro}. By allowing $\theta_{\mS}$ to be closer to one, however, we will get more bias and lower variance. We leave this strategy of building biased estimators for future work. It is conceivable that SAG could be analyzed using reasonably small modifications of the tools developed in this paper. Doing this would be important due to at least four reasons: i) SAG was the first variance-reduced method for problem~\eqref{eq:prob}, ii) the existing analysis of SAG is not satisfying, iii) one may be able to obtain a better rate, iv) one may be able to develop and analyze novel variants of SAG.
\section{Convergence Analysis for General Sketches}
\label{sec:theogen}

In this section we establish a convergence theorem (Theorem~\ref{theo:convgen}) which applies to general sketching matrices $\mS$ (that is, arbitrary distributions $\cD$ from which they are sampled). By design, we keep the setting in this section general, and only deal with specific instantiations and special cases in Section~\ref{sec:minibatch_sketches}.

\subsection{Two expected smoothness constants} \label{sec:L1-L2}

We first formulate two {\em expected smoothness} assumptions tying together $f$, its Jacobian $\Jac(x)$ and the distribution $\cD$ from which we pick sketch matrices $\mS$. These assumptions, and the associated expected smoothness constants, play a key role in the convergence result.

Our first assumption concerns the expected smoothness of the stochastic gradients $\nabla f_{\mS}$ of the stochastic reformulation \eqref{eq:probstochvarh}.\footnote{A similar relation to \eqref{eq:ES1}  holds for the stochastic optimization reformulation of linear systems studied by Richt\'{a}rik and Tak\'{a}\v{c}~\cite{ASDA}. Therein, this relation holds as an identity with $\cL_1=1$ (see Lemma 3.3 in \cite{ASDA}). However, the function $f_{\mS}$ considered there is entirely different and, moreover,  $f(x^*)=0$ and $\nabla f_{\mS}(x^*)=0$ for all $\mS$. }

\begin{assumption}[Expected smoothness of the stochastic gradient] \label{ass:ES1} There is a constant $\cL_1>0$ such that
\begin{equation} \label{eq:ES1} \ED{ \norm{ \nabla f_{\mS}(x) - \nabla f_{\mS}(x^*)}_2^2 } \leq 2  \cL_1 (f(x)-f(x^*)), \qquad \forall x\in \R^d.\end{equation}
\end{assumption}

It is easy to see from \eqref{eq:stochgradplain} and \eqref{eq:controlgradJ}  that  \begin{equation}\label{eq:8gd088gs899}\norm{\nabla f_{\mS}(x) - \nabla f_{\mS}(y)}_2^2 = \tfrac{1}{n^2}\| (\Jac(x) - \Jac(y))\theta_{\mS} \Proj_{\mS} \ones \|_2^2 =\norm{\nabla f_{\mS, \mJ}(x) - \nabla f_{\mS,\mJ}(y)}_2^2\end{equation} for all $\mJ\in \R^{d\times n}$ and $x,y\in \R^d$, and hence the expected smoothness assumption can equivalently be understood from the point of view of the controlled stochastic reformulation. The above assumption is not particularly restrictive. Indeed, in Theorem~\ref{lem:Lfhatuni} we provide formulae for $\cL_1$ for smooth functions $f$ and for a class of minibatch samplings $\mS=\mI_S$. These formulae can  be seen as proofs that Assumption~\ref{ass:ES1} is satisfied for a large class of practically relevant sketches $\mS$ and functions $f$.  However, as we have seen when discussing the complexity results summarized  in Table~\ref{tbl:complexity_summary}, these formulae are also useful in our interpretation of the resulting convergence rates of variants of JacSketch.

Our second expected smoothness assumption concerns the Jacobian of $F$.

\begin{assumption}[Expected smoothness of the Jacobian] \label{ass:ES2}  There is a constant $\cL_2>0$ such that 		
\begin{equation}
	\ED{\norm{(\Jac(x)-\Jac(x^*)) \Proj_{\mS}  }_{\mW^{-1}}^2 } \leq  2\cL_2 (f(x) -f(x^*)), \qquad \forall x\in \R^d, \label{eq:ES2}
 \end{equation}
 	where the norm is the weighted Frobenius norm defined in \eqref{eq:fro_norm}.
\end{assumption}

It is easy to see (see Lemma~\ref{lem:bo98gd}, Eq.~\eqref{eq:ih0shg09s}) that for any matrix $\mM\in \R^{d\times n}$, we have
$\ED{\|\mM \Proj_{\mS}\|_{\mW^{-1}}^2} = \|\mM\|_{\ED{\mH_{\mS}}}^2,$ where \begin{equation}\label{eq:H_S}\mH_\mS \eqdef \mS (\mS^\top \mW \mS^\top)^\dagger \mS^\top \overset{\eqref{eq:PSdef} }{=} \Proj_{\mS} \mW^{-1}.\end{equation} Therefore, \eqref{eq:ES2} can be equivalently written in the form 
\begin{equation}
\norm{\Jac(x)-\Jac(x^*)   }_{\ED{\mH_{\mS}}}^2  \leq  2\cL_2 (f(x) -f(x^*)), \qquad \forall x\in \R^d, \label{eq:ES2-equiv}
 \end{equation}
which suggests that the above condition indeed measures the variation/smoothness of the Jacobian under a specific weighted Frobenius norm. To the best of our knowledge, the above expected smoothness conditions are new, and have not been considered in the literature before. 

\subsection{Stochastic condition number} \label{sec:kappa}

By the {\em stochastic condition number} associated with $\mW$ and $\cD$ we mean the constant defined by
		\begin{equation} \label{eq:kappa}
	 \kappa = \kappa(\cD,\mW) \eqdef \lambda_{\min} (\ED{\Proj_\mS}).
		\end{equation}
In the next lemma we show that $0 \leq \kappa \leq 1$ for all distributions $\cD$ for which the expectation \eqref{eq:kappa} exists.

\begin{lemma} \label{lem:stochcondition}
For all distributions $\mathcal{D},$ we have the bounds $
0 \leq \kappa \leq 1.
$
\end{lemma}
\begin{proof}
It is not difficult to show that $\mW^{1/2} \mH_\mS \mW^{1/2}$ is the orthogonal projection matrix that projects onto $\Range{\mW^{1/2} \mS}$. Consequently, $0 \preceq \mW^{1/2} \mH_\mS \mW^{1/2} \preceq \mI$ and, after taking expectation, we get
$
0 \preceq \mW^{1/2}\ED{\mH_\mS} \mW^{1/2} \preceq \mI.
$
Finally, this implies that
\begin{equation} \label{eq:lambdamaxminkappa}
 0\leq \lambda_{\max}(\mI - \mW^{1/2}\ED{\mH_\mS} \mW^{1/2}) = 1 - \lambda_{\min}( \mW^{1/2}\ED{\mH_\mS} \mW^{1/2}) \leq 1.
\end{equation}
 \hfill \qed
\end{proof}

In our convergence theorem we will assume that $\kappa>0$. This can be achieved by choosing a suitable distribution $\cD$ and it holds trivially for all the examples we develop. The condition $\kappa>0$ essentially says that the distribution is sufficiently rich. This condition number was first proposed in \cite{Gower2015} in the context of randomized algorithms for solving linear systems. We refer the reader to that work for details on sufficient assumptions about $\cD$ guaranteeing $\kappa>0$.  Below we give an  example.

\begin{example} Let $\mW\succ 0$, and let $\cD$ be given by setting $\mS = e_i$ with probability $p_i>0$. Then 
\begin{eqnarray*} \kappa  &\overset{\eqref{eq:kappa}}{=}&  \lambda_{\min}\left( \mW^{1/2}\ED{\Proj_{\mS}}\mW^{-1/2}\right) \quad = \quad \lambda_{\min}\left(\sum_{i=1}^n \frac{p_i}{e_i^\top \mW e_i} \mW^{1/2} e_i  e_i^\top \mW^{1/2} \right).
\end{eqnarray*}
Since the vectors $\mW^{1/2}e_i$ span $\R^n$ and $p_i>0$ for all $i$, the matrix is positive definite and hence $\kappa>0$. In particular, when $\mW=\mI$, then the expected projection matrix is equal to $\diag(p_1,\dots,p_n)$ and $\kappa = \min_i p_i>0$. If instead of unit basis vectors $\{e_i\}$ we use vectors that span $\R^n$, using similar arguments we can also conclude that $\kappa>0$.
\end{example}

\subsection{Convergence theorem} \label{sec:theoremONE}

Our main convergence result, which we shall present shortly, holds for $\mu$-strongly convex functions. However, it turns out we can establish the result for a somewhat larger family of functions. This family is described next.
 
\begin{assumption}[One point strong convexity] \label{ass:str_convex_at_solution} Function $f$ for some $\mu>0$ satisfies
\begin{equation} \label{eq:strconv}
 f(x^*) \geq f(x) + \dotprod{\nabla f(x), x^*-x} + \frac{\mu}{2} \norm{x^*-x}_{2}^2, \qquad \forall x \in \R^d.
\end{equation}
\end{assumption}

We are now ready to present the main result of this section. 

\begin{theorem}[Convergence of JacSketch for General Sketches] \label{theo:convgen} Let $\mW\succ 0$. Let $f$ satisfy Assumption~\ref{ass:str_convex_at_solution}. Let Assumption~\ref{def:sketch} be satisfied (i.e, $\mS$ is an unbiased sketch and $\theta_{\mS}$ is the associated bias-correcting random variable).   Let the expected smoothness assumptions be satisfied: Assumption~\ref{ass:ES1}  and Assumption~\ref{ass:ES2}. Assume that $\kappa>0$.  Let the sketch residual be defined as  in \eqref{eq:rhointro}, i.e, 
	\begin{equation}\label{eq:theorhoXX}
	\rho = \rho(\theta_{\mS},\cD, \mW) \overset{\eqref{eq:rhointro}}{=} \lambda_{\max}\left( \mW^{1/2}\left( \ED{\theta_{\mS}^2 \Proj_{\mS} \ones \ones^\top \Proj_{\mS}} -\ones \ones^\top\right) \mW^{1/2}\right) \geq 0.
	\end{equation}

Choose any $x^0\in \R^d$ and $\mJ^0\in \R^{d\times n}$. Let $\{x^k,\mJ^k\}_{k\geq 0}$ be the random iterates produced by JacSketch (Algorithm~\ref{alg:SketchJac}). Consider the Lyapunov function
\begin{equation}\label{eq:lyapgen}
\Psi^k \eqdef \norm{x^k-x^*}_2^2 +  \frac{\alpha}{2 \cL_2} \norm{\mJ^{k} -\Jac(x^*)}_{\mW^{-1}}^2.
\end{equation}
   If the stepsize satisfies
	\begin{equation}
	0 \leq \alpha \leq \min\left\{ \frac{1}{4 \cL_1 }, \, \frac{\kappa}{4 \cL_2 \rho/n^2   +\mu  }\right\},
	\label{eq:alphaboundXX}
	\end{equation}
	then 
	\begin{equation} \label{eq:conv2XX}
	\E{\Psi^{k}} \quad \leq \quad (1-\mu \alpha)^k \cdot \Psi^0,
	\end{equation}
If we choose $\alpha$ to be equal to the upper bound in \eqref{eq:alphaboundXX}, then
	\begin{equation}\label{eq:itercomplexgen}
	k\geq \max \left\{ \frac{4\cL_1}{\mu}, \; \frac{1}{\kappa} + \frac{4\rho \cL_2}{\kappa \mu n^2}  \right\} \log\left (\frac{1}{\epsilon} \right) \quad \Rightarrow \quad \E{\Psi^{k}}  \leq \epsilon \Psi^0.\end{equation}
\end{theorem}

Recall that the  iteration complexity expression from \eqref{eq:itercomplexgen} is listed in row 1 of Table~\ref{tbl:complexity_summary}.

 The Lyapunov function we use is simply the sum of the squared distance between $x^k$ to the optimal $x^*$ and the distance of our Jacobian estimate $\mJ^k$ to the optimal Jacobian $\Jac(x^*).$ Hence, the theorem says that both the iterates $\{x^k\}$ and the Jacobian estimates $\{\mJ^k\}$ converge.

\subsection{Projection lemmas and the stochastic condition number $\kappa$}
\label{sec:projcond}

In this section we collect some basic results on  projections. Recall from \eqref{eq:PSdef} that $\Proj_\mS = \mS (\mS^\top \mW \mS)^{\dagger} \mS^\top  \mW$ and from \eqref{eq:H_S} that $\mH_\mS = \mS (\mS^\top \mW \mS)^{\dagger} \mS^\top $.
\begin{lemma} \label{lem:proj}
	\begin{equation}\label{eq:proj}
	\Proj_\mS \mW^{-1}(\mI-\Proj_\mS)^\top = 0.
	\end{equation}
Furthermore,
\begin{equation}\label{eq:projext}\ED{\Proj_\mS \mW^{-1}\Proj_\mS^\top} = \ED{\mH_\mS} \quad \text{and}\quad \ED{(\mI-\Proj_\mS) \mW^{-1} (\mI-\Proj_\mS)^\top} = \mW^{-1} - \ED{\mH_\mS}.\end{equation}
\end{lemma}
\begin{proof}
Using the pseudoinverse property $\mA^\dagger \mA \mA^\dagger = \mA^\dagger$ we have that
\begin{equation} \label{eq:temp8h8e8}\Proj_\mS \mW^{-1}\Proj_\mS^\top \overset{\eqref{eq:PSdef}}{=} \mS (\mS^\top \mW \mS)^{\dagger} \mS^\top  \mW \mS (\mS^\top \mW  \mS)^{\dagger} \mS^\top   \overset{\eqref{eq:H_S}}{=} \Proj_\mS \mW^{-1}  = \mH_\mS, \end{equation}
and as a consequence~\eqref{eq:proj} holds.
Moreover,
\begin{equation} \label{eq:temp8h8e82}(\mI-\Proj_\mS) \mW^{-1} (\mI-\Proj_\mS)^\top \overset{\eqref{eq:proj}}{=}  \mW^{-1} (\mI-\Proj_\mS)^\top  \overset{\eqref{eq:H_S}}{=} \mW^{-1} - \mH_\mS.\end{equation}
Finally, taking expectation over~\eqref{eq:temp8h8e8} and~\eqref{eq:temp8h8e82} gives~\eqref{eq:projext}. \hfill \qed 
\end{proof}

\begin{lemma} \label{lem:bo98gd}
	For any matrices $\mM, \mN \in \R^{d\times n}$ we have the identities
	\[\norm{\mM ( \mI-\Proj_\mS) + \mN \Proj_\mS}_{\mW^{-1}}^2 = \norm{\mM (\mI-\Proj_\mS)}_{\mW^{-1}}^2 +  \norm{\mN \Proj_\mS}_{\mW^{-1}}^2\]
	and
		\begin{equation}\label{eq:ih0shg09s}\ED{\norm{ \mN \Proj_\mS}_{\mW^{-1}}^2} =  \norm{\mN}_{\ED{\mH_\mS}}^2.\end{equation}
	Furthermore,
\begin{equation} \label{eq:orthoexpka} \ED{\norm{\mM ( \mI-\Proj_\mS) + \mN \Proj_\mS}_{\mW^{-1}}^2} \leq  (1-\kappa) \| \mM\|_{\mW^{-1}}^2 + \| \mN\|_{\ED{\mH_\mS}}^2. \end{equation}
\end{lemma}
\begin{proof}
First, note that
\begin{eqnarray*}\norm{\mM (\mI-\Proj_\mS) + \mN \Proj_\mS}_{\mW^{-1}}^2&=& \norm{\mM (\mI-\Proj_\mS)}_{\mW^{-1}}^2 +  \norm{\mN \Proj_\mS}_{\mW^{-1}}^2  + \Tr{ \mM^\top \mN\Proj_\mS \mW^{-1} (\mI-\Proj_\mS)^\top} \\
&\overset{\eqref{eq:proj}}{ =}& \norm{\mM (\mI-\Proj_\mS)}_{\mW^{-1}}^2 +  \norm{\mN \Proj_\mS}_{\mW^{-1}}^2 .
\end{eqnarray*}

By taking expectations in $\cD$, we get
\begin{eqnarray*}\ED{\norm{\mM (\mI-\Proj_\mS) + \mN \Proj_\mS}_{\mW^{-1}}^2} &=& \ED{\norm{\mM (\mI-\Proj_\mS)}_{\mW^{-1}}^2} +  \ED{\norm{\mN \Proj_\mS}_{\mW^{-1}}^2}\\	
&\overset{\eqref{eq:projext}}{ =}& \norm{\mM}^2_{\mW^{-1}-\ED{\mH_\mS}} + \norm{\mN}^2_{\ED{\mH_\mS}}\\
&\leq & (1-\kappa) \|\mM\|_{\mW^{-1}}^2 + \|\mN\|_{\ED{\mH_\mS}}^2,
\end{eqnarray*}
where in the last step we used the estimate
\begin{eqnarray*}
 \mW^{-1} - \ED{\mH_\mS} &=& \mW^{-1/2} (\mI - \mW^{1/2}\ED{\mH_\mS}\mW^{1/2}) \mW^{-1/2}\\ &\preceq&  \lambda_{\max} (\mI - \mW^{1/2}\ED{\mH_\mS} \mW^{1/2}) \mW^{-1} \quad  \overset{\eqref{eq:lambdamaxminkappa}}{=} \quad (1-\kappa) \, \mW^{-1}.\end{eqnarray*}
 \hfill \qed

\end{proof}

\subsection{Key lemmas} \label{sec:2KEYlemmas}

We first establish two lemmas. The first lemma provides an upper bound on the quality of new Jacobian estimate in terms of the quality of the current estimate and function suboptimality.  If the second term on the right hand side was not there, the lemma would be postulating a contraction on the quality of the Jacobian estimate. 

\begin{lemma} \label{lem:jac-contractXX}
Let Assumption~\ref{ass:ES2}  be satisfied. Then iterates of Algorithm~\ref{alg:SketchJac} satisfy
\begin{equation} 
 \ED{\norm{\mJ^{k+1} -\Jac(x^*)}_{\mW^{-1}}^2} \leq 
 (1-\kappa) \norm{\mJ^{k}-\Jac(x^*)}_{\mW^{-1}}^2 +  2\cL_2(f(x^k) -f(x^*)),
 \end{equation}
 where $\kappa$ is defined in \eqref{eq:kappa}.
\end{lemma}

\begin{proof}
Subtracting $\Jac(x^*)$ from both sides of~\eqref{eq:jacobsol} gives
\begin{eqnarray} \label{eq:jacDFstarstepWXX}
\mJ^{k+1} -\Jac(x^*) & \overset{\eqref{eq:jacobsol} }{=} & 
 \underbrace{(\mJ^{k}-\Jac(x^*))}_{\mM}(\mI- \Proj_{\mS_k}) + \underbrace{ (\Jac(x^k) -\Jac(x^*))}_{\mN}\Proj_{\mS_k}. 
\end{eqnarray}
Taking	norms on both sides,    then expectation with respect to $\mS_k$ and then using Lemma~\ref{lem:bo98gd}, we get
\begin{eqnarray*}
\ED{\norm{\mJ^{k+1} -\Jac(x^*)}_{\mW^{-1}}^2} & \overset{\eqref{eq:orthoexpka}}{\leq} & 
(1-\kappa) \norm{\mM}_{\mW^{-1}}^2 + \norm{\mN}_{\ED{\mH_{\mS_k}}}^2 \nonumber \\
  &  \overset{\eqref{eq:ES2}}{\leq } & (1-\kappa) \norm{\mJ^{k}-\Jac(x^*)}_{\mW^{-1}}^2 + 2\cL_2 (f(x^k) -f(x^*)).
\end{eqnarray*}
\hfill \qed
\end{proof}

We now bound the second moment of $g^k$. The lemma implies that as $x^k$ approaches $x^*$ and $\mJ^k$ approaches $\Jac(x^*)$, the variance of $g^k$ approaches zero. This is a key property of JacSketch which elevates it into the ranks of variance-reduced methods.

\begin{lemma}\label{lem:gradient_bounddeltaXX}  Let $\mS$ be an unbiased sketch. Let Assumption~\ref{ass:ES1} be satisfied (i.e., assume that inequality \eqref{eq:ES1} holds for some $\cL_1>0$). Then the second moment of the estimated gradient is bounded by
\begin{equation}\label{eq:gradbndsubdeltaXX}
\ED{\norm{g^k}_2^2 } \leq 4 \cL_1  (f(x^k) -f(x^*))
+  2 \frac{\rho}{n^2} \norm{\mJ^{k} -\Jac(x^*)}_{\mW^{-1}}^2,\end{equation}
where   $\rho$ is defined in \eqref{eq:theorhoXX}. 
\end{lemma}

\begin{proof}
Adding and subtracting $\tfrac{\theta_{\mS_k}}{n} \Jac(x^*) \Proj_{\mS_k} \ones$ in~
\eqref{eq:g^k=intro}
 gives
\[g^k = \underbrace{\frac{1}{n}\mJ^{k}\ones  - \frac{\theta_{\mS_k}}{n} (\mJ^{k} -\Jac(x^*)) \Proj_{\mS_k}  \ones}_{b} + \underbrace{\frac{\theta_{\mS_k}}{n} (\Jac(x^k)-\Jac(x^*))\Proj_{\mS_k} \ones}_{a}. \] 
Taking norms  on both sides and using the bound $\|a+b\|_2^2\leq 2\|a\|_2^2 + 2\|b\|_2^2$ gives
\begin{eqnarray}
\norm{g^{k}}_2^2 & \leq  & 
\underbrace{\frac{2}{n^2} \norm{(\Jac(x^k)-\Jac(x^*))\Proj_{\mS_k} \theta_{\mS_k} \ones }_2^2}_{a^k} +
\underbrace{\frac{2}{n^2} \norm{ \theta_{\mS_k} (\mJ^{k} -\Jac(x^*)) \Proj_{\mS_k} \ones  -\mJ^{k} \ones}_2^2}_{b^k}. \label{eq:gknorm2bndXX}
\end{eqnarray}

In view of Assumption~\ref{ass:ES1} 
(combine \eqref{eq:ES1} and \eqref{eq:8gd088gs899}),  we have \begin{equation} \label{eq:buv9fg9d98}\ED{a^k} \leq 4 \cL_1  (f(x^k)-f(x^*)),\end{equation} where the expectation is taken with respect to $\mS_k$. Let us now bound $\ED{b^k}$. Using the fact that $\Jac(x^*)\ones = 0$, we can write
{\footnotesize
\begin{eqnarray}
\ED{b^k}
& = & \frac{2}{n^2} \ED{\norm{(\mJ^{k} -\Jac(x^*)) \theta_{\mS_k} \Proj_{\mS_k} \ones -(\mJ^{k}-\Jac(x^*)) \ones}_2^2}\nonumber\\ 
& = & \frac{2}{n^2} \ED{\norm{(\mJ^{k} -\Jac(x^*)) (\theta_{\mS_k} \Proj_{\mS_k} - \mI)\ones}_2^2} \nonumber \\
&=&  \frac{2}{n^2}\ED{ \ones^\top (\theta_{\mS_k} \Proj_{\mS_k} - \mI)^\top  (\mJ^{k} -\Jac(x^*))^\top (\mJ^{k} -\Jac(x^*)) (\theta_{\mS_k} \Proj_{\mS_k} - \mI)\ones  } \nonumber \\
&=& \frac{2}{n^2} \ED{\Tr{\ones^\top (\theta_{\mS_k} \Proj_{\mS_k} - \mI)^\top  (\mJ^{k} -\Jac(x^*))^\top (\mJ^{k} -\Jac(x^*)) (\theta_{\mS_k} \Proj_{\mS_k} - \mI)\ones  }} \nonumber \\
&=&  \frac{2}{n^2}\ED{\Tr{\ones^\top (\theta_{\mS_k} \Proj_{\mS_k} - \mI)^\top \mW^{1/2} \mW^{-1/2} (\mJ^{k} -\Jac(x^*))^\top (\mJ^{k} -\Jac(x^*)) \mW^{-1/2} \mW^{1/2} (\theta_{\mS_k} \Proj_{\mS_k} - \mI)\ones  }} \nonumber \\
&=& \frac{2}{n^2} \ED{\Tr{ \mW^{-1/2} (\mJ^{k} -\Jac(x^*))^\top (\mJ^{k} -\Jac(x^*)) \mW^{-1/2} \mW^{1/2} (\theta_{\mS_k} \Proj_{\mS_k} - \mI)\ones  \ones^\top (\theta_{\mS_k} \Proj_{\mS_k} - \mI)^\top \mW^{1/2}}} \nonumber \\
&=&\frac{2}{n^2} \Tr{ \mW^{-1/2} (\mJ^{k} -\Jac(x^*))^\top (\mJ^{k} -\Jac(x^*)) \mW^{-1/2} \ED{\mW^{1/2} (\theta_{\mS_k} \Proj_{\mS_k} - \mI)\ones  \ones^\top (\theta_{\mS_k} \Proj_{\mS_k} - \mI)^\top \mW^{1/2}}} \nonumber.
\end{eqnarray}
}
If we now let $v=\mW^{1/2} (\theta_{\mS_k} \Proj_{\mS_k} - \mI)\ones$ and $\mM=(\mJ^{k} -\Jac(x^*)) \mW^{-1/2}$, then we can continue:
\begin{eqnarray}
\ED{b^k}
&=& \frac{2}{n^2}\Tr{\mM^\top \mM \ED{v v^\top}} \quad \leq \quad  \frac{2}{n^2}\Tr{\mM^\top \mM}  \lambda_{\max}\left( \ED{vv^\top} \right) \nonumber \\
 & \overset{\eqref{eq:fro_norm}}{=} & \frac{2}{n^2} \norm{\mJ^{k} -\Jac(x^*)}_{\mW^{-1}}^2  \lambda_{\max}\left( \ED{vv^\top} \right)
\quad =\quad   \frac{2 \rho}{n^2} \norm{\mJ^{k} -\Jac(x^*)}_{\mW^{-1}}^2, 
 \label{eq:andc9oa9nqYY} 
\end{eqnarray}
where in the last step we have used the assumption that $\theta_{\mS_k}$ is bias-correcting: 
\begin{equation}
\lambda_{\max}\left( \ED{vv^\top} \right)   \overset{\eqref{eq:unbiased}}{=}  \lambda_{\max} \left( \mW^{1/2}  \ED{\theta_{\mS_k}^2\Proj_{\mS_k} \ones \ones^\top \Proj_{\mS_k}^\top} \mW^{1/2} - \mW^{1/2} \ones \ones^\top \mW^{1/2} \right) \overset{\eqref{eq:theorhoXX}}{=} \rho.
\end{equation}
 It now only remains to substitute \eqref{eq:buv9fg9d98} 
and \eqref{eq:andc9oa9nqYY} into \eqref{eq:gknorm2bndXX} to arrive at~\eqref{eq:gradbndsubdeltaXX}. 
\hfill \qed
\end{proof}

\subsection{Proof of Theorem~\ref{theo:convgen}}

With the help of the above lemmas, we now proceed to the proof of the theorem.  In view of the strong convexity assumption \eqref{eq:strconv}, we have
\begin{eqnarray}
\dotprod{\nabla f(y), y- x^*}  &\geq & f(y) -f(x^*) +\frac{\mu}{2} \norm{y - x^*}_2^2. \label{eq:strconv3}\end{eqnarray}

By using the relationship $x^{k+1} = x^k -\alpha g^k$, the fact that $g^k$ is an unbiased estimate of the gradient $\nabla f(x^k)$, and using one-point strong convexity \eqref{eq:strconv3}, we get
\begin{eqnarray}
\ED{\norm{x^{k+1} -x^*}_2^2 } &\overset{ \eqref{eq:xupdate}}{=} & \ED{\norm{x^k  -x^* - \alpha g^{k}}_2^2} \nonumber \\
& \overset{\eqref{eq:unbiasedgrad}}{=} &  \norm{x^k  -x^*}_2^2 -2\alpha \dotprod{\nabla f(x^k), x^k  -x^*}  + \alpha^2\ED{\norm{g^{k}}_2^2}\nonumber \\ 
&\overset{\eqref{eq:strconv3}}{\leq }  &    (1-\alpha\mu)\norm{x^k  -x^*}_2^2 +\alpha^2\ED{\norm{g^{k}}_2^2}  -2\alpha (f(x^k)-f(x^*)).
 \label{eq:convstepsub1XX}
\end{eqnarray}
Next, applying Lemma~\ref{lem:gradient_bounddeltaXX} leads to the estimate
\begin{eqnarray}
\ED{\norm{x^{k+1} -x^*}_2^2 } & \overset{\eqref{eq:gradbndsubdeltaXX}}{ \leq}  &   (1-\alpha\mu)\norm{x^k  -x^*}_2^2  +2\alpha\left( 2\alpha \cL_1  -1\right) (f(x^k)-f(x^*))
 \nonumber \\
&&  +  2\alpha^2 \frac{\rho}{n^2}\norm{\mJ^{k} -\Jac(x^*)}_{\mW^{-1}}^2.
 \label{eq:convstepsub1XX}
\end{eqnarray}
Let $\sigma = 1 / (2 \cL_2)$.
Adding $\sigma\alpha \ED{\norm{\mJ^{k+1} -\Jac(x^*) }_{\mW^{-1}}^2}$ to both sides of the above inequality and substituting in the definition of $\Psi^k$ from~\eqref{eq:lyapgen}, it follows that  
\begin{eqnarray}
\ED{\Psi^{k+1}} & \overset{\eqref{eq:convstepsub1XX}}{\leq} & (1-\alpha\mu)\norm{x^k  -x^*}_2^2 +2\alpha\left( 2\alpha \cL_1  -1\right) (f(x^k)-f(x^*))
 \nonumber \\
&& +  2\alpha^2 \frac{\rho}{n^2} \norm{\mJ^{k} -\Jac(x^*)}_{\mW^{-1}}^2 +  \sigma\alpha \ED{\norm{ \mJ^{k+1} -\Jac(x^*) }_{\mW^{-1}}^2} \nonumber\\
& \overset{\text{(Lemma~\ref{lem:jac-contractXX})}}{\leq} & (1-\alpha\mu)\norm{x^k  -x^*}_2^2 +2\alpha\underbrace{\left(  \cL_2 \sigma +  2\alpha \cL_1 -1\right)}_{\text{I}} (f(x^k)-f(x^*)) \nonumber \\
&& +\sigma\alpha\underbrace{\left( 1-\kappa +2  \frac{\alpha \rho}{\sigma n^2}\right)}_{\text{II}}\norm{\mJ^{k} -\Jac(x^*)}_{\mW^{-1}}^2 .\label{eq:nqqecpr92XX}
\end{eqnarray}

We now choose $\alpha$ so that  $\text{I} \leq 0$ and  $\text{II} \leq 1-\alpha \mu$, which can be written as
\begin{eqnarray}
\alpha \quad \leq \quad   \frac{1-\cL_2 \sigma}{2\cL_1} \quad \text{and} \quad 
\alpha \quad  \leq \quad\frac{\kappa}{2\rho/(\sigma n^2) + \mu}. \label{eq:a98sjd9a8js}
\end{eqnarray}
If $\alpha$ satisfies the above two inequalities, then \eqref{eq:nqqecpr92XX} takes on the simplified form
$
\ED{\Psi^{k+1}} \leq  (1-\alpha \mu) \Psi^k.
$
By taking expectation again and using the tower rule, we get $\E{\Psi^{k}} \leq  (1-\alpha \mu)^k  \Psi^0$. Note that as long as $k\geq \frac{1}{\alpha \mu} \log \frac{1}{\epsilon}$, we have $\E{\Psi^k} \leq \epsilon \Psi^0$. Recalling that $\sigma =1/(2 \cL_2)$, and choosing $\alpha$ to be the minimum of the two upper bounds \eqref{eq:a98sjd9a8js} gives the upper bound on~\eqref{eq:alphaboundXX}, which in turn leads to \eqref{eq:itercomplexgen}.
\hfill \qed

\section{Minibatch Sketches} \label{sec:minibatch_sketches}

In this section we focus on special cases of Algorithm~\ref{alg:SketchJac} where one  computes $\nabla f_i(x^k)$ for $i\in S^k$, where  $S^k$ is a random subset (mini-batch) of $[n]$ chosen in each iteration according to some fixed probability law. As we have seen in the introduction, this is achieved by choosing $\mS_k = \mI_{S_k}$.

We say that $\mS$ is a {\em minibatch sketch} if $\mS = \mI_{S}$ for some random set (sampling) $S$, where $\mI_S \in \R^{n\times |S|}$ is a column submatrix of the $n\times n$ identity matrix $\mI$ associated with columns indexed by the set $S$. That is, the distribution $\cD$ from which the sketches $\mS$ are sampled is defined by \[\Prb{\mS = \mI_C} = p_{C} , \qquad C\subseteq [n],\]
where $\sum_{C \subseteq [n]} p_C = 1$ and $p_C\geq 0$ for all $C$.

\subsection{Samplings} \label{sec:samplings}

We now formalize the notion of a random set, which we will refer to by the name sampling. A {\em sampling} is a random set-valued mapping with values being the subsets of $[n]$. A sampling $S$ is uniquely characterized by the probabilities $p_C\eqdef \Prb{S = C}$ associated with every subset $C$  of $ [n]$.

\begin{definition}[Types of samplings]  We say that sampling $S$ is non-vacuous if $\Prb{S=\emptyset} = 0$ (i.e., $p_{\emptyset} = 0$).  Let $p_i\eqdef \Prb{i \in S} = \sum_{C : i\in C} p_C$. We say that $S$ is proper if $p_i>0$ for all $i$. We say that $S$ is uniform if $p_i=p_j$ for all $i,j$. We say that $S$ is $\tau$--uniform if it is uniform and  $|S|=\tau$ with probability 1. In particular, the unique sampling which assigns equal probabilities to all subsets of $[n]$ of cardinality $\tau$ and zero probabilities to all other subsets is called the $\tau$--nice sampling. 
\end{definition}

We refer the reader to \cite{PCDM, ESO} for a background reading on samplings and their properties.

\begin{definition}[Support] \label{def:unif_support} The support of a sampling $S$ is the set
of subsets of $[n]$ which are chosen by $S$ with positive probability: $\support(S) \eqdef \{C \;:\; p_C > 0\}$. We say that $S$ has uniform support if
\[c_1 \eqdef | \{ C \in \support(S) \;:\; i\in C  \} | = | \{ C \in \support(S) \;:\; j\in C  \} | \] 
for all $i,j\in [n]$. In such a case we say that the support is $c_1$--uniform.
\end{definition}

To illustrate the above concepts, we now list a few examples with  $n=4$. 
\begin{example} The sampling defined by setting $p_{\{1,2\}} = p_{\{3,4\}} = 0.5$ is non-vacuous, proper, $2$--uniform ($p_i=0.5$ for all $i$ and $|S|=2$ with probability 1),  and has $1$--uniform support. If we change the probabilities to $p_{\{1,2\}}=0.4$ and $p_{\{3,4\}}=0.6$, the sampling is no longer uniform (since $p_1=0.4\neq 0.6=p_3$), but it still has $1$--uniform support, is proper and non-vacuous. Hence, a sampling with uniform support need not be uniform. On the other hand, a uniform sampling need not have uniform support. As an example, consider sampling $S$ defined via $p_{\{1\}} = 0.4$, $p_{\{2,3\}} = p_{\{3,4\}} = p_{\{2,4\}} = 0.2$. It is uniform (since $p_i=0.4$ for all $i$). However, while element $1$ appears in a single  set of its support, elements $2,3$ and $4$ each appear in two sets. So, this sampling does not have uniform support.
\end{example}

\begin{example} A uniform sampling need not be $\tau$--uniform for any $\tau$. For example, the sampling defined by setting $p_{\{1,2,3,4\}} = 0.5$, $p_{\{1,2\}} = 0.25$ and $p_{\{3,4\}} = 0.25$ is uniform (since $p_i=0.75$ for all $i$), but as it assigns positive probabilities to sets of at least two different cardinalities, it is not $\tau$--uniform for any~$\tau$.  
\end{example}

\begin{example} Further, the sampling defined by setting $p_{\{1,2\}} = 1/6$, $p_{\{1,3\}} = 1/6$, $p_{\{1,4\}}=1/6$, $p_{\{2,3\}}=1/6$, $p_{\{2,4\}}=1/6$, $p_{\{3,4\}}=1/6$ is non-vacuous, $2$--uniform ($p_i=1/2$ for all $i$ and $|S|=2$ with probability 1), and has $3$--uniform support.  The sampling defined by setting $p_{\{1,2\}} = 1/3$, $p_{\{2,3\}} = 1/3$, $p_{\{3,1\}} = 1/3$ is non-vacuous, proper, $2$--uniform ($p_i=2/3$ for all $i$ and $|S|=2$ with probability 1) and has $2$--uniform support. 
\end{example}

Note that a sampling with uniform support is  necessarily proper as long as $c_1>0$. However, it need not be non-vacuous. For instance, the sampling $S$ defined by setting $p_{\emptyset}=1$ has $0$--uniform support and is vacuous.  From now on, we only consider samplings with the following properties.
 
\begin{assumption} \label{ass:proper} $S$ is non-vacuous and has $c_1$--uniform support with $c_1\geq 1$.
\end{assumption}

Note that if $S$ is a non-vacuous sampling with $1$--uniform support, then its support is necessary a partition of $[n]$. We shall pay specific attention to such samplings in Section~\ref{sec:theopart} as for them we can develop a stronger analysis than that provided by Theorem~\ref{theo:convgen}.

 \subsection{Minibatch sketches and projections}

In the next result we describe some basic properties of the projection matrix $\Proj_{\mS} = \mS(\mS^{\top} \mW \mS)^\dagger \mS^\top \mW$  associated with a minibatch sketch $\mS$. 
 
\begin{lemma} \label{prop:bus98g90s09} Let $\mW = \diag(w_1,\dots,w_n)$. Let $S$ be any sampling, $\mS=\mI_S$ be the associated minibatch sketch,  and let $\mP$ be the probability matrix\footnote{The notion of a probability matrix associated with a  sampling was first introduced in \cite{PCDM} in the context of parallel coordinate descent methods, and further studied in \cite{ESO}.} associated with sampling $S$: $\mP_{ij} = \Prb{i \in S \; \&\;  j\in S}$. Then
\begin{enumerate}
\item[(i)]  $ \Proj_{\mS} = \mI_S \mI_S^\top $. This is a diagonal matrix with the $i$th diagonal element equal to 1 if $i\in S$, and $0$ if $i\notin S$.
\item[(ii)] 
$\Proj_{\mS} \ones = e_S  \eqdef \sum_{i\in S} e_i .$
\item[(iii)] $\ED{\Proj_{\mS} \ones \ones^\top \Proj_{\mS}} = \sum_{C \subseteq [n]} p_C e_C e_C^\top = \mP$
\item[(iv)] $\ED{\Proj_{\mS}} = \diag(\mP)$
\item[(v)] The stochastic condition number defined in \eqref{eq:kappa} is given by $\kappa = \min_{i} p_i$
\item[(vi)] Let $S$ satisfy Assumption~\ref{ass:proper}.  Then the random variable \begin{equation}
\label{eq:bu80d09hbjdd}\theta_{\mS} \eqdef \frac{1}{c_1 p_S },\end{equation} defined on $\support(S)$, is bias-correcting.That is,
$\ED{\Proj_{\mS} \theta_{\mS} \ones} = \ones.$
\end{enumerate}
\end{lemma}
\begin{proof}
\begin{itemize}
\item[(i)] This follows by noting that $\mI_S^\top \mW \mI_S$ is the $|S|\times |S|$ diagonal matrix with diagonal entries corresponding to $w_i$ for $i\in S$, which in turn can be used to show that $(\mI_S^\top \mW \mI_S)^{-1} \mI_S^\top \mW = \mI_S^\top $.
\item[(ii)] This follows from (i) by noting that $\mI_S^\top \ones$ is the vector of all ones in $\R^{|S|}$.
\item[(iii)] Using (ii), we have $\Proj_{\mS} \ones \ones^\top \Proj_{\mS} = e_S e_S^\top$. By linearity of expectation, $\left(\ED{e_S e_S^\top}\right)_{ij} =\ED{(e_S e_S^\top)_{ij}}  = \ED{ 1_{i,j \in S}} = \Prb{i\in S \;\&\; j\in S} = \mP_{ij}$, where $1_{i,j\in S} = 1$ if $i,j\in S$ and $1_{i,j\in S} = 0$ otherwise.
\item[(iv)]  This follows  from (i) by taking expectations of the diagonal elements of $\Proj_{\mS}$.
\item[(v)] Follows from (iv).
\item[(vi)] Indeed, 
\begin{equation} \label{eq:unbiasedpart}
\ED{\theta_{\mS} \Proj_\mS \ones} \overset{\text{(ii)}}{=} \sum_{C \in \support(S)} p_C\theta_C e_C    
 \overset{\eqref{eq:bu80d09hbjdd}}{=} \frac{1}{c_1}\sum_{C \in \support(S)}  e_C  = \ones,
\end{equation}
where the last equation follows from the assumption that the support of $S$ is $c_1$--uniform. \hfill \qed
\end{itemize}
\end{proof}

The following simple observation will be useful in the computation  of the constant $\cL_1$. The proof is straightforward and involves a double counting argument.

\begin{lemma}\label{lem:hi8s8bvf786sfvs} Let $S$ be a sampling satisfying Assumption~\ref{ass:proper}. Moreover, assume that $S$ is $\tau$--uniform. Then 
$\frac{|\support(S)|}{c_1} = \frac{n}{\tau}$.
Consequently, 
$ \kappa = p_1 = p_2 = \dots = p_n = \frac{\tau}{n} = \frac{c_1}{|\support(S)|}$,
where $\kappa$ is the stochastic condition number associated with  the minibatch sketch $\mS=\mI_S$.
\end{lemma}

\subsection{JacSketch for minibatch sampling = minibatch SAGA}

As we have mentioned in Section~\ref{sec:SAGA-intro} already, JacSketch admits a particularly simple form for minibatch sketches, and corresponds to known and new variants of SAGA. Assume that  $S$ satisfies Assumption~\ref{ass:proper} and let  $\mW=\diag(w_1,\dots,w_n)$. In view of Lemma~\ref{prop:bus98g90s09}(vi), this means that the random variable $\theta_{\mS} = \frac{1}{c_1 p_S }$ is bias-correcting, and due to  Lemma~\ref{prop:bus98g90s09}(ii), we have $\Proj_{\mS_k} \ones = e_{S_k}= \sum_{i\in S_k} e_i$. Therefore, 
\begin{equation} \label{eq:g^k-SAGA_XXX} g^{k} \overset{\eqref{eq:g^k=intro}}{=} \frac{1}{n} \mJ^k \ones + \frac{\theta_{\mS_k}}{n} \sum_{i\in S_k} (\nabla f_i(x^k)- \mJ^k_{:i}) = \frac{1}{n}\left( \sum_{i\notin S_k} \mJ^k_{:i}  + \sum_{i\in S_k} \left(1-\tfrac{1}{c_1 p_{S_k}} \right) \mJ^k_{:i} + \tfrac{1}{c_1 p_{S_k}} \nabla f_i(x^k)  \right). \end{equation}
By  Lemma~\ref{prop:bus98g90s09}(i), $\Proj_{\mS_k} = \mI_{S_k} \mI_{S_k}^\top$. In view of \eqref{eq:jacobsolWintro}, the Jacobian estimate gets updated as follows
\begin{equation} \label{eq:SAGA-Jacobian_update} \mJ^{k+1}_{:i} = \begin{cases} \mJ^k_{:i} & \quad i\notin S_k,\\
 \nabla f_i(x^k)&  \quad i\in S_k. \end{cases}
\end{equation}

The resulting minibatch SAGA method is formalized as Algorithm~\ref{alg:SketchJac_subsample}.

\begin{algorithm}
	\begin{algorithmic}[1]
		\State \textbf{Parameters:} Sampling $S$ satisfying Assumption~\ref{ass:proper},  $\mW=\diag(w_1,\dots,w_n)$, stepsize $\alpha>0$
		\State \textbf{Initialization:} Choose  $x^0\in \R^d$, $\mJ^0 \in \R^{d \times n}$ 
		\Comment Initialization
		\For {$k =  0, 1, 2, \dots$}
		
		\State Sample a fresh set $S_k\sim S$
			
		\State $g^k = \frac{1}{n} \mJ^k \ones + \frac{1}{ n c_1 p_{S_k}} \sum_{i\in S_k} (\nabla f_i(x^k)- \mJ^k_{:i})$       \label{ln:gradupdate} \Comment Update gradient estimate  	

		\State $ \mJ^{k+1}_{:i} = \begin{cases} \mJ^k_{:i} & \quad i\notin S_k\\
		\nabla f_i(x^k) & \quad i\in S_k.\end{cases}$
		 \label{ln:jacupdate}				\Comment Update Jacobian estimate	
		
		\State $x^{k+1} = x^k - \alpha g^{k}$ \label{ln:xupdate}		
		\Comment Take a step		 
		
		\EndFor
	\end{algorithmic}
	\caption{JacSketch: Mini-batch  SAGA}
	\label{alg:SketchJac_subsample}
\end{algorithm}

Below we specialize the formula for $g^k$ to a few interesting special cases.

\begin{example}[Standard SAGA]
Standard uniform SAGA is obtained by setting $S_k = \{i\}$ with probability $1/n$ for each $i\in [n]$. Since the support of this sampling is $1$--uniform, we set  $c_1=1$. This leads to the gradient estimate
\begin{equation}\label{eq:SAGAstandard} g^k = \frac{1}{n} \mJ^k \ones +  \nabla f_i(x^k)- \mJ^k_{:i}.\end{equation}
\end{example}

\begin{example}[Non-uniform SAGA]
However, we can use non-uniform probabilities instead. Let $S_k = \{i\}$ with probability $p_i>0$ for each $i\in [n]$. Since the support of this sampling is 1--uniform, we have $c_1=1$. So, the gradient estimate has the form 
\begin{equation} \label{eq:nonuniSAGAup}g^k = \frac{1}{n} \mJ^k \ones + \frac{1}{np_i} (\nabla f_i(x^k)- \mJ^k_{:i}).\end{equation}
\end{example}

\begin{example}[Uniform minibatch SAGA, version 1]
Let $C_1,\dots,C_q$ be nonempty subsets of forming a  partition $[n]$.  Let $S_k = C_j$ with probability $p_{C_j}>0$. The support of this sampling is $1$--uniform,  and hence we can choose $c_1=1$. This leads to the gradient estimate 
\[g^k = \frac{1}{n} \mJ^k \ones + \frac{1}{np_{C_j}} \sum_{i\in C_j}(\nabla f_i(x^k)- \mJ^k_{:i}).\]
\end{example}

\begin{example}[Uniform minibatch SAGA, version 2] \label{ex:8hsh}Let $S_k$ be chosen uniformly at random from all subsets of $[n]$ of cardinality $\tau \geq 2$. That is, $\mS_k$ is the $\tau$-nice sampling, and the probabilities are equal to $p_{S_k} = 1/{n \choose \tau}$. This sampling has $c_1$--uniform support with $c_1 = {n-1 \choose \tau-1} = \frac{\tau}{n} {n \choose \tau}$. Thus, $n c_1 p_{S_k}= \tau$, and we have
\begin{equation}\label{eq:SAGAmini-nice}
g^k = \frac{1}{n} \mJ^k \ones + \frac{1}{\tau} \sum_{i\in S_k}(\nabla f_i(x^k)- \mJ^k_{:i}).\end{equation}
\end{example}

\begin{example}[Gradient descent] Consider the same situation as in Example~\ref{ex:8hsh}, but with $\tau=n$. That is, we choose $S_k = [n]$ with probability $1$, and $c_1=1$. Then
\[g^k = \frac{1}{n} \mJ^k \ones + \frac{1}{n} \sum_{i=1}^n(\nabla f_i(x^k)- \mJ^k_{:i}) = \nabla f(x^k).\]
\end{example}


\subsection{Expected smoothness constants $\cL_1$ and $\cL_2$ } \label{sec:L1_L_2-xxx}

Here we compute the expected smoothness constants $\cL_1$ and $\cL_2$ in the case of $\mS$ being a minibatch sketch $\mS=\mI_S$, and  assuming that $f$ is convex and smooth. We first formalize the notion of smoothness we will use.

\begin{assumption}\label{ass:L_C-smoothness}  For $\emptyset \neq C\subseteq [n]$ define \begin{equation}\label{eq:f_C-def}f_C(x) \eqdef \frac{1}{|C|}\sum_{i\in C} f_i(x).\end{equation} For each $\emptyset\neq C\subseteq[n]$ and all $x\in \R^d$, the function $f_C$ is $L_C$--smooth and convex. That is,  there exists $L_C\geq 0$ such that the following inequality holds 
\begin{equation} \label{eq:bdf7gd899u}\|\nabla f_C(x) - \nabla f_C(x^*)\|_2^2 \leq 2L_C \left(f_C(x) - f_C(x^*) - \langle \nabla f_C(x^*), x-x^*\rangle \right), \qquad \forall x\in \R^d.\end{equation}
Let $L_i = L_{\{i\}}$ for $i\in [n]$.
\end{assumption}

The above assumption is somewhat non-standard. Note that, however, if we instead assume that each $f_i$ is convex and $L_i$-smooth, then the above assumption holds for $L_C = \frac{1}{|C|}\sum_{i\in C} L_i$. In some cases, however, we may have better estimates of the constants $L_C$ than those provided by the averages of the $L_i$ values. The value of these constants will have a direct influence on $\cL_1$ and $\cL_2$, which is why we work with this more refined assumption instead.

\begin{lemma}[Smoothness of the Jacobian]\label{lem:jacsmoothness} Assume that $f_i$ is convex and $L_i$--smooth for all $i\in [n]$. Define 
$L_{\max} \eqdef \max_{i} L_i$ and $\mD_L \eqdef \diag(L_1, \ldots, L_n) \in \R^{n\times n}.$ Then
\begin{equation}\label{eq:DFDFstL-1}
	\norm{\Jac(x) - \Jac(x^*)}_{\mD_L^{-1}}^2  \leq 2n(f(x) - f(x^*)), \qquad \forall x\in \R^d.
\end{equation}
\end{lemma}
\begin{proof}
Indeed,
\begin{eqnarray*}
\norm{\Jac(x) - \Jac(x^*)}_{\mD_L^{-1}}^2  & \overset{\eqref{eq:fro_norm}}{=}& \norm{(\Jac(x) - \Jac(x^*)) \mD^{-1/2}_L}^2 \quad  \overset{\eqref{eq:fro_norm}}{=} \quad  \sum_{i=1}^n  \frac{1}{L_{i}}\norm{ \nabla f_i(x) - \nabla f_i(x^*)}_2^2 \nonumber \\
& \leq &
2\sum_{i=1}^n    (f_i(x) - f_i(x^*) -\dotprod{\nabla f_i(x^*), x-x^*}) \quad  \overset{\eqref{eq:prob}}{=} \quad  2n(f(x) - f(x^*)),
\end{eqnarray*}
where in the last step we used the fact that $\sum_{i=1}^n \nabla f_i(x^*) =  n \nabla f(x^*) =  0.$ \hfill \qed
\end{proof}

\begin{theorem}[Expected smoothness] \label{lem:Lfhatuni}
Let $\mS=\mI_S$ be a minibatch sketch where $S$ is a sampling satisfying Assumption~\ref{ass:proper} (in particular, the support of $S$ is $c_1$--uniform). Consider the bias-correcting random variable $\theta_\mS$ given in \eqref{eq:bu80d09hbjdd}. Further, let $f$ satisfy Assumption~\ref{ass:L_C-smoothness}. 
Then the expected smoothness assumptions (Assumptions~\ref{ass:ES1} and \ref{ass:ES2}) are satisfied with constants $\cL_1$ and $\cL_2$ given by\footnote{Recall that   $p_i = \Prb{i\in S}$ for $i\in [n]$, $p_C = \Prb{S=C}$ for $C\subseteq [n]$ and $\mW= \diag(w_1,\ldots, w_n)\succ 0$.}
\begin{eqnarray}
\cL_1 =  \frac{  1}{n c_1^2} \max_i \left\{\sum_{C \in \support(S) \;:\; i\in C} \frac{|C| L_C}{p_C}\right\}, \qquad
\cL_2  =  n \
\max_i \left\{\frac{p_i L_i}{w_i}\right\},
\label{eq:LhatFnonuni}
\end{eqnarray}
where $L_i = L_{\{i\}}$. If moreover, $S$ is $\tau$-uniform, then\footnote{Note that $c_1 =|\{C\in \support(S) \;:\; 1\in C\}|$, and hence $\cL_1$ has the form of a maximum over averages.}
\begin{eqnarray}
\cL_1 = \LGconst \eqdef  \max_i  \left\{\frac{1}{c_1} \sum_{C \in \support(S) \;:\; i\in C} L_C \right\}, \qquad   
\cL_2  =  \tau \
\max_i \left\{\frac{L_i}{w_i}\right\}.
\label{eq:LhatFuni}
\end{eqnarray}
\end{theorem}

\begin{proof}
Let $\mR = \Jac(x) -\Jac(x^*)$ and $A=\ED{\norm{\nabla f_{\mS}(x) - \nabla f_{\mS}(x^*)}_2^2}$. Then
\begin{eqnarray*}
A &\overset{\eqref{eq:8gd088gs899}}{=} & 
\ED{\frac{\theta_{\mS}^2}{n^2}\norm{\mR \Proj_{\mS}  \ones }_2^2} 
\quad \overset{\eqref{eq:bu80d09hbjdd}}{=} \quad \sum_{C \in \support(S)} \frac{p_C}{c_1^2 p_C^2 n^2 } \norm{\mR \Proj_{\mI_C}  \ones }_2^2 \\
&=&    \sum_{C \in \support(S)} \frac{1}{c_1^2 p_C n^2 }  \Tr{\ones^\top \Proj_{\mI_C}^\top  \mR^\top \mR \Proj_{\mI_C}  \ones  } \quad = \quad \sum_{C \in \support(S)} \frac{1}{c_1^2 p_C n^2 }     \Tr{\mR^\top \mR \Proj_{\mI_C}  \ones \ones^\top \Proj_{\mI_C}^\top  } \\
& \overset{\text{Lem~\ref{prop:bus98g90s09}(iii)}}{=} & \sum_{C \in \support(S)} \frac{1}{c_1^2 p_C n^2 }   \Tr{\mR^\top \mR  e_C e_C^\top} \quad = \quad  \sum_{C \in \support(S)} \frac{1}{c_1^2 p_C n^2 }    \norm{(\Jac(x) -\Jac(x^*))e_C}_2^2 \\
& = &  \sum_{C \in \support(S)} \frac{|C|^2}{c_1^2 p_C n^2 }    \norm{\nabla f_{C}(x)-\nabla f_{C}(x^*)}_2^2 .
\end{eqnarray*}

Using \eqref{eq:bdf7gd899u} and  \eqref{eq:f_C-def}, we can continue:
\begin{eqnarray}
A & \overset{\eqref{eq:bdf7gd899u}}{\leq} &   \sum_{C \in \support(S)} \frac{2L_C |C|^2}{c_1^2 p_C n^2 }       (f_{C}(x) - f_{C}(x^*) -\dotprod{\nabla f_{C}(x^*), x-x^*})\nonumber \\
& \overset{\eqref{eq:f_C-def}}{=} &   \frac{2}{c_1^2 n^2}\sum_{C \in \support(S)} \frac{L_C |C|^2}{ p_C  }       \frac{1}{|C|}\sum_{i\in C} (f_{i}(x) - f_{i}(x^*) -\dotprod{\nabla f_{i}(x^*), x-x^*})\nonumber \\
& = &   \frac{2}{c_1^2 n^2}\sum_{C \in \support(S)}        \sum_{i\in C} (f_{i}(x) - f_{i}(x^*) -\dotprod{\nabla f_{i}(x^*), x-x^*})\frac{L_C |C|}{ p_C  }\nonumber \\
& = &   \frac{2}{c_1^2 n^2}   \sum_{i=1}^n \sum_{C \in \support(S)\;:\; i\in C}      (f_{i}(x) - f_{i}(x^*) -\dotprod{\nabla f_{i}(x^*), x-x^*})\frac{L_C |C|}{ p_C  }\nonumber \\
& = &   \frac{2}{c_1^2 n^2}   \sum_{i=1}^n   (f_{i}(x) - f_{i}(x^*) -\dotprod{\nabla f_{i}(x^*), x-x^*}) \sum_{C \in \support(S)\;:\; i\in C}    \frac{L_C |C|}{ p_C  }\nonumber \\
& \leq &   \frac{2}{c_1^2 n} \max_i \left\{ \sum_{C \in \support(S)\;:\; i\in C}    \frac{L_C |C|}{p_C}  \right\}  \frac{1}{n}\sum_{i=1}^n   (f_{i}(x) - f_{i}(x^*) -\dotprod{\nabla f_{i}(x^*), x-x^*}), \label{eq:oiaaasdaim}
\end{eqnarray}
where in this last inequality  we have used convexity of $f_i$ for $i\in [n]$. Since
	\[
\frac{1}{n}\sum_{i=1}^n   (f_{i}(x) - f_{i}(x^*) -\dotprod{\nabla f_{i}(x^*), x-x^*}) =	
	f(x) - f(x^*) -\dotprod{\nabla f(x^*), x-x^*} = f(x) - f(x^*),\]
the formula for $\cL_1$ now follows  by comparing~\eqref{eq:oiaaasdaim} to~\eqref{eq:ES1}. In order to establish the formula for $\cL_2$, we estimate
	\begin{eqnarray}
	\ED{\norm{\mR \Proj_\mS}_{\mW^{-1}}^2} & \overset{\eqref{eq:fro_norm}}{=} &
	\ED{\norm{\mR \Proj_\mS \mW^{-1/2}}_{\mI}^2} \quad \overset{\eqref{eq:fro_norm}}{=}  \quad \Tr{ \mR^\top \mR \ED{\Proj_\mS \mW^{-1} \Proj_\mS^\top}  } \notag \\
	& \overset{\eqref{eq:projext}}{=} &\Tr{\mR^\top \mR \ED{\mH_\mS}  } \quad = \quad  \Tr{\mD_L^{-1/2} \mR^\top \mR \mD_L^{-1/2} \mD_L^{1/2} \ED{\mH_\mS} \mD_L^{1/2}} \nonumber \\
	& \leq & \norm{\mR}_{\mD_L^{-1}}^2 \lambda_{\max}\left(\mD_L^{1/2} \ED{\mH_\mS} \mD_L^{1/2}\right) \nonumber  \\ 
	&\overset{\eqref{eq:DFDFstL-1}}{\leq} &   2n \lambda_{\max}\left(\mD_L^{1/2} \ED{\mH_\mS} \mD_L^{1/2}\right) (f(x^k) -f(x^*)).
	\label{eq:RLWEZWbnd}
	\end{eqnarray}

From Lemma~\ref{prop:bus98g90s09}(iv) we have  $\ED{\mH_\mS} = \ED{\Proj_\mS}\mW^{-1} = \mP\mW^{-1} = \diag\left(\frac{p_1}{w_1},\dots,\frac{p_n}{w_n}\right)$, and hence $\mD_L^{1/2} \ED{\mH_\mS} \mD_L^{1/2} = \diag\left( \frac{p_1 L_1 }{w_1},\cdots, \frac{p_n L_n }{w_n}\right)$.
 Comparing to the definition of $\cL_2$ in~\eqref{eq:ES2} to \eqref{eq:RLWEZWbnd}, we conclude that
\[
\cL_2 = 
  n\lambda_{\max}\left(\mD_L^{1/2}\mW^{-1} \mD_L^{1/2}\right) = n \max_{i}  \left \{ \frac{p_i L_i}{w_i} \right \}.
\]

The specialized formulas \eqref{eq:LhatFuni} for $\tau$--uniform sampling follow as special cases of the general formulas \eqref{eq:LhatFnonuni} by applying Lemma~\ref{lem:hi8s8bvf786sfvs}.
\hfill \qed
\end{proof}

In the next result we establish some inequalities relating the quantities $L$, $L_{\max}$, $L_C$ and $\LGconst.$ In particular, the results says that for a certain family of samplings $S$ (the same for which we have defined the quantity $\LGconst$ in \eqref{eq:LhatFuni}), the expected smoothed constant $\LGconst$ is lower-bounded by the average of $L_C$ over $C\in \cG=\support(S)$, and upper-bounded by $L_{\max}$.

\begin{theorem} \label{lem:two_ineqXXX} Let $S$ be a $\tau$--uniform sampling ($\tau\geq 1$) with $c_1$--uniform support ($c_1\geq 1$). Let $\cG=\support(S)$. Then \begin{equation}\label{eq:h98gf09hf} f(x) = \frac{1}{|\cG|}\sum_{C\in \cG} f_C(x).\end{equation} Moreover,
\begin{equation} \label{eq:b7878fg89fg656} L\quad \leq \quad\frac{1}{|\cG|} \sum_{C\in \cG} L_C \quad \leq \quad \LGconst \quad \leq \quad  L_{\max}.\end{equation}
The last inequality holds without the need to assume $\tau$--uniformity.
\end{theorem}
\begin{proof}
 Using the fact that $S$ has $c_1$--uniform support, and utilizing a double-counting argument, we observe that 
$\sum_{C\in \cG} |C| f_C(x) = c_1 \sum_{i=1}^n f_i(x)$. Multiplying both sides by $\frac{1}{n c_1}$, and since $|C|=\tau$ for all $C\in \cG$, we get $\frac{\tau |\cG|}{c_1 n} \frac{1}{|\cG|}\sum_{C\in \cG} f_C(x) = \frac{1}{n}\sum_{i=1}^n f_i(x) = f(x).$ To obtain \eqref{eq:h98gf09hf}, it now only remains to use the identity \begin{equation}\label{eq:ub98gd9djbjh}\frac{\tau |\cG|}{ c_1 n} = 1\end{equation}
which was shown in Lemma~\ref{lem:hi8s8bvf786sfvs}. The first inequality in \eqref{eq:b7878fg89fg656} follows from \eqref{eq:h98gf09hf} using standard arguments (identical to those that lead to the inequality $L\leq \bar{L}$). 

Let us now establish the second inequality in \eqref{eq:b7878fg89fg656}. Define $\LGi \eqdef \frac{1}{c_1}\sum_{C\in \cG \;:\; i\in C} L_C$.  Again using a double-counting argument we observe that
$\tau \sum_{C\in \cG} L_C = c_1 \sum_{i=1}^n \LGi.$ Multiplying both sides of this equality by $\frac{|\cG|}{c_1 n}$ and using identity \eqref{eq:ub98gd9djbjh}, we get
$\frac{1}{|\cG|} \sum_{C\in \cG} L_C = \frac{1}{n}\sum_{i=1}^n \LGi \leq \max_i \LGi = \LGconst.$
 We will now establish the last inequality by  proving that $\LGi \leq L_{\max}$ for any $i$:
\begin{eqnarray*}\LGi &=& \frac{1}{c_1} \sum_{C\in \cG\;:\; i\in C} L_C \quad \leq \quad  \frac{1}{c_1} \sum_{C\in \cG\;:\; i\in C} \frac{1}{|C|} \sum_{i\in C} L_i 
\quad \leq \quad    \frac{1}{c_1} \sum_{C\in \cG\;:\; i\in C} \frac{1}{|C|} \sum_{i\in C} L_{\max} \\
&= & L_{\max}  \frac{1}{c_1} \sum_{C\in \cG\;:\; i\in C} \underbrace{\frac{1}{|C|} \sum_{i\in C} 1}_{=1} 
\quad \leq \quad   L_{\max}  \underbrace{\frac{1}{c_1} \sum_{C\in \cG\;:\; i\in C} 1}_{=1} \quad \leq \quad L_{\max}. \end{eqnarray*}
Note that we did not need to assume $\tau$--uniformity to prove that $\LGconst\leq L_{\max}$.
\hfill \qed
\end{proof}

\subsection{Estimating the sketch residual $\rho$}
\label{sec:computeconsts}

In this section we compute the sketch residual $\rho$ for several classes of samplings $S$.  Let $\cG=\support(S)$. We will assume throughout this section that $S$ is non-vacuous, has $c_1$--uniform support (with $c_1\geq 1$), and is $\tau$--uniform. 

Further, we assume that $\mW=\diag(w_1,\dots,w_n)$, and that the bias-correcting random variable $\theta_\mS$ is chosen as  $\theta_\mS=\tfrac{1}{c_1 p_S} = \tfrac{|\cG|}{c_1}$ (see \eqref{eq:unbiasedpart} and Lemma~\ref{lem:hi8s8bvf786sfvs}). In view of the above, since $\Proj_{\mI_C}\ones = e_C$, the sketch residual is given by
	\begin{eqnarray}
	\rho & \overset{\eqref{eq:theorhoXX}}{=}& 
	 \lambda_{\max}\left(\mW^{1/2}\left( \frac{|\cG|^2}{c_1^2} \ED{ \Proj_{\mS} \ones \ones^\top \Proj_{\mS}} -\ones \ones^\top\right)\mW^{1/2}\right) \nonumber \\
&	= &  \lambda_{\max}\left(\mW^{1/2}\left( \frac{|\cG|}{c_1^2} \sum_{C \in \cG}  e_C e_C^\top -\ones \ones^\top\right) \mW^{1/2}\right)  \quad = \quad \lambda_{\max}\left(\left( \frac{|\cG|}{c_1^2} \sum_{C \in \cG}  e_C e_C^\top -\ones \ones^\top\right) \mW\right)  \label{eq:theorhoXX2},
	\end{eqnarray}
where the last equality follows by permuting the multiplication of matrices within the $\lambda_{\max}.$ 


In the  following text we calculate upper bounds for $\rho$ for $\tau$--partition and $\tau$--nice samplings. Note that Theorem~\ref{theo:convgen} still holds if we use an upper bound of $\rho$ in place of $\rho$. 

\begin{theorem} \label{thm:08hs9sd} If $S$ is the $\tau$--partition sampling, then 
	\begin{equation} \label{eq:rhounipart}
	\rho \leq  \frac{n}{\tau} \max_{C \in \cG} \sum_{i \in C} w_i.
	\end{equation}
\end{theorem}

\begin{proof}
	Using Lemma~\ref{lem:hi8s8bvf786sfvs}, and since  $c_1 =1$, we get  $\frac{|\cG|}{c_1^2} = \frac{n}{\tau}$. Consequently,
	\begin{eqnarray}\rho
	& \overset{\eqref{eq:theorhoXX2}}{\leq} &
	\frac{n}{\tau}\lambda_{\max}\left(\sum_{C \in \cG} e_C e_C^\top \mW\right) \quad = 
	\quad  \frac{n}{\tau}\lambda_{\max}\left(\sum_{C \in \cG} e_C w_C^\top \right), \label{eq:rhounipartbnd}\end{eqnarray}
		where $w_C=\sum_{i \in C} w_i e_i $ and we used that $-\mW^{1/2} \ones \ones^\top \mW^{1/2}$ is negative semidefinite. When $\mW = \mI$, the above bound is tight.
By Gershgorin's theorem, every eigenvalue $\lambda$ of the matrix is bounded by at least one of the inequalities $\lambda \leq \sum_{i \in C} w_i$ for $ C \in \cG$. 	Consequently, from~\eqref{eq:rhounipartbnd} we have that $\rho \leq  \frac{n}{\tau} \max_{C \in \cG} \sum_{i \in C} w_i. $ \hfill \qed
	
\end{proof}

Next we give an useful upper bound on $\rho$ for a large family of uniform  samplings (for proof, see Appendix~\ref{app:rho}).

\begin{theorem} \label{thm:rhobnbcollect}
Let  $\cG$ be a collection of subsets of $[n]$ with the property that 
the number of sets $C\in \cG$ containing distinct elements $i,j\in [n]$ is the same for all $i, j$. In particular, define 
	\begin{equation}\label{eq:c_2_xxx}
	c_2\eqdef |\{C \, : \, \{1,2\} \subseteq C, \, C \in\cG\}|.
	\end{equation}
Now define a sampling $S$ by setting $S=C\in \cG$ with probability $\frac{1}{|\cG|}$. Moreover, assume that the support of $S$ is $c_1$--uniform. Consider the minibatch sketch $\mS=\mI_{S}$.
\begin{itemize}
\item[(i)] If $\mW =\diag(w_{1}, \ldots, w_{n})$, then
	\begin{equation} \label{eq:rhocollect}
	\rho \leq \max_{i=1,\ldots, n} \left\{ \left( \frac{|\cal G|}{c_1}  -1\right) w_i+\sum_{j\neq i} w_j \left|\frac{| \cG | c_2}{c_1^2}-1\right|\right\}.
	\end{equation}
\item[(ii)]	If $\mW = \mI$, then
	\begin{equation} \label{eq:rhocollectI}
	\rho  =  \max\left\{\frac{| \cG | }{c_1}\left(1+ (n-1) \frac{c_2}{c_1}\right) -n  ,\frac{| \cG | }{c_1}\left(1  - \frac{c_2}{c_1}\right)\right\}.
	\end{equation}
\end{itemize}	

\end{theorem}

Note that as long as $\tau\geq 2$, the $\tau$--nice sampling $S$  satisfies the assumptions of the above theorem. Indeed, $\cG$ is the support of $S$ consisting of all subsets of $[n]$ of size $\tau$, $|\cG| = {n \choose \tau}$, $c_1= {n-1 \choose \tau-1}$, and $c_2={n-2 \choose \tau-2}$. As a result, bound~\eqref{eq:rhocollect} simplifies to
	\begin{equation} \label{eq:rhosetcollectW}
	\rho \leq \left( \frac{n}{\tau}  -1\right)  \max_{i=1,\ldots, n} \left\{ w_i+\frac{1}{n-1}\sum_{j\neq i} w_j \right\},
	\end{equation}
	and \eqref{eq:rhocollectI} simplifies to
	\begin{equation} \label{eq:rhosetcollectI}
	\rho =  \frac{n}{\tau}\frac{n-\tau}{n-1}.
	\end{equation}

\subsection{Calculating the iteration complexity for special cases}
\label{sec:98hs98hs8gh}

In this section we consider minibatch SAGA (Algorithm~\ref{alg:SketchJac_subsample}) and calculate its  iteration complexity in special cases using Theorem~\ref{theo:convgen} by pulling together the formulas for $\cL_1, \cL_2, \kappa$ and $\rho$ established in previous sections.  In particular, assume $S$ is $\tau$--uniform and has $c_1$--uniform support with $c_1\geq 1$. In this case, formula \eqref{eq:LhatFuni} for $\cL_1,\cL_2$ from 
Lemma~\ref{lem:Lfhatuni} applies and we have $\cL_1 = \LGconst$ and $\cL_2 = \tau \max_i \left\{\frac{L_i}{w_i}\right\}$.


 Moreover, by Lemma~\ref{lem:hi8s8bvf786sfvs}, $\kappa=\tfrac{\tau}{n}$. By Theorem~\ref{theo:convgen}, if we use the stepsize 
\begin{equation}\label{eq:stepsizeuniform}
\alpha  =  \min\left\{ \frac{1}{4 \cL_1 }, \, \frac{\kappa}{4 \cL_2 \rho/n^2   +\mu  }\right\} = \frac{1}{4}\min\left\{ \frac{1}{ \LGconst }, \, \frac{1 }{ \frac{\rho}{n}\max_{j=1,\ldots,n} \left\{ \frac{L_j}{w_j} \right\}    +\frac{\mu}{4} \frac{n}{\tau}   }\right\},
\end{equation}
then the iteration complexity is given by
\begin{equation}
\max \left\{ \frac{4 \cL_1}{\mu}, \, \frac{1}{\kappa} + 
\frac{4 \rho \cL_2 }{\kappa \mu   n^2}    \right\} \log \left(\frac{1}{\epsilon}\right)
=
\max \left\{ \frac{4\LGconst}{\mu}, \, \frac{n}{\tau} +  \frac{4 \rho}{\mu n}
\max_{i} \left\{ \frac{L_i}{w_i} \right\}     \right\} \log \left(\frac{1}{\epsilon}\right). \label{eq:itercompluni}
\end{equation}
Complexity~\eqref{eq:itercompluni} is listed in line 9 of Table~\ref{tbl:complexity_summary}. The complexities in lines  3, 5 and 10--13 arise as special cases of \eqref{eq:itercompluni} for specific  choices  of $S$:
\begin{itemize}
\item In line 3 we have {\em gradient descent}. This arises for the choice $\mW=\mI$ and $S=[n]$ with probability 1. In this case, $\tau=n$, $\LGconst = L$ and $\rho=0$. So, \eqref{eq:itercompluni} simplifies to
\begin{equation}\label{eq:bu98dg8sbh*YH}
\frac{4L}{\mu} \log \left(\frac{1}{\epsilon}\right).
\end{equation}
\item In line 5 we have {\em uniform SAGA.} We choose $\mW=\mI$ and $S=\{i\}$ with probability $1/n$. We have $\tau=1$ and $\LGconst = L_{\max}$. In view of Theorem~\ref{thm:08hs9sd}, $\rho\leq n$. So, \eqref{eq:itercompluni} simplifies to
\begin{equation}\label{eq:bu98dg8sbh*YH097809}
\left(n+\frac{4L_{\max}}{\mu}\right) \log \left(\frac{1}{\epsilon}\right).
\end{equation}
\item In line 10 we choose $\mW=\mI$ and $S$ is the $\tau$-nice sampling. In this case, Theorem~\ref{thm:rhobnbcollect} says that $\rho =  \frac{n}{\tau}\frac{n-\tau}{n-1}$ (see \eqref{eq:rhosetcollectI}). Therefore, 
 \eqref{eq:itercompluni} reduces to
\begin{equation}\label{eq:no8sh09hs9h0(((}
\max \left\{ \frac{4\LGconst}{\mu}, \, \frac{n}{\tau} +  \frac{n-\tau}{(n-1)\tau}\frac{4 L_{\max}}{\mu }
    \right\} \log \left(\frac{1}{\epsilon}\right).
\end{equation}
\item In line 11 we choose $\mW=\diag(L_i)$ and $S$ is the $\tau$-nice sampling. Theorem~\ref{thm:rhobnbcollect} says that $\rho \leq \tfrac{n-\tau}{\tau}\left(\tfrac{n-2}{n-1}L_{\max} + \tfrac{n}{n-1}
\bar{L}\right)$ (see \eqref{eq:rhosetcollectW}). Therefore, 
 \eqref{eq:itercompluni} reduces to
\begin{equation}\label{eq:nh09s8h09sh9JJ}
\max \left\{ \frac{4\LGconst}{\mu}, \, \frac{n}{\tau} +  \frac{n-\tau }{\tau n} \frac{4 \left( \tfrac{n-2}{n-1}L_{\max} + \tfrac{n}{n-1}\bar{L}\right)}{\mu}
    \right\} \log \left(\frac{1}{\epsilon}\right).
\end{equation}
To simplify the above expression, one may further use the bound $\tfrac{n-2}{n-1}L_{\max} + \tfrac{n}{n-1}\bar{L} \leq L_{\max} + \bar{L}$. In Table~\ref{tbl:complexity_summary} we have listed the complexity in this simplified form. Whether \eqref{eq:no8sh09hs9h0(((} or \eqref{eq:nh09s8h09sh9JJ} is better depends on the constants $\{L_i\}$. Indeed, when there exists $i$ such that $L_{i} \gg L_j$, for $j \neq i$ then
\[ \left(\frac{1}{\tau}  -\frac{1}{n} \right)  \left(L_{\max}+ \bar{L} \right) \approx
\frac{n-\tau}{n\tau}(L_{\max}+\frac{1}{n}L_{\max})
=
\frac{n-\tau}{n\tau}\frac{n+1}{n}L_{\max}
\leq \frac{1}{\tau}\frac{n-\tau}{n-1} L_{\max},
 \] 
 thus ~\eqref{eq:nh09s8h09sh9JJ} is smaller than~\eqref{eq:no8sh09hs9h0(((}.
On the other extreme, when $L_i = L_j$ for all $i,j$, then
\[ \left(\frac{1}{\tau}  -\frac{1}{n} \right)  \left(L_{\max}+ \bar{L} \right) =
2\frac{n-\tau}{n\tau}L_{\max}
\geq \frac{1}{\tau}\frac{n-\tau}{n-1} L_{\max},
 \] 
 so long as $n \geq 1$. In this case~\eqref{eq:nh09s8h09sh9JJ} is larger than
 \eqref{eq:no8sh09hs9h0(((}.
 
\item In line 12 of Table~\ref{tbl:complexity_summary} we let $\mW=\mI$ and $S$ is the $\tau$-partition sampling. In view of Theorem~\ref{thm:08hs9sd}, $\rho\leq \tfrac{n}{\tau}\tau = n$ and
hence
 \eqref{eq:itercompluni} reduces to
\begin{equation}\label{eq:nbo9s80s9h}
\max \left\{ \frac{4\LGconst}{\mu}, \, \frac{n}{\tau} + \frac{4 L_{\max}}{\mu}
    \right\} \log \left(\frac{1}{\epsilon}\right).
\end{equation}
\item  In line 13 of Table~\ref{tbl:complexity_summary} we let $\mW=\diag(L_i)$ and $S$ is the $\tau$-partition sampling. In view of Theorem~\ref{thm:08hs9sd}, $\rho\leq  \frac{n}{\tau}\max_{C\in \cG} \sum_{i\in C} L_i$ and
hence
 \eqref{eq:itercompluni} reduces to
\begin{equation}\label{eq:bs98g9ofgebd}
\max \left\{ \frac{4\LGconst}{\mu}, \, \frac{n}{\tau} + \frac{4 \max_{C\in \cG} \sum_{i\in C} L_i}{\mu \tau}
    \right\} \log \left(\frac{1}{\epsilon}\right).
\end{equation}
Note that the bound in~\eqref{eq:bs98g9ofgebd} is  better than~\eqref{eq:nbo9s80s9h}  because
$ \max_{C \in \cG}\sum_{i \in C} L_i \leq \tau L_{\max}.$
\end{itemize}

\subsection{Comparison with previous mini-batch SAGA convergence results}
\label{sec:hofmann}

Recently in~\cite{HofmannLLM152015}, a method that includes a mini-batch variant of SAGA  was proposed. This work is the most closely related to our minibatch SAGA, and was developed independently from ours. The methods described in~\cite{HofmannLLM152015} can be cast in our framework.
In the language of our paper, in \cite{HofmannLLM152015} the authors update the Jacobian estimate according to~\eqref{eq:SAGA-Jacobian_update}, where $S_k$ is sampled according to a uniform probability with $p_i = \tau/n,$ for all $i=1,\ldots, n.$  What~\cite{HofmannLLM152015} do differently is that instead of introducing the bias-corecting random variable $\theta_{\mS}$ to maintain an unbiased gradient estimate, the gradient estimate is updated using the standard SAGA update~\eqref{eq:SAGAstandard} and this sampling process is done independently of how $S_k$  is sampled for the Jacobian update. Thus at every iteration a gradient $\nabla f_i(x^k)$ is sampled to compute~\eqref{eq:SAGAstandard}, but is then discarded and not used to update the Jacobian update so as to maintain the independence between $\mJ^{k}$ and $g^k.$ By introducing the bias-correcting random variable $\theta_{\mS}$ in our method we avoid the data-hungry strategy used in~ \cite{HofmannLLM152015}.

The analysis provided in~\cite{HofmannLLM152015} shows that, by choosing the stepsize appropriately,  the expectation of a Lyapunov function similar to~\eqref{eq:lyapgen} is less than $\epsilon>0$ after
\begin{equation}\label{eq:hofcomplex} \frac{1}{2} \left(\frac{n}{\tau} + K +\sqrt{\frac{n^2}{\tau^2}+K^2}\right)\log \left(\frac{1}{\epsilon}\right)\end{equation}
iterations, where $K \eqdef \frac{4 L_{\max}}{ \mu}$. When $\tau =1$ this gives an iteration complexity of $O(n +K)\log\frac{1}{\epsilon},$ which is essentially the same complexity as the standard SAGA method.  
The main issue with this complexity is that it decreases only very modestly as $\tau$ increases.
  In particular, on the extreme end when $\tau =n$, since $K \geq 4$, we can approximate 
 $(1 +K)^2 \approx 1+K^2$ and the resulting complexity~\eqref{eq:hofcomplex} becomes
\[ \left(1+\frac{4L_{\max}}{\mu}\right)\log \left(\frac{1}{\epsilon}\right).\]
Yet we know that $\tau =n$ corresponds to  gradient descent, and thus the iteration complexity should be $O(\tfrac{ L}{ \mu }\log(1/\epsilon)),$ which is what we recover in the analysis of all our mini-batch variants. In  Figures~\ref{fig:comp1},~\ref{fig:comp2} and~\ref{fig:comp3} in the experiments in Section~\ref{sec:experiments} we illustrate how~\eqref{eq:hofcomplex} descreases very modestly as $\tau$ increases.

\section{A Refined Analysis  with a Stochastic Lyapunov Function}
\label{sec:theopart}

In this section we perform a refined analysis of JacSketch applied with  a minibatch sketch $\mS =\mI_S $ for a special class of samplings $S$ which pick uniformly at random from a partition of $[n]$ into sets of size $\tau$. \footnote{This is only possible when $n$ is a multiple of $\tau$.}

\begin{assumption} Let $\cG$ be a partition of $[n]$ into sets of size $\tau$. Assume that the sampling $S$ picks sets from the partition $\cG$ uniformly at random. That is, $p_C \eqdef \Prb{S=C}= \frac{\tau}{n}$ for $C\in \cG = \support(S)$. A sampling with these properties is called a $\tau$--partition sampling.
\end{assumption}

 In the terminology introduced in Section~\ref{sec:samplings}, a $\tau$--partition sampling is non-vacuous, proper and $\tau$--uniform. Its support is a partition of $[n]$, and is $1$--uniform. It satisfies Assumption~\ref{ass:proper}.

Restricting our attention to $\tau$--partition samplings will allow us to perform a more in-depth analysis of JacSketch using a {\em stochastic Lyapunov function}. Unlike Theorem~\ref{theo:convgen}, and as explained in Section~\ref{sec:intro-summary-of-results}, our main result  in this section (Theorem~\ref{theo:convpart}) is capable of obtaining the conjectured  rate $O((n+\tfrac{\bar{L}}{\mu}) \log \tfrac{1}{\epsilon})$ for SAGA with importance sampling.

One of the key reasons why we restrict our attention to $\tau$-partition samplings is the fact that  
\begin{equation}\label{eq:89h98gs8}\mI_{C_1}^\top \mI_{C_2} = \begin{cases}
\mI \in \R^{\tau\times \tau}, & \quad C_1 = C_2,\\
0 \in \R^{\tau \times \tau}, & \quad C_1 \neq C_2, \end{cases}\end{equation} for $C_1, C_2\in \cG$. Recall from Lemma~\ref{prop:bus98g90s09} that if $\mW=\mI$, then $\Proj_{\mI_C} = \mI_C \mI_C^\top$. Consequently, for $C_1,C_2\in \cG$ we have
\begin{equation}\label{eq:orthoprojs}
 C_1\neq C_2 \quad \Rightarrow \quad \Proj_{\mI_{C_1}} \Proj_{\mI_{C_2}} =0, \qquad  \qquad C_1 =C_2  \quad \Rightarrow \quad  (\mI-\Proj_{\mI_{C_1}}) \Proj_{\mI_{C_2}} = 0.
\end{equation}
This orthogonality property will be fundamental for controlling the convergence of the gradient estimate in Lemma~\ref{lem:gradcontrapart}. 

\subsection{Convergence theorem}

Recall from  \eqref{eq:controlgradJ}  that the stochastic gradient of the controlled stochastic reformulation \eqref{eq:probstochcon} of the original  finite-sum  problem \eqref{eq:prob} is given by
\begin{equation}\label{eq:gkupdatepartCapp} \nabla f_{\mI_S,\mJ}(x) = \frac{1}{n} \mJ \ones + \frac{1}{p_S n} (\Jac(x)-\mJ)   \Proj_{\mI_S}\ones \end{equation}
provided that  we use the minibatch sketch $\mS = \mI_S$ and  bias-correcting variable $\theta_{\mS} = \theta_{\mI_S} = 1/p_S$ given by Lemma~\ref{prop:bus98g90s09}(vi). This object will appear in our Lyapunov function, evaluated at $x=x^*$ and $\mJ=\mJ^k$. We are now ready to present the main result of this section.

\begin{theorem}[Convergence for  minibatch sketches with $\tau$-partition samplings] \label{theo:convpart}

Let
\begin{itemize}
\item $\mS$ be a minibatch sketch (i.e., $\mS=\mI_S$)\footnote{We can alternatively set $\mS = e_S$ and the same results will hold.} , where $S$ is a $\tau$--partition sampling with support $\cG = \support(S)$,
\item $f_C \eqdef \tfrac{1}{|C|}\sum_{i\in C} f_i$  be $L_C$--smooth and  $\mu$--strongly convex (for $\mu>0$) for all $C \in \cG$,  
\item $\mW = \mI$, $\theta_{\mS} = \frac{1}{p_S}$,
\item  $\{x^k, \mJ^k\}$ be the iterates produced by JacSketch.
\end{itemize}

 Consider the \emph{stochastic Lyapunov function}
\begin{equation}\label{eq:stochlyap}
\Psi_S^k  \eqdef  \norm{x^{k} -x^*}_2^2  +2\sigma_{S} \alpha \norm{\frac{1}{n}\mJ^{k}\ones - \nabla f_{\mI_S,\mJ^{k}}(x^*) }_2^2,\end{equation}
where $\sigma_S = \frac{n}{4\tau L_{S}}$ is a {\em stochastic Lyapunov constant}.  If we use a stepsize that satisfies
\begin{equation}
\alpha \leq  \min_{C \in \cG}\ \frac{p_C }{\mu+\frac{4 L_C\tau}{n} }, \label{eq:alphatheosim}
\end{equation}
then
\begin{equation} \label{eq:recursion_908y8hd8}
\E{\Psi_S^{k} } \quad \leq \quad (1-\mu \alpha)^k \cdot \E{ \Psi_S^0  }.
\end{equation}
This means that if we choose the stepsize equal to the upper bound \eqref{eq:alphatheosim}, then
\begin{equation} \label{eq:itercomplexpart}
 k\geq \max_{C \in \cG}\left\{ \frac{1}{p_C}+\frac{4 L_C}{\mu}\frac{\tau}{n p_C}  \right\} \log\left(\frac{1}{\epsilon} \right) \quad \Rightarrow \quad \E{\Psi_S^k} \leq \epsilon \cdot \E{\Psi^0_S}.
\end{equation}
\end{theorem}

\subsection{Gradient estimate contraction}

Here we will show that our gradient estimate contracts in the following sense.

\begin{lemma}\label{lem:gradcontrapart} Let $S$ be the $\tau$--partition sampling, and $\sigma(S) \eqdef \sigma_{S}\geq 0$ be any non-negative random variable. Then
\begin{eqnarray} 
\E{\sigma_{S}  \norm{  \frac{1}{n}\mJ^{k+1}\ones - \nabla f_{\mI_S, \mJ^{k+1}}(x^*) }_2^2} &\leq
&  \E{\sigma_S ( 1 -p_S) \norm{ \frac{1}{n}\mJ^{k}\ones -  \nabla f_{\mI_S, \mJ^{k}}(x^*) }_2^2} \nonumber \\
&&+	\E{ \sigma_{S} p_{S} \norm{ \nabla f_{\mI_S, \mJ^{k}}(x^{k}) - \nabla f_{\mI_S, \mJ^{k}}(x^*)}_2^2}.
\label{eq:jaccontract2}
\end{eqnarray}
\end{lemma}

\begin{proof} For simplicity, in this proof we let $\Jac^k = \Jac(x^k)$ and $\Jac^* = \Jac(x^*)$. Rearranging~\eqref{eq:gkupdatepartCapp}, we have 
\begin{eqnarray}
\frac{1}{n}\mJ^{k+1}\ones - \nabla f_{\mI_S, \mJ^{k+1}}(x^*)  &\overset{\eqref{eq:gkupdatepartCapp} }{=}& \frac{1}{np_{S}}(\mJ^{k+1}-\Jac^* ) \Proj_{\mI_S} \ones \label{eq:gkstartident}\\
  & \overset{\eqref{eq:jacobsol}}{=} & 
\frac{1}{np_{S}}\left(\mJ^{k} -(\mJ^{k}-\Jac^{k}) \Proj_{\mI_{S_k}}-\Jac^* \right) \Proj_{\mI_S}  \ones \nonumber\\
& =&  \frac{1}{np_{S}}(\mJ^{k}-\Jac^*)(\mI- \Proj_{\mI_{S_k}})\Proj_{\mI_S}\ones  + \frac{1}{np_{S}}(\Jac^{k} -\Jac^*)\Proj_{\mI_{S_k}}\Proj_{\mI_S}\ones.\nonumber 
\end{eqnarray}
Taking norm squared on both sides gives
{\footnotesize
\begin{eqnarray}
\Big\|\frac{1}{n}\mJ^{k+1}\ones-\nabla f_{\mI_S, \mJ^{k+1}}(x^*)\Big\|_2^2& = & \underbrace{\frac{1}{n^2p_{S}^2}\Big\|
  \overbrace{(\mJ^{k}-\Jac^*)}^{\mA}(\mI- \Proj_{\mI_{S_k}})\Proj_{\mI_S}\ones
  \Big\|_2^2}_{\text{I}}  +\underbrace{\frac{1}{n^2p_{S}^2}\Big\|\overbrace{ 
(\Jac^{k} -\Jac^*)}^{\mR}\Proj_{\mI_{S_k}}\Proj_{\mI_S}\ones  
  \Big\|_2^2}_{\text{II}}\nonumber \\
  & & + 2\frac{1}{n^2p_{S}^2}\underbrace{\dotprod{
  (\mJ^{k}-\Jac^*)(\mI- \Proj_{\mI_{S_k}})\Proj_{\mI_S}\ones  
(\Jac^{k} -\Jac^*)\Proj_{\mI_{S_k}}\Proj_{\mI_S}\ones    
  }}_{\text{III}}.\label{eq:isuUYB&98}
\end{eqnarray}
}

First, it follows from \eqref{eq:orthoprojs} that expression III is zero.
We now multiply expressions I and II by $\sigma_{S}$  and bound certain conditional expectations of these terms.   Since $S$ and $S_k$ are independent samplings, we have
{\footnotesize	
	\begin{eqnarray}
	\E{\frac{\sigma_{S}}{n^2 p_{S}^2}\norm{\mA(\mI- \Proj_{\mI_{S_k}})\Proj_{\mI_S}\ones}_2^2 \;|\; \mA} & =&
	\sum_{C \in \cG}\sum_{C' \in \cG} p_{C} p_{C'}  \frac{\sigma_{C}}{n^2 p_{C}^2} \norm{\mA (\mI- \Proj_{\mI_{C'}})\Proj_{\mI_{C}}\ones}_2^2\nonumber \\
	& \overset{\eqref{eq:orthoprojs}}{=}&
		\sum_{C \in \cG} \frac{\sigma_{C}}{n^2 p_{C}}   \norm{\mA\Proj_{\mI_{C}}\ones}_2^2 	\sum_{C' \in \cG, \, C' \neq C} p_{C'} \nonumber \\
	& =& 	\sum_{C \in \cG} \frac{\sigma_{C}}{n^2 p_{C}}  ( 1 -p_C)\norm{\mA\Proj_{\mI_{C}}\ones}_2^2\nonumber \\
	& =& 	\sum_{C \in \cG} p_C  \sigma_{C}( 1 -p_C) \frac{1}{n^2 p_{C}^2} \Big\|\mA\Proj_{\mI_{C}}\ones\Big\|_2^2\nonumber \\	
	& \overset{\eqref{eq:gkstartident}}{ =} & \E{\sigma_S ( 1 -p_S)\Big\|\frac{1}{n}\mJ^{k}\ones -  \nabla f_{\mI_S, \mJ^{k}}(x^*) \Big\|_2^2 \;|\; \mJ^{k}}.\label{eq:na9sd23}
	\end{eqnarray}
	}
		
Taking conditional expectation over expression II yields
	\begin{eqnarray}
	\E{\frac{\sigma_{S}}{n^2p_{S}^2}\Big\| \mR \Proj_{\mI_{S_k}}\Proj_{\mI_S}\ones\Big\|_2^2 \;|\; \mR, S_k} 
	&= &  \sum_{C \in \cG}p_C \frac{\sigma_C}{n^2 p_C^2}   \norm{ \mR \Proj_{\mI_{S_k}}\Proj_{\mI_C}\ones}_2^2 \nonumber \\
	& \overset{\eqref{eq:orthoprojs}}{=}&
\frac{\sigma_{S_{k}}}{ n^2 p_{S_{k}}}   \norm{ \mR \Proj_{\mI_{S_k}}\Proj_{\mI_{S_k}} \ones}_2^2 \nonumber \\
&= &\frac{\sigma_{S_{k}}}{ n^2 p_{S_{k}}}   \norm{ \mR \Proj_{\mI_{S_k}} \ones}_2^2 \nonumber \\
		& =&  \sigma_{S_k} p_{S_k} \norm{ \nabla f_{\mI_{S_k}, \mJ^{k}}(x^{k}) - \nabla f_{\mI_{S_k}, \mJ^{k}}(x^*)}_2^2, \label{eq:na9sd232}
	\end{eqnarray}		
	where in the last equation  we used the identity
	\begin{equation}\label{eq:stochgraddiff}
	\norm{ \nabla f_{\mI_C, \mJ}(x) -\nabla f_{\mI_C, \mJ}(y) }_2^2 = 
	\norm{ \tfrac{1}{n p_C}(\Jac(x) - \Jac(y)) \Proj_C \ones }_2^2, \quad \forall \mJ \in \R^{d \times n}, \forall C \in \cG,
	\end{equation}
	which in turn is a specialization of \eqref{eq:8gd088gs899} to the minibatch sketch $\mS=\mI_S$ and the specific choice of the bias-correcting variable $\theta_\mS = 1/p_S$. It remains to take expectation of \eqref{eq:na9sd23} and~\eqref{eq:na9sd232}, apply the tower property,  and combine this with \eqref{eq:isuUYB&98}. \hfill \qed
\end{proof}

\subsection{Bounding the second moment of $g^k$}

In the next lemma we bound the second moment of our gradient estimate $g^k$. 


\begin{lemma}\label{lem:gradient_bounddelta} 
	The second moment of the gradient estimate is bounded by
\begin{eqnarray}\label{eq:gradbndsubdelta2}
\E{\norm{g^k}_2^2 \;|\; \mJ^k, x^k} & \leq 	&
2\E{\norm{\nabla f_{\mI_{S}, \mJ^k}(x^k)-\nabla f_{\mI_{S}, \mJ^k}(x^*) }_2^2 \;|\; \mJ^k,x^k} \notag \\
&& \quad +2\E{\Big\| \nabla f_{\mI_{S}, \mJ^{k}}(x^*)-\frac{1}{n}\mJ^{k}\ones\Big\|_2^2 \;|\; \mJ^k, x^k}.\end{eqnarray}
\end{lemma}

\begin{proof}
	Adding and subtracting $\tfrac{1}{np_{S_k}} \Jac(x^*) \Proj_{\mI_{S_k}} \ones $ from~\eqref{eq:gkupdatepartCapp} gives
	\[g^k = \frac{1}{n}\mJ^{k}\ones  - \frac{1}{n p_{S_k}}(\mJ^{k} - \Jac(x^*)) \Proj_{\mI_{S_k}}  \ones + \frac{1}{n p_{S_k}}(\Jac(x^k)-\Jac(x^*))\Proj_{\mI_{S_k}} \ones. \] 
	Taking norm squared on both sides, and using the bound $\|a+b\|_2^2\leq 2\|a\|_2^2 + 2\|b\|_2^2$ gives
	\begin{eqnarray}
	\norm{g^{k}}_2^2 & \leq  & 
		\frac{2}{n^2 p_{S_k}^2} \norm{(\Jac(x^k)-\Jac(x^*))\Proj_{\mI_{S_k}} \ones }_2^2 +
	\frac{2}{n^2 } \norm{ \tfrac{1}{p_{S_k}}(\mJ^{k} - \Jac(x^*)) \Proj_{\mI_{S_k}} \ones   -\mJ^{k} \ones}_2^2 \nonumber \\
	& \overset{\eqref{eq:stochgraddiff}}{=} &
	2\norm{\nabla f_{\mI_{S_k}, \mJ^k}(x^k)-\nabla f_{\mI_{S_k}, \mJ^k}(x^*)}_2^2 +
	\frac{2}{n^2} \underbrace{\norm{ \tfrac{1}{p_{S_k}}(\mJ^{k} - \Jac(x^*)) \Proj_{\mI_{S_k}} \ones   -\mJ^{k} \ones}_2^2}_{A}. \label{eq:gknorm2bnd}
	\end{eqnarray}
	Taking expectation of the $A$ term, we get
	{\footnotesize
	\begin{eqnarray*}
	\mathbb{E}\Big[ \Big\| \underbrace{\tfrac{1}{p_S}(\mJ^{k} - \Jac(x^*)) \Proj_{\mI_{S}} \ones}_{X}   -\underbrace{\mJ^{k} \ones}_{\E{X}}\Big\|_2^2 \;|\; \mJ^k,x^k \Big]& \leq &	
	\E{\norm{\tfrac{1}{p_S}(\mJ^{k} - \Jac(x^*)) \Proj_{\mI_S} \ones}_2^2 \;|\; \mJ^k, x^k}\nonumber \\
	& \overset{\eqref{eq:gkstartident}}{=}& n^2 \E{\norm{\nabla f_{\mI_S, \mJ^{k}}(x^*) -\frac{1}{n}\mJ^{k}\ones}_2^2 \;|\; \mJ^k, x^k},
	\end{eqnarray*}
	}
where we used the inequality $\E{\norm{X-\E{X}}_2^2} \leq \E{\norm{X}_2^2}$.
	The result follows by combining the above with~\eqref{eq:gknorm2bnd}. \hfill \qed
\end{proof}

\subsection{Smoothness and strong convexity of $f_{\mI_C, \mJ}$}

Recalling the setting of Theorem~\ref{theo:convpart}, we assume that each $f_C$ is $\mu$--strongly convex and $L_C$--smooth: 
\[
  f_C(y) + \langle \nabla f_C(y), x-y \rangle + \frac{\mu}{2} \norm{x-y}_2^2 \leq f_C(x) \leq  f_C(y) + \langle \nabla f_C(y), x-y \rangle + \frac{L_C}{2} \norm{x-y}_2^2\]
for all $ C\in \cG$. 
It is known (see Section 2.1 in~\cite{NesterovBook}) that the above conditions imply the following inequality:
\begin{equation}
 \dotprod{\nabla f_C(x)-\nabla f_C(y), x-y} \geq \frac{\mu L_C }{\mu +L_C} \norm{x-y}_2^2 +\frac{1}{\mu+L_C} \norm{\nabla f_C(x)-\nabla f_C(y)}_2^2, \label{eq:dotprod_ineq}
\end{equation}
for all $ x,y \in \R^d$. A  consequence of these assumptions that will be useful to us is that the function $f_{\mI_C, \mJ}$ is $\frac{\tau \mu}{n p_C}  $--strongly convex and $\frac{\tau L_C}{n p_C}  $--smooth. This can in turn be used to establish the next lemma, which will be used in the proof of Theorem~\ref{theo:convpart}:
\begin{lemma} \label{lem:strconwithsmoothst}Under the assumptions of Theorem~\ref{theo:convpart} (in particular, assumptions on $f$ and $S$), we have
	\begin{equation}
	\dotprod{\nabla f(x)-\nabla f(y), x-y} \geq \frac{\mu}{2} \norm{x-y}_2^2 +\ED{\frac{n p_S}{2\tau L_S} \norm{\nabla f_{\mI_{S}, \mJ}(x)-\nabla f_{\mI_{S}, \mJ}(y)}_2^2}, \label{eq:dotprod_Cineq}
	\end{equation}
for all $x,y \in \R^d$ and $ \mJ \in \R^{d \times n}$.
\end{lemma}


\begin{proof}
Applying~\eqref{eq:dotprod_ineq} to the function $f_{\mI_S, \mJ}$  gives
\begin{eqnarray}
 \dotprod{\nabla f_{\mI_S, \mJ}(x)-\nabla f_{\mI_S, \mJ}(y), x-y} &\geq & \frac{\tau}{n p_S}\frac{ \mu L_S }{\mu +L_S} \norm{x-y}_2^2 +\frac{n p_S}{\tau(\mu+L_S)}\norm{\nabla f_{\mI_S, \mJ}(x)-\nabla f_{\mI_S, \mJ}(y)}_2^2\nonumber \\
 & \geq &  \frac{\tau\mu  }{2n p_S} \norm{x-y}_2^2 +\frac{n p_S}{2\tau L_S}\norm{\nabla f_{\mI_S, \mJ}(x)-\nabla f_{\mI_S, \mJ}(y)}_2^2\nonumber.
\end{eqnarray}
Taking expectation over both sides over $S$, noting that $p_S = \frac{\tau}{n}$, and recalling that $\nabla f_{\mI_S,\mJ}(x)$ is an unbiased estimator of $\nabla f(x)$, we get the result.
 \hfill \qed
\end{proof}

\subsection{Proof of Theorem~\ref{theo:convpart}}

Let $\EE{k}{\cdot}$ denote expectation conditional on $\mJ^k$ and $x^k$. We can write
\begin{eqnarray}
\EE{k}{\norm{x^{k+1} -x^*}_2^2 } &\overset{ \eqref{eq:xupdate}}{=} & \EE{k}{\norm{x^k  -x^* - \alpha g^{k}}_2^2} \nonumber \\
& \overset{\eqref{eq:unbiasedgrad}}{=} &  \norm{x^k  -x^*}_2^2 -2\alpha \dotprod{\nabla f(x^k), x^k  -x^*}  + \alpha^2\EE{k}{\norm{g^{k}}_2^2}\nonumber \\ 
& \overset{\eqref{eq:dotprod_Cineq}}{\leq} & (1-\mu \alpha )\norm{x^k  -x^*}_2^2  -\alpha\EE{k}{\frac{n p_S}{\tau  L_S} \norm{\nabla f_{\mI_S, \mJ^k}(x^k)-\nabla f_{\mI_S, \mJ^k}(x^*)}_2^2}    + \alpha^2\EE{k}{\norm{g^{k}}_2^2}\nonumber  \\
& \overset{\eqref{eq:gradbndsubdelta2}}{\leq} &  
(1-\mu \alpha )\norm{x^k  -x^*}_2^2   + 2\alpha^2\EE{k}{\norm{ \frac{1}{n}\mJ^{k}\ones-\nabla f_{\mI_S, \mJ^{k}}(x^*)}_2^2}
  \nonumber \\
 & &+2\alpha\EE{k}{\left(\alpha-\frac{n p_S}{2\tau L_S} \right) \norm{\nabla f_{\mI_S, \mJ^k}(x^k)-\nabla f_{\mI_S, \mJ^k}(x^*)}_2^2}.
\label{eq:convstepsub1} 
\end{eqnarray}
Next, after taking expectation in \eqref{eq:convstepsub1}, applying the tower property, and  subsequently adding the  term \\ $2\alpha\E{\sigma_{S}  \norm{\frac{1}{n}\mJ^{k+1}\ones - \nabla f_{\mI_S,\mJ^{k+1}}(x^*) }_2^2}$ to both sides of the resulting inequality, we get
\begin{eqnarray}
 \E{\Psi_S^{k+1}}  & \leq &
\E{ (1-\mu \alpha )\norm{x^k  -x^*}_2^2}  +2\alpha\E{ \left(\alpha-\frac{n p_S}{2\tau L_S} \right) \norm{\nabla f_{\mI_S, \mJ^k}(x^k)-\nabla f_{\mI_S, \mJ^k}(x^*)}_2^2}\nonumber  \\
&&   + 2\alpha^2\E{\norm{\frac{1}{n}\mJ^{k}\ones -  \nabla f_{\mI_S, \mJ^{k}}(x^*)}^2}
+2\alpha\E{\sigma_{S}  \norm{\frac{1}{n}\mJ^{k+1}\ones - \nabla f_{\mI_S,\mJ^{k+1}}(x^*) }_2^2} \nonumber
\\
&\overset{\eqref{eq:jaccontract2}}{\leq}  &
\mathbb{E}\Big[\underbrace{\left(1-\mu \alpha \right)}_{\text{I}}\norm{x^k  -x^*}_2^2\Big]
 +2 \alpha  \mathbb{E}\Big[\sigma_S \underbrace{\left( 1 -p_S +\frac{\alpha}{\sigma_S} \right)}_{\text{II}}\norm{ \frac{1}{n}\mJ^{k}\ones - \nabla f_{\mI_S, \mJ^{k}}(x^*)}_2^2\Big] \nonumber  \\
&&      +2\alpha\mathbb{E}\Big[\underbrace{\left(\alpha+\sigma_{S} p_{S}-\frac{n p_S}{2\tau L_S} \right)}_{\text{III}} \norm{\nabla f_{\mI_S, \mJ^k}(x^k)-\nabla f_{\mI_S, \mJ^k}(x^*)}_2^2\Big].\label{eq:amo89n238hr}
\end{eqnarray}

Next, we determine a bound on $\alpha$ so that III $\leq0$. Choosing
\begin{equation}
 \alpha  +\sigma_{C} p_{C} -\frac{n p_C}{2\tau  L_{C}} \leq 0, \quad \forall C\in \cG
\qquad \Rightarrow  \qquad \alpha \leq \frac{n p_C}{2\tau L_{C}} - \sigma_{C} p_{C}, \quad \forall C\in \cG, \label{eq:920ncn29}
\end{equation}
guarantees that III $\leq 0$, and thus the last term in term in~\eqref{eq:amo89n238hr} can be safely dropped. Next, to build a recurrence and conclude the convergence proof, we bound the stepsize $\alpha$ so that II $\leq$ I; that is,
\begin{equation}\label{eq:andiuoh29}
 1- p_{C} +\frac{\alpha}{\sigma_{C} } \leq  1- \alpha \mu, \quad \forall C\in \cG \qquad \Rightarrow\qquad \alpha \leq \frac{\sigma_{C} p_{C} }{\mu\sigma_{C}+1},\quad \forall C\in \cG. \end{equation}
 
Consequently,
\[
\E{\Psi^{k+1}_S}   \leq   \E{(1-\mu \alpha )\norm{x^k  -x^*}_2^2}   + 
2  \alpha  \E{\sigma_{S}(1-\mu \alpha )\norm{\frac{1}{n}\mJ^{k}\ones - \nabla f_{\mI_S, \mJ^{k}}(x^*) }_2^2} \\
= (1-\mu \alpha)\E{\Psi^k_S}.\nonumber
\]

Since $\sigma_S = \frac{n}{4\tau L_{S}}$, in view of \eqref{eq:920ncn29} and~\eqref{eq:andiuoh29} the combined bound on $\alpha$ is
\[
\alpha \leq \min \left\{ \frac{np_C}{4 \tau L_C} ,
\frac{p_C }{\mu+4\frac{\tau}{n} L_{C}} \right\}  = \frac{p_C }{\mu+4\frac{\tau}{n} L_{C}}, \quad \forall C \in \cG.
\]
Hence, we have established the recursion \eqref{eq:recursion_908y8hd8}.

\subsection{Calculating the iteration complexity in special cases}
\label{subsec:calitercomspec}

In this section we consider the special case of JacSketch analyzed via Theorem~\ref{theo:convpart}---minibatch SAGA with $\tau$--partition sampling---and look at further special cases by varying the minibatch size $\tau$ and probabilities. Our aim is to justify the complexities appearing in Table~\ref{tbl:complexity_summary}. In view of Theorem~\ref{theo:convpart} the iteration complexity is given by
\begin{equation}
\max_{C\in \cG}  \left( \frac{1}{p_C} + \frac{\tau}{n p_C} \frac{4 L_C}{\mu}\right) \log \left(\frac{1}{\epsilon}\right), \label{eq:b9898d}
\end{equation}
where $\cG=\support(S)$. Complexity~\eqref{eq:b9898d} is listed in line 2 of Table~\ref{tbl:complexity_summary}. The complexities in lines  4, 6, 8 and 14 arise as special cases of \eqref{eq:b9898d} for specific  choices  of $\tau$ and probabilities $p_C$.

\begin{itemize}
\item In line 4 we have {\em gradient descent}. This is obtained by choosing $\cG = \{[n]\}$ (whence  $p_{[n]}=1$, $\tau=n$ and $L_{[n]}=L$),  which is why \eqref{eq:b9898d} simplifies to
\begin{equation} \label{eq:ih80089h0fh89s8*^b}\left(1 + \frac{4 L}{\mu} \right) \log \left(\frac{1}{\epsilon}\right). \end{equation}

\item In line 6 we consider {\em uniform SAGA}. That is, we choose $\tau=1$ and $p_i=1/n$ for all $i$. We have $\cG = \{\{1\}, \{2\}, \dots, \{n\}\}$ and $L_{\{i\}}=L_i$. Therefore,
 \eqref{eq:b9898d} simplifies to
\begin{equation} \label{eq:ihb98sg9s08gb}\left( n + \frac{4 L_{\max}}{\mu} \right) \log \left(\frac{1}{\epsilon}\right). \end{equation}
 This is essentially the same\footnote{With the difference being that in~\cite{SAGA_Nips} the iteration complexity is $2\left(n+\left.L_{\max}\right/\mu   \right)\log \left(\frac{1}{\epsilon}\right),$ thus a small constant change.} complexity result given in~\cite{SAGA_Nips}.
 
\item In line 8 we consider {\em SAGA with importance sampling}. This is the same setup as above, except we choose 
\begin{equation} \label{eq:optprobs}
p_i  =  \frac{\mu n +4 L_i}{\sum_{j=1}^n n\mu +4 L_j} ,
\end{equation}
which is the optimal choice minimizing the  complexity bound in $p_1, \dots,p_n$.  With these optimal probabilities, the stepsize bound becomes
$
\alpha \leq \frac{1}{n\mu +4 \bar{L}},
$
 and by choosing the maximum allowed stepsize  the resulting iteration complexity is 
\begin{equation} \label{eq:itercomplexopt}
\left(n +\frac{4 \bar{L}}{\mu}  \right)
 \log \left(\frac{1}{\epsilon}\right).
\end{equation}
Now consider the probabilities $p_i = \frac{L_i}{\sum_{j=1}^n L_j}$  suggested in~\cite{Schmidtnonuni}. Using our bound, these lead to the    complexity 
  \begin{equation} \label{eq:itercomplexLimark}
 \max_{i=1,\ldots, n}\left\{ \frac{\sum_{j=1}^n L_j}{L_i}+4\frac{ \sum_{j=1}^n L_j}{\mu n }  \right\} \log \frac{1}{\epsilon}  = 
\left(\frac{n\bar{L} }{L_{\min}}+\frac{4\bar{L}}{\mu} \right)\log \left(\frac{1}{\epsilon}\right).
\end{equation}
Comparing this with~\eqref{eq:itercomplexopt}, we see that this non-uniform sampling offers a significant speed up over uniform sampling if  $ n\mu \leq L_{\min}.$  However, our rate \eqref{eq:itercomplexopt} is always better than both \eqref{eq:ihb98sg9s08gb} and  \eqref{eq:itercomplexLimark}. The rate we establish was conjectured to hold for a ``properly'' designed SAGA method in~\cite{Schmidtnonuni}; and we resolve this conjecture.
\item Finally, in line 14 of Table~\ref{tbl:complexity_summary} we optimize over probabilities  $p_C$ directly; that is we extend the importance sampling described above to any $\tau$. Minimizing the complexity bound over the probabilities, and noting that $|\cG| = \frac{n}{\tau}$, this leads to the rate
 \begin{equation} \label{eq:itercomplexoptmini}
\left(\frac{n}{\tau} +\frac{4\frac{1}{|\cG|}\sum_{C \in \cG}L_C}{\mu}  \right)
 \log \left(\frac{1}{\epsilon}\right) .
\end{equation}

This iteration complexity also applies to the reduced memory variant of SAGA~\eqref{eq:jacreducedmem}. This is because Theorem~\ref{theo:convpart} also holds for sketches $\mS = e_S$ where $S$ is a $\tau$--partition sampling. To see this, note that  our analysis in this section relies on the orthogonality property~\eqref{eq:orthoprojs} which also holds for $\mS = e_S$  since (for $\mW=\mI$) we have:
\[ \Proj_{e_{C_1}} \Proj_{e_{C_2}} = \frac{1}{\tau}e_{C_1}(\underbrace{e_{C_1}^\top e_{C_2}}_{=0})e_{C_2}^\top \frac{1}{\tau}= 0,\quad \mbox{for }C_1,C_2 \in \cG,\quad C_1 \neq C_2.\]   Lemmas~\ref{lem:gradcontrapart}, \ref{lem:gradient_bounddelta} and~\ref{lem:strconwithsmoothst} depend on the sketch through $\nabla f_{\mS, \mJ}(x^*)$ only, which in turn depends on the sketch through $ \Proj_{\mS}e$, and it is easy to see that if either $\mS = \mI_S$ or $\mS = e_S$, we have    $ \Proj_{\mS}e = e_S.$

\end{itemize}

\section{Experiments} \label{sec:experiments}

We perform several experiments to validate the theory, and also test the practical relavance of non-uniform SAGA~\eqref{eq:nonuniSAGAup} with the optimized probability distribution~\eqref{eq:optprobs}.  All of our code for these experiments was written in Julia and can be found on github in \url{https://github.com/gowerrobert/StochOpt.jl}. 

In our experiments we test either ridge regression 
\begin{equation}\label{eq:ridge}
f(x) = \frac{1}{2n}\norm{\mA^\top x - y}_2^2 + \frac{\lambda}{2}\norm{x}_2^2,\end{equation}
or logistic regression
\begin{equation}\label{eq:logistic}
f(x) = \frac{1}{n}\sum_{i=1}^n 
\log\left(1 +e^{-y_i\dotprod{a_i,x}}\right) +\frac{\lambda}{2}\norm{x}_2^2,
\end{equation}
where $\mA = [a_1,\ldots, a_n] \in \R^{d \times n},$ $y \in \R^n$ is the given data and $\lambda >0$ the regularization parameter.

\subsection{New non-uniform sampling using optimal probabilities}

First we compare non-uniform SAGA using the new optimized importance probabilities~\eqref{eq:optprobs} against using the probabilities 
$p_i = \left.L_i \right/\overline{L}$
 as suggested in~\cite{Schmidtnonuni}. When $n\mu$ is significantly smaller than $L_i$ for all $i$ then the two sampling are very similar. But when  $n\mu$ is relatively large, then the optimized probabilities~\eqref{eq:optprobs}  can be much closer to a uniform distribution as compared to using $p_i = \left.L_i \right/\overline{L}$. We illustrate this by solving a ridge regression problem~\eqref{eq:ridge}, using generated data such that
\begin{equation}
\mA^\top x = y +\epsilon,
\end{equation}
where the elements of $\mA$ and $x$ are sampled from the standard Gaussian distribution $\cN(0,1)$, and the elements of $\epsilon$ are sampled from $\cN(0,10^{-3})$. 
s It is not hard to see that the  smoothness constants $\{L_i\}$ are given by $L_i=\norm{a_i}_2^2+\lambda$ for $i\in [n]$. 
We scale the columns of $\mA$ so that $\norm{a_1}_2^2 =1$ and $\norm{a_i}_2^2 = \frac{1}{n^2},$ for $i=2,\ldots, n,$ and set the regularization parameter $\lambda = \frac{1}{n^2}.$ Consequently, $L_{\max}  = 1+ \frac{1}{n^2}$, $L_{i} = \frac{2}{n^2}$ for $i=1,\ldots, n$,
$\overline{L} = \frac{(n+1)^2 -1}{n^3}$ and $\mu = \tfrac{1}{n}\lambda_{\min}(\mA\mA^\top)+ \frac{1}{n^2}$.  In this case the iteration complexity of non-uniform SAGA with the optimal probabilities~\eqref{eq:itercomplexopt} is given by
\begin{equation} \label{eq:itercomplexoptexe}
\left(n +4\frac{(n+1)^2 -1}{\mu n^3}  \right)
 \log \left(\frac{1}{\epsilon} \right).\end{equation}  
The complexity~\eqref{eq:itercomplexLimark} which results from using the probabilities $p_i = \left.L_i \right/\overline{L}$ is given by
\begin{equation} \label{eq:itercomplexLimarkexe}
\frac{(n+1)^2 -1}{ n^3} \left(\frac{n^3}{2}+\frac{4}{\mu} \right)\log\left( \frac{1}{\epsilon}\right) .
\end{equation}
Now we consider the regime where $n \rightarrow \infty,$ in which case $\mu \rightarrow \cO(\frac{1}{n^2})$ and consequently~\eqref{eq:itercomplexoptexe}$\rightarrow \cO(n)\log\frac{1}{\epsilon}$   and in contrast~\eqref{eq:itercomplexLimarkexe} $\rightarrow \cO(n^2)\log\frac{1}{\epsilon}.$

Thus the iteration complexity~\eqref{eq:itercomplexLimarkexe} will grow quadratically while~\eqref{eq:itercomplexoptexe} grows linearly in $n$. We illustrate this in Figures~\ref{fig:SAGALi1},~\ref{fig:SAGALi2}  and~\ref{fig:SAGALi3} where we set $n=10$, $n=100$ and $n=1000$, respectively. In all figures we see that  \texttt{SAGA-opt} (SAGA with optimized probabilities) is the fastest method. On the other hand \texttt{SAGA-Li} (SAGA with $p_i =L_i/\overline{L}$) stalls in Figure~\ref{fig:SAGALi2} and~\ref{fig:SAGALi3} when $n$ is larger, performing even worst as compared to the standard SAGA method with uniform probabilities (\texttt{SAGA-uni}).

    \newcommand{\x}{0.30}
\begin{figure}
\centering
 \begin{subfigure}[t]{\x\textwidth}
        \centering
\includegraphics[width =  \textwidth ]{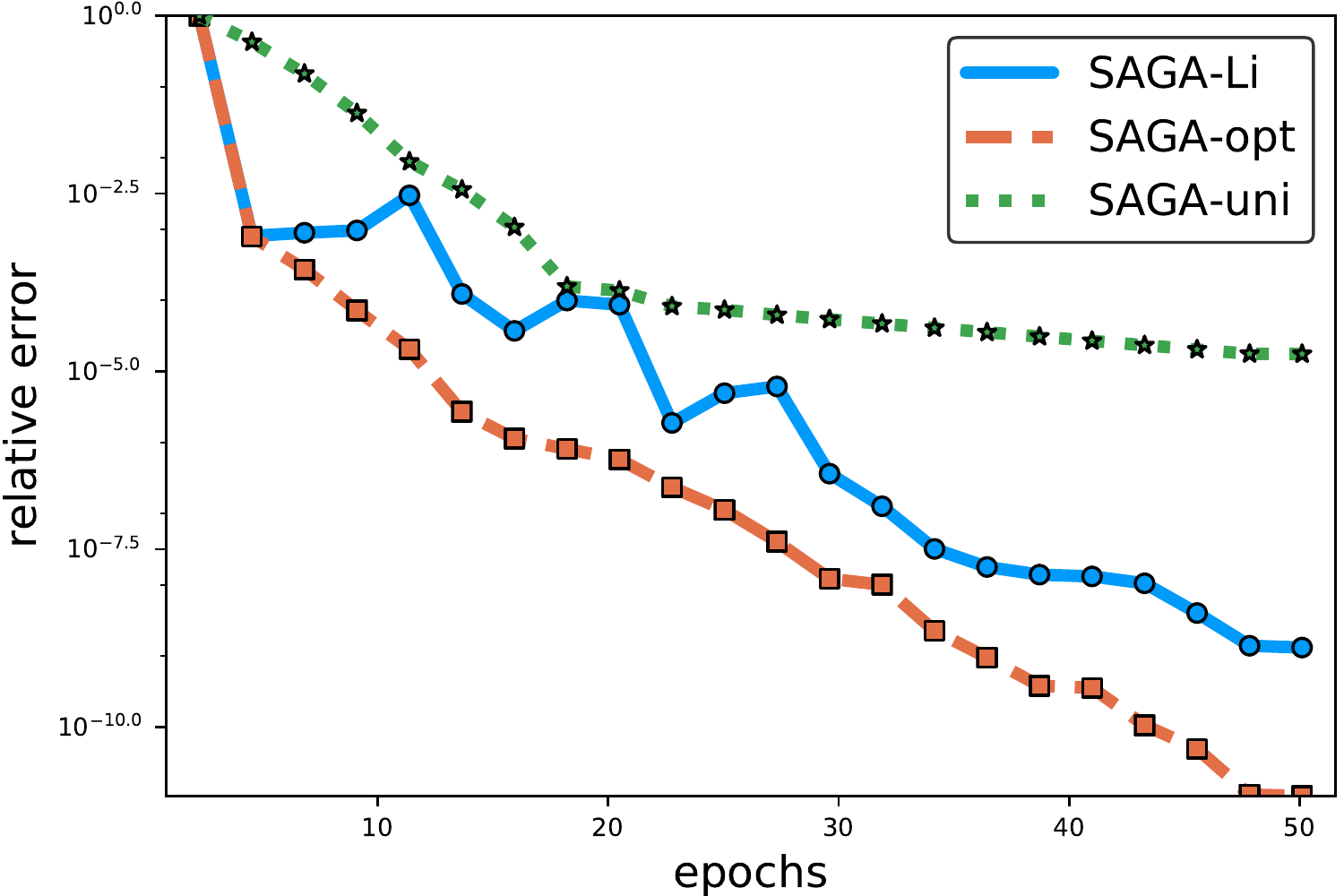}
       \caption{$n =10$ }\label{fig:SAGALi1}
    \end{subfigure}%
         \hspace{0.03\textwidth}
 \begin{subfigure}[t]{\x\textwidth}
        \centering
\includegraphics[width =  \textwidth ]{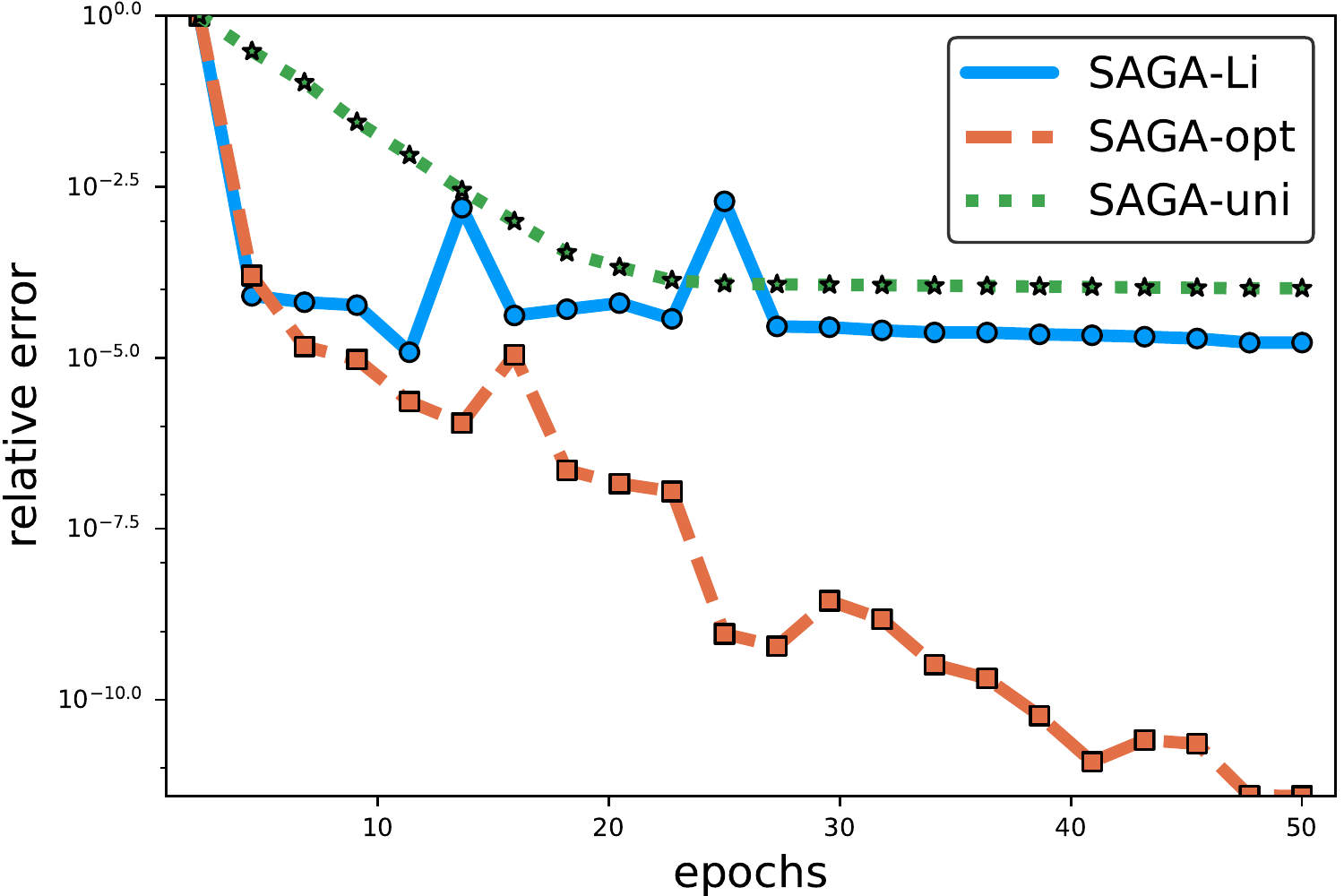}
       \caption{$n =100$ }\label{fig:SAGALi2}
    \end{subfigure}%
     \hspace{0.03\textwidth}
         \begin{subfigure}[t]{\x\textwidth}
                \centering
\includegraphics[width =  \textwidth]{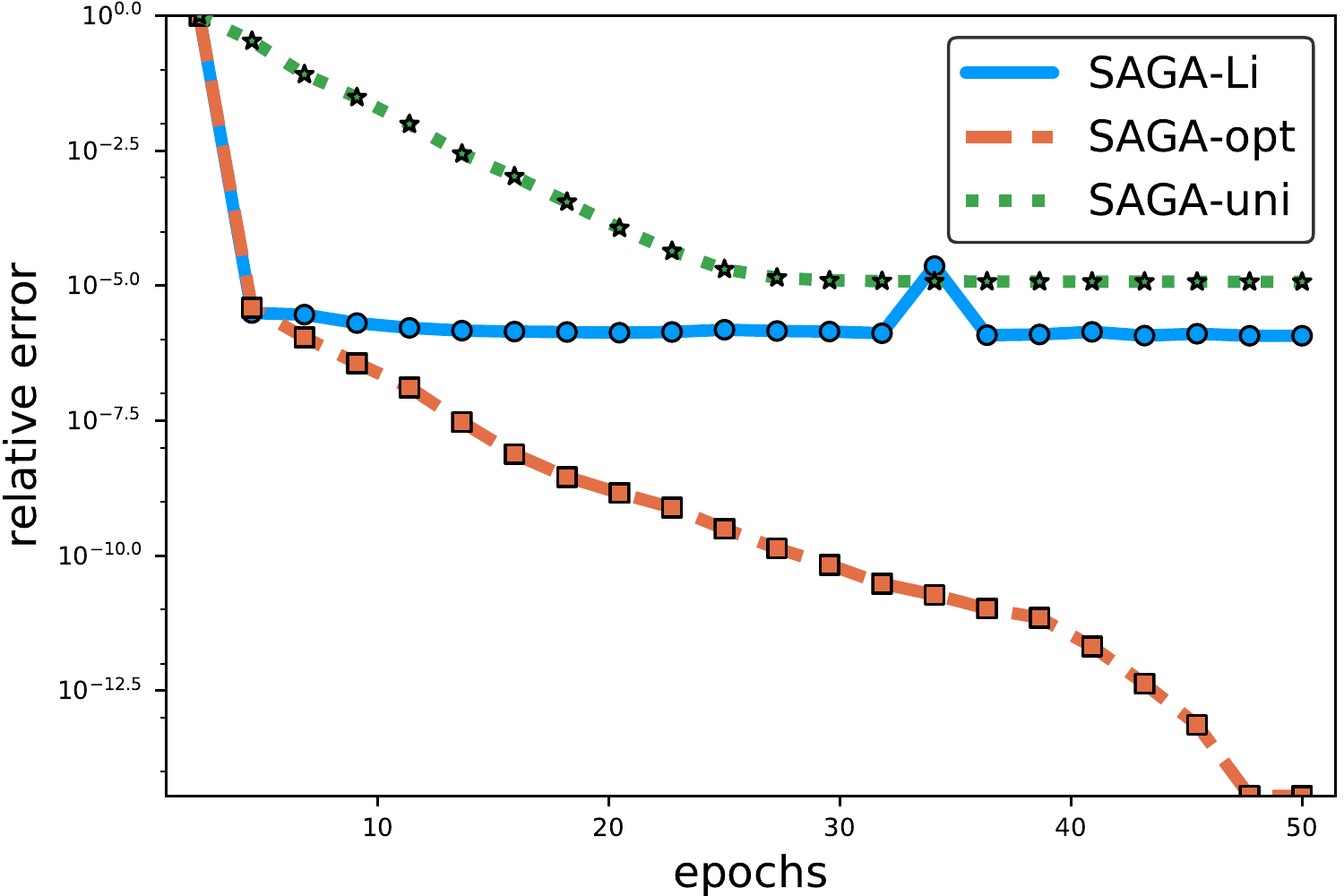}
        \caption{$n =1000$ }\label{fig:SAGALi3}
    \end{subfigure}%
        \caption{Comparing the performance of SAGA with importance sampling based on the optimized probabilities~\eqref{eq:optprobs} (\texttt{SAGA-opt}), $p_i = L_i/\overline{L}$ (\texttt{SAGA-Li}) and $p_i = 1/n$ (\texttt{SAGA-uni}) for an artificially constructed ridge regression problem as $n$ grows.}
\end{figure}

These experiments, together with our theoretical results, leads us to the following observation regarding data pre-processing and data scaling
\begin{remark}
A standard good practice for pre-processing in classification or regression problems is to scale the data so that the standard deviation of each feature equals one. Which in our setting is equivalent to scaling the rows of $\mA\mA^\top$ so that $\norm{\mA_{i:}}_2^2 =1$ for $i =1,\ldots, d.$ In contrast, the iteration complexity of SAGA indicates that one should scale the columns of $\mA \mA^\top$ so that $\norm{\mA_{:j}}_2^2= \norm{a_j}_2^2=1$ for $j =1,\ldots, n.$ Fortunately, both the columns and rows of $\mA \mA^\top$ can be simultaneously scaled using the Sinkhorn algorithm to solve the matrix scaling problem $\mA\mA^\top e = e$ and $\mA^\top \mA e = e.$ 
\end{remark}

\subsection{Optimal mini-batch size}

\label{subsec:optimalminibatch}

\begin{figure}
\centering
 \begin{subfigure}[t]{\x\textwidth}
        \centering
\includegraphics[width =  \textwidth ]{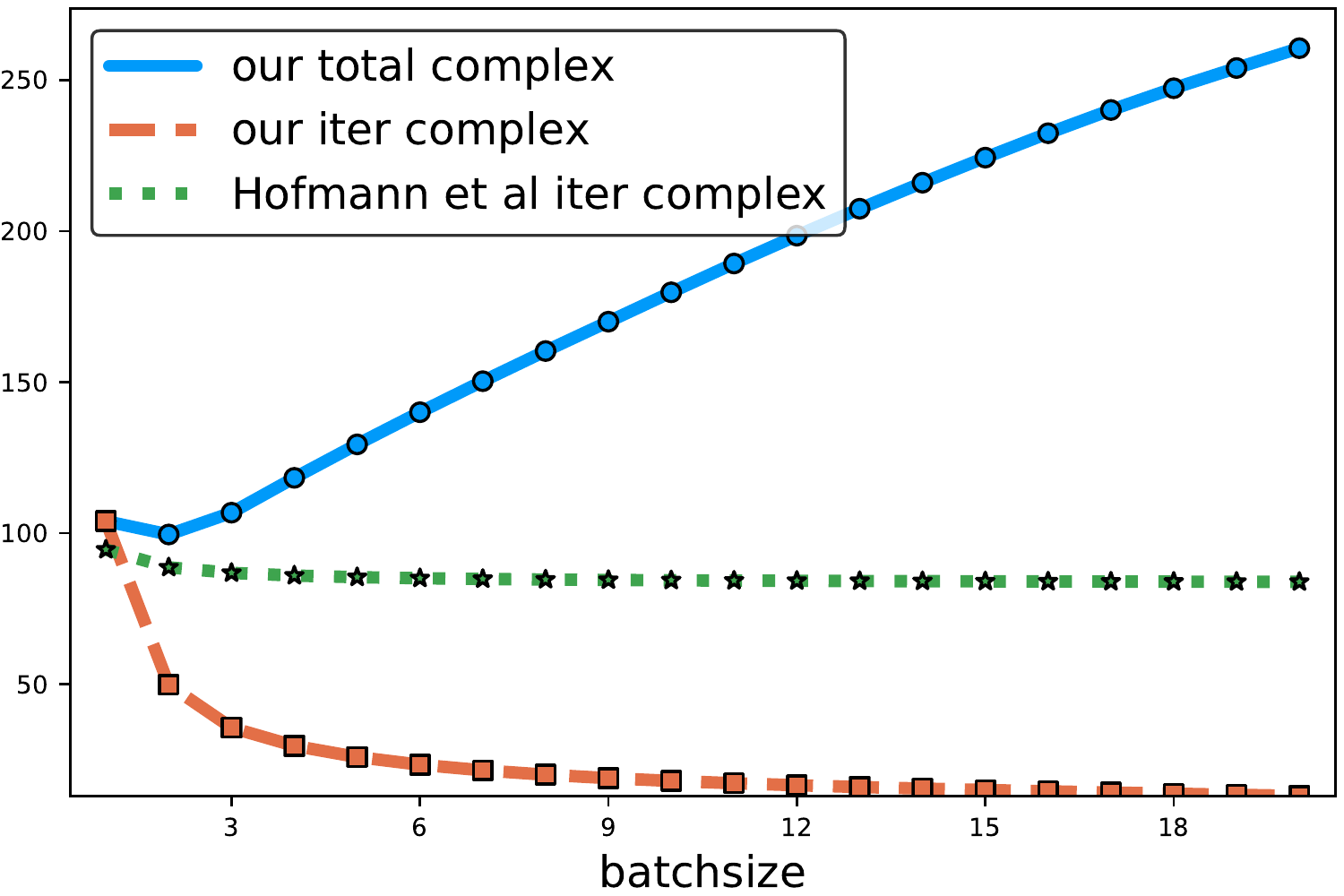}
        \caption{Randomly generated $20 \times 20$ gaussian matrix }\label{fig:comp1}
    \end{subfigure}%
     \hspace{0.03\textwidth}
         \begin{subfigure}[t]{\x\textwidth}
        \centering
\includegraphics[width =  \textwidth ]{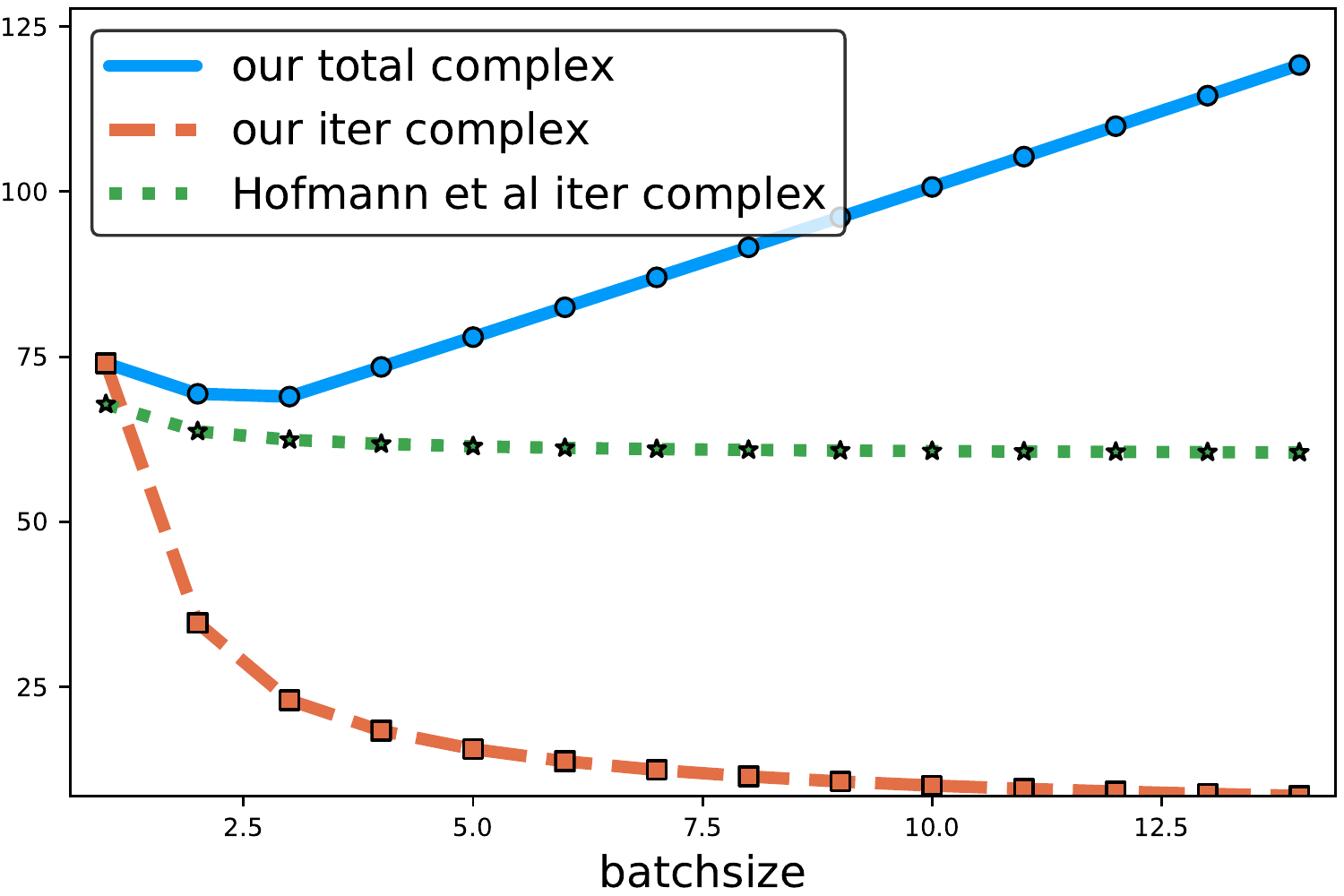}
        \caption{\texttt{australian} problem from LIBSVM~\cite{Chang2011} }\label{fig:comp2}
    \end{subfigure}%
     \hspace{0.03\textwidth}
         \begin{subfigure}[t]{\x\textwidth}
        \centering
\includegraphics[width =  \textwidth ]{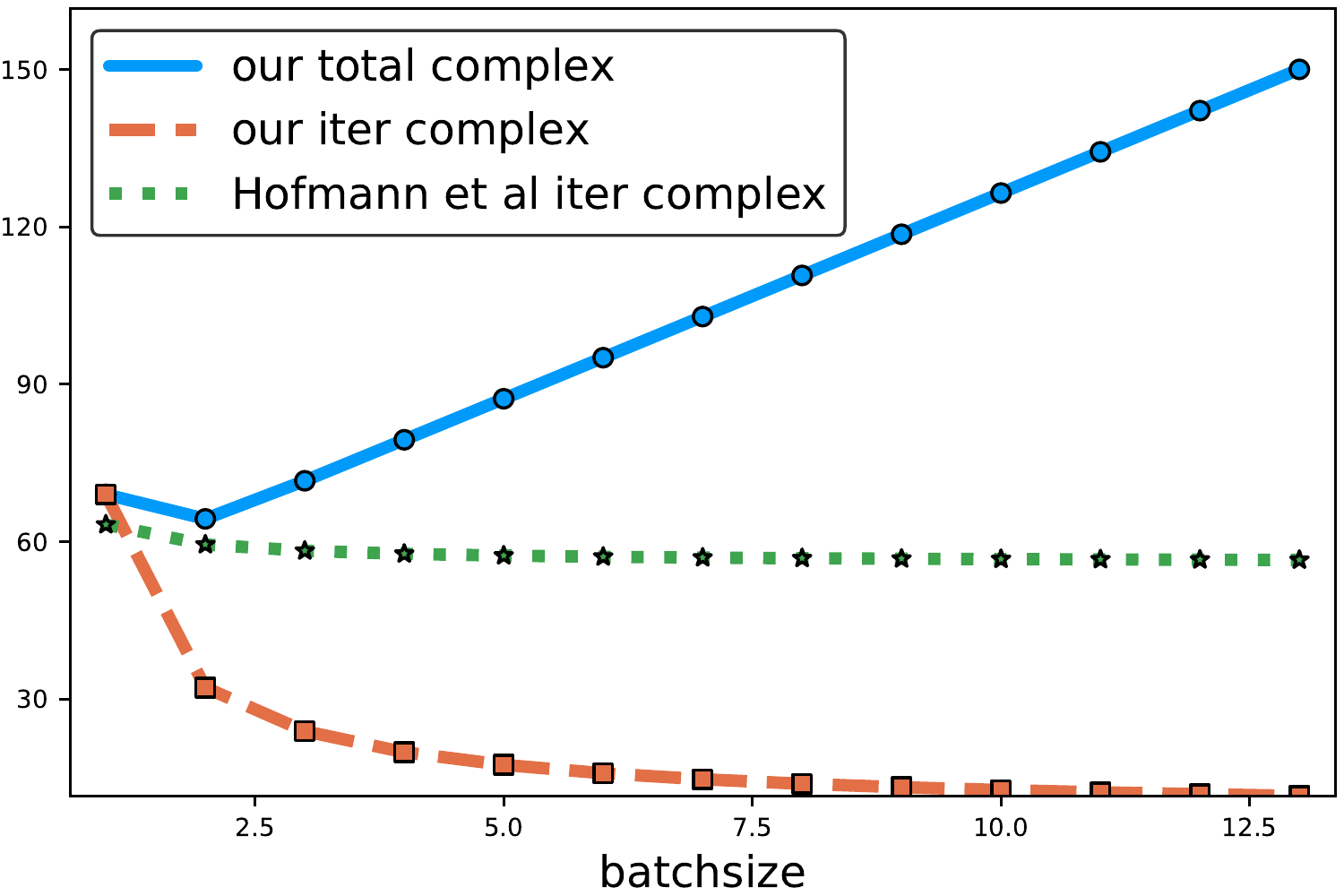}
        \caption{\texttt{heart} problem from LIBSVM~\cite{Chang2011} }\label{fig:comp3}
    \end{subfigure}%
        \caption{The iteration complexity of minibatch SAGA~\eqref{eq:SAGAmini-nice} vs the mini-batch size $\tau$ for two ridge regression problems~\eqref{eq:ridge}. We used $\lambda = L_{\max}/n.$}
\end{figure}

Our analysis of the mini-batch SAGA is precise enough as to inform an optimal mini-batch size. For instance, consider $\tau$--nice sampling and the resulting iteration complexity~\eqref{eq:nh09s8h09sh9JJ}. Theorem~\ref{lem:Lfhatuni} suggests that for any $\tau\in [n]$, the terms within the maximum in~\eqref{eq:nh09s8h09sh9JJ} are bounded by
\begin{eqnarray}
 L_{\max}  \quad \geq & \LGconst &  \geq \quad L \\
L_{\max} +\frac{\mu n}{4} \quad \geq & \displaystyle C(\tau) \,\,\eqdef\,\, \frac{1}{\tau}\frac{n-\tau}{n-1} L_{\max}     +\frac{\mu}{4}\frac{n}{\tau} & \geq  \quad \frac{\mu}{4}.
\end{eqnarray}

Moreover, the upper and lower bounds are realized for $\tau =1$ and $\tau =n$, respectively. Consequently, for $\tau$ small, we have  $\LGconst \geq C(\tau)$. On the other hand, for $\tau$ large we have  $\LGconst \leq C(\tau).$ Furthermore, $ C(\tau)$ decreases super-linearly in $\tau$ while $\LGconst$ tends to decrease more modestly. Consequently, the point where 
$\LGconst$ overtakes $C(\tau)$ is often the best for the overall complexity of the method. To better appreciate these observations,   we plot the  evolution of the iteration complexity~\eqref{eq:nh09s8h09sh9JJ}, the total complexity and the iteration complexity as predicted by Hofmann et al.~\cite{HofmannLLM152015} (see~\eqref{eq:hofcomplex}) as $\tau$ increases in Figures~\ref{fig:comp1}, \ref{fig:comp2} and~\ref{fig:comp3} for three different linear least squares problems. Since each step of mini-batch SAGA computes $\tau$ stochastic gradients, s the total complexity is $\tau$ times the iteration complexity. In each figure we can see that our iteration complexity initially decreases super-linearly, then at some point the complexity is dominated by  $\LGconst$ and the iteration complexity decreases sublinearly. Up to this point we can observe an improvement in overall total complexity. This is in contrast to the iteration complexity given by Hofmann et al. that shows practically no improvement as $\tau$ increases.

 Though our analysis predicts only modest improvements in total complexity, and suggests that $\tau =2$ or $\tau =3$ is optimal, we must bear in mind that this corresponds to $10\%$ and $20\%$ of the data for these small dimensional problems. We conjecture that for larger  problems, this improvement in total complexity will also be larger. 
 
 \begin{figure}
\centering
\begin{subfigure}[t]{\textwidth}
        \centering
\includegraphics[width = 0.4\textwidth ]{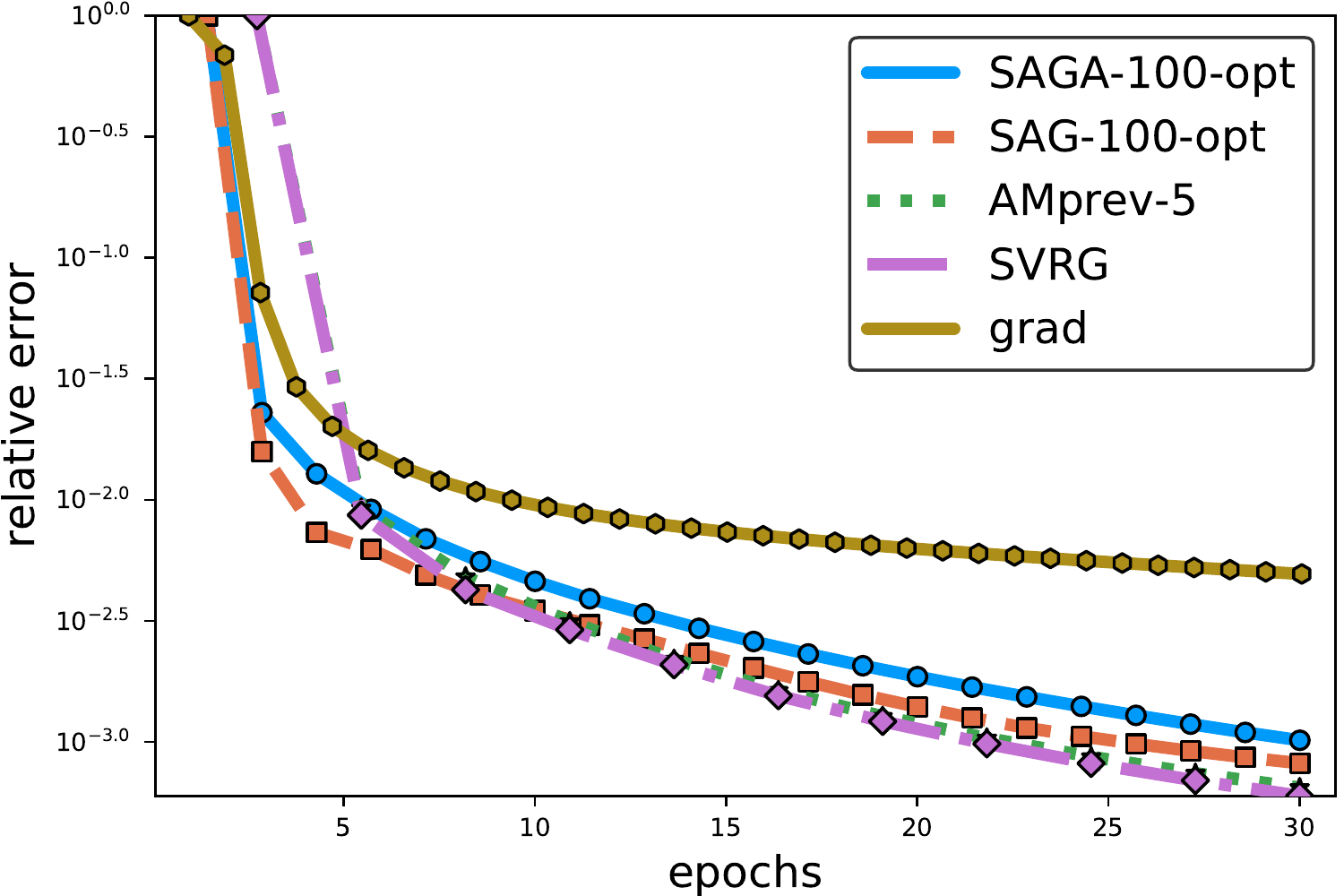}
\hspace{0.03\textwidth}
\includegraphics[width =  0.4\textwidth ]{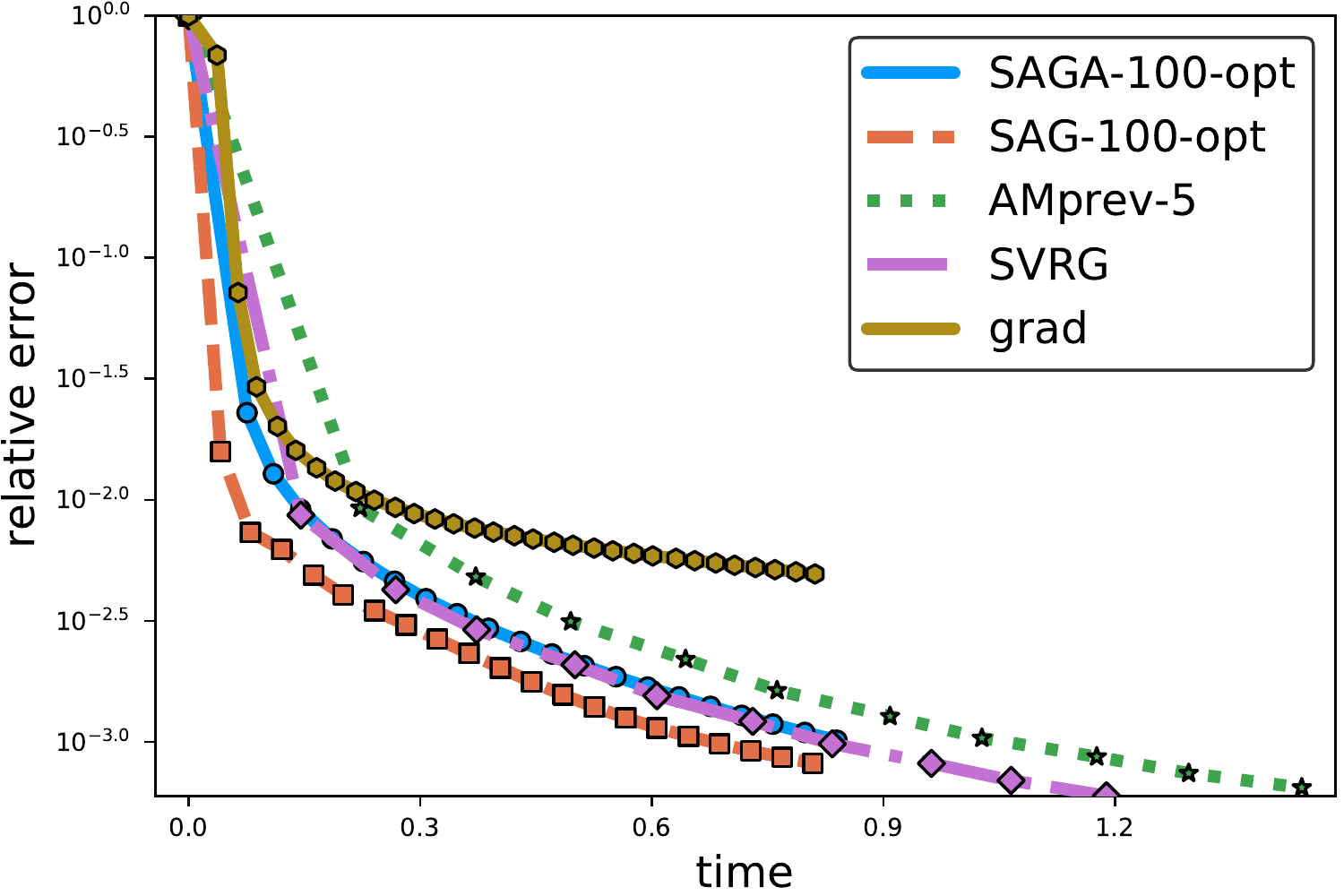}
        \caption{ \texttt{mushrooms} }\label{fig:comp1}
\end{subfigure}\\
\begin{subfigure}[t]{\textwidth}
        \centering
\includegraphics[width =  0.4\textwidth ]{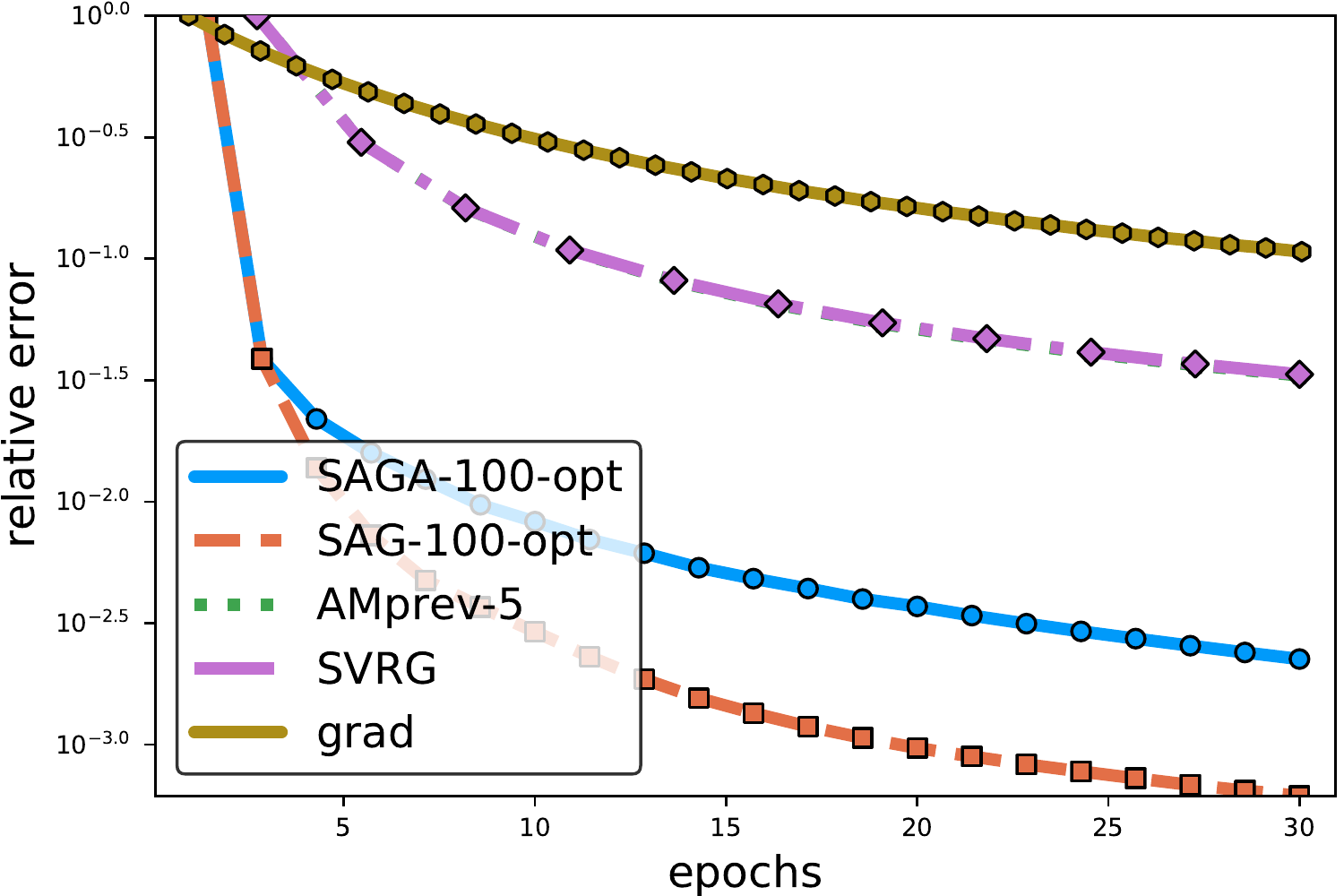}
\hspace{0.03\textwidth}
\includegraphics[width =  0.4\textwidth ]{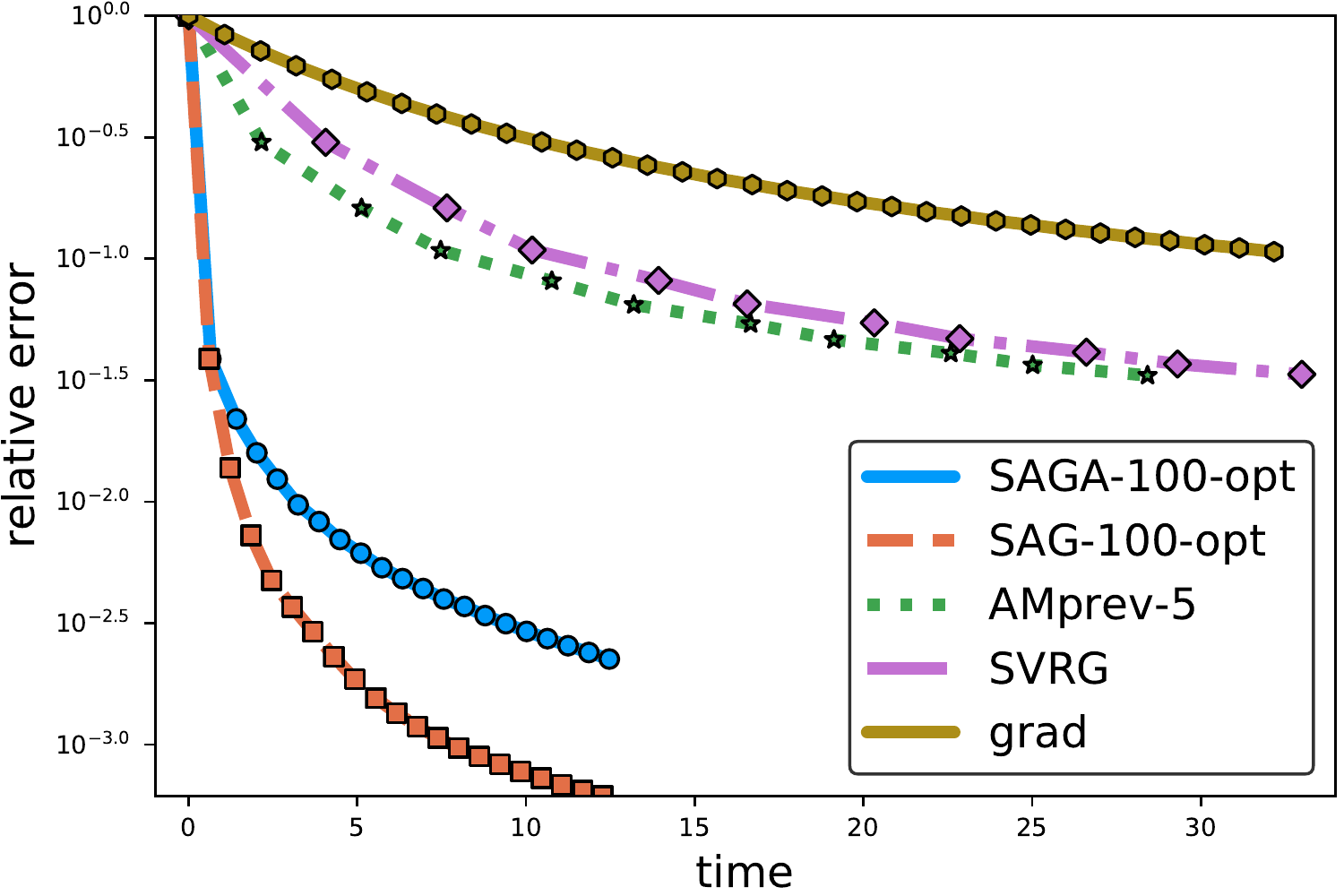}
        \caption{ \texttt{w8a} }\label{fig:comp2}
\end{subfigure}\\
\begin{subfigure}[t]{\textwidth}
        \centering
\includegraphics[width =  0.4\textwidth ]{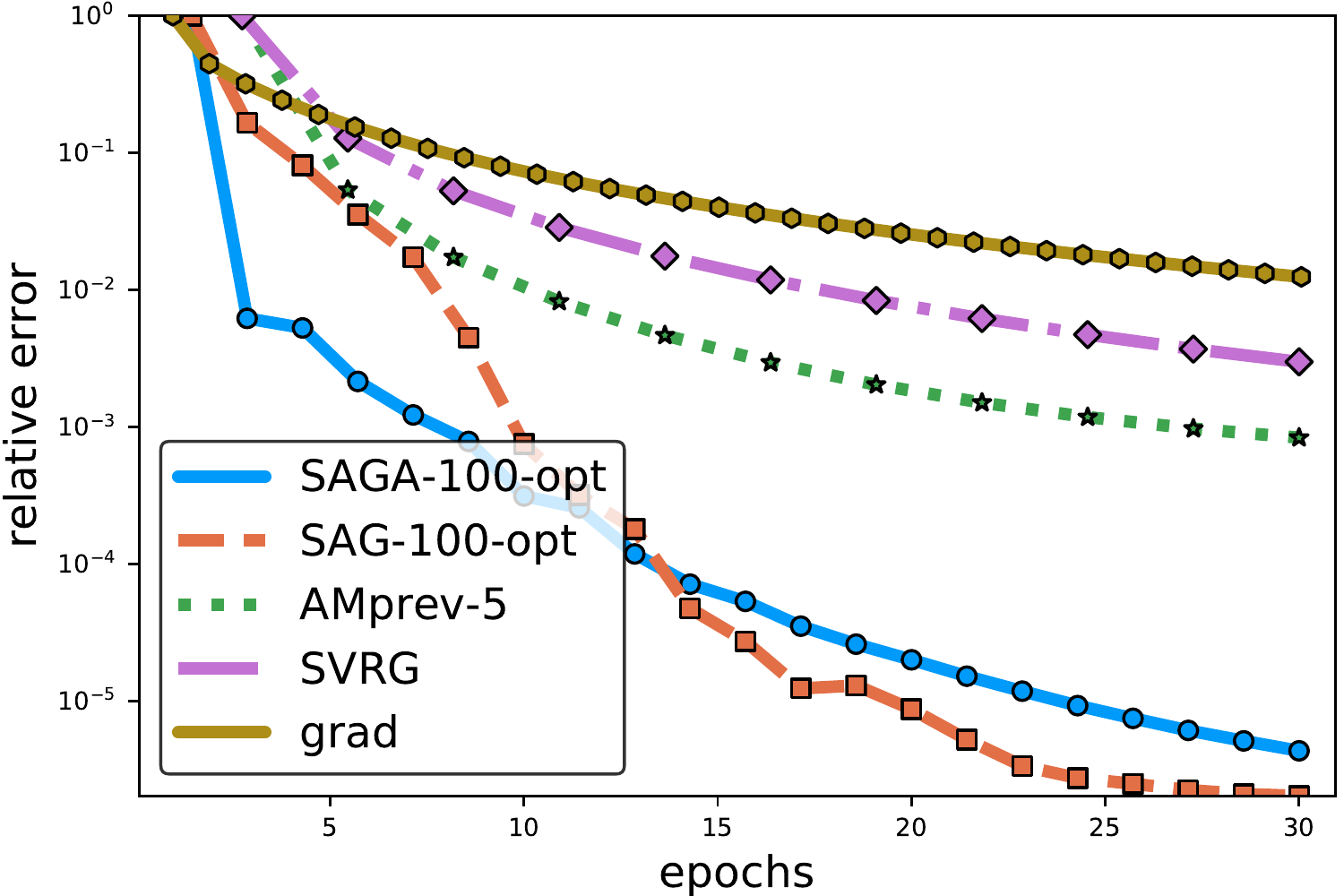}
\hspace{0.03\textwidth}
\includegraphics[width =  0.4\textwidth ]{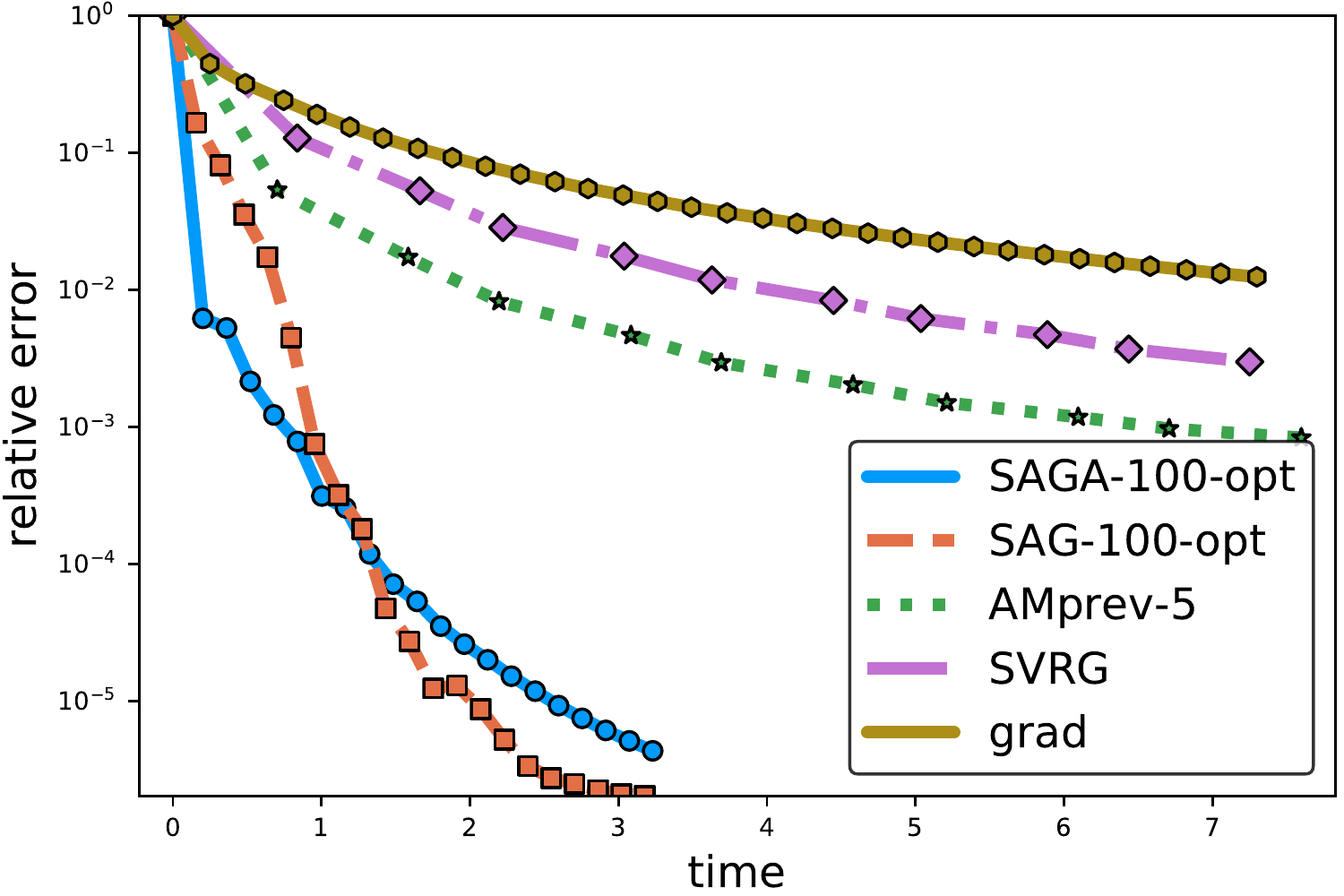}
        \caption{\texttt{a9a}   }\label{fig:comp3}
\end{subfigure}%
        \caption{Comparison of the methods on logistic regression problems~\eqref{eq:logistic} with data taken from LIBSVM~\cite{Chang2011}.}
\end{figure}
 
To use these insights in practice, we need to be able to efficiently determine the $\tau$ which corresponds  to the point at which the convergence regimes switches from being dominated by $C(\tau)$ to being dominated by $\LGconst$.
 This surmounts to choosing $\tau$ so that
\[  \LGconst  =
 \frac{1}{\tau}\frac{n-\tau}{n-1} L_{\max}     +\frac{\mu}{4}\frac{n}{\tau}. \]
Estimating $L_{\max}$ and $\mu$ is often possible, but  the cost of computing $\LGconst$ has a combinatorial dependency on  $n$ and $\tau.$ Thus to have a practical way of choosing $\tau$, we first need to bound  $\LGconst$. This can be done for losses with linear classifiers using concentration bounds. We leave this for future work.

\subsection{Comparative experiments}

We now compare the performance of \texttt{SAGA-opt} to several known methods such as \texttt{SVRG}~\cite{Johnson2013}, \texttt{grad} (gradient descent with fixed stepsizes) and \texttt{AMprev} (an improved version of \texttt{SVRG} that uses second order information)~\cite{bach-roux-gower2018}.  For the stepsize of \texttt{SAGA-opt} and \texttt{SAG-opt}, we found the stepsize $\alpha  \leq \frac{1}{n\mu + 4\bar{L}} $  given by theory to be a bit too conservative. Instead do we away with the $4$ and used $\alpha = \frac{1}{n\mu + \bar{L}} $ instead. For the remaining methods we used a grid search over $L_{\max} \times  2^m$ for $m = 21, 19, 17, \ldots, -10,-11.$

To illustrate how biased gradient estimates can perform well in practice (despite lack of proper theoretical understanding of these methods), we also test \texttt{SAG-opt}: a method that  uses the same Jacobian updates as \texttt{SAGA-opt},  but instead uses the biased gradient estimate $g^k = \frac{1}{n}\mJ^{k+1} \ones$. See Section~\ref{sec:sketchproject} for more details on biased gradient estimates.

In Figures~\ref{fig:comp1}, \ref{fig:comp2} and~\ref{fig:comp3} we compare the methods on three logistic regression problems~\eqref{eq:logistic} based on three different data sets taken from LIBSVM~\cite{Chang2011}. In all these problems the two methods with optimized non-uniform sampling \texttt{SAG-opt} and \texttt{SAGA-opt} were faster in terms of both epochs and time. The next best method was \texttt{AM-prev}, followed by \texttt{SVRG} and \texttt{grad}. It is interesting to see how well \texttt{SAG-opt} performs in practice, despite having biased gradient estimates. This is why we believe it is important to advance the analyse of biased gradient estimates as future work.

\section{Conclusion} \label{sec:conclusion}

We now provide a brief summary of some of the key contributions of this paper and a few selected pointers to possible future research directions.

\subsection{Summary of key contributions}

We developed and analyzed JacSketch---a novel family of variance reduced methods based on Jacobian sketching---and provided a link between variance reduction for empirical risk minimization and  recent results from the field of randomized numerical linear algebra on sketch-and-project type methods for solving linear systems. In particular, it turns out that variance reduction is obtained by taking an  SGD step on a stochastic optimization problem whose solution is the unknown Jacobian.  As a consequence of our analysis, we resolved the conjecture of \cite{Schmidtnonuni}  in the affirmative by proving a properly designed importance sampling for SAGA leading to the iteration complexity of $\cO(n + \tfrac{\bar{L}}{\mu}  ) \log\left(\frac{1}{\epsilon} \right)$. For this purpose we developed a new proof technique using a {\em stochastic Lyapunov function}. Our complexity result for uniform  mini-batch SAGA perfectly interpolates between the best known convergence rates of SAGA and gradient descent, and is  sufficiently precise as to inform the choice of the batch size that minimizes the over all complexity of the method. Additionally we design and analyse a 
 {\em reduced memory} variant of SAGA as a special case.
%

\subsection{Future work}

For future work we see many possible avenues including the following.

\paragraph{Structured sparse weight matrices.} One may wish to explore combinations of a weight matrix and different sketches to  design new efficient methods further improving iteration complexity. For this the weighting matrix will have to be highly structured (e.g., block diagonal or very sparse) so that the Jacobian update~\eqref{eq:jacobsol} can be  computed efficiently.


\paragraph{Bias-variance trade-off.} One can try to explore the bias-variance trade-off as opposed to merely focus on the extremes only: SAG (minimum variance) and SAGA (no bias). There is also no empirical evidence that unbiased estimators outperform the biased ones.

\paragraph{Johnson-Lindenstrauss sketches.}
One can design completely new methods using different sparse sketches, such as the fast Johnson-Lindenstrauss transform~\cite{Ailon:2009} or the Achlioptas transform~\cite{Achlioptas:2003}. The resulting method can then be analyzed through Theorem~\ref{theo:convgen}. But first these sketches need to be adapted to ensure we get an efficient method. In particular, computing $\Jac(x) \mS$ is only efficient if $\mS$ is row sparse, i.e., most of the rows of $\mS$ contain zeros only.

{\printbibliography}

\appendix 

\section{Proof of Inequality \eqref{eq:inew_intro}}

\begin{lemma} \label{lem:two_ineq} Let $S$ be a sampling whose support $\cG=\support(S)$ is  a partition of $[n]$. Moreover, assume all sets of this partition have cardinality $\tau$. Then
\[\frac{1}{|\cG|}\sum_{C\in \cG} L_C \quad \leq  \quad \bar{L} \quad \leq \quad \max_{C\in \cG} \frac{1}{\tau} \sum_{i\in C} L_i.\]
\end{lemma}
\begin{proof}
 By assumption, $|\cG|= \tfrac{n}{\tau}$. The first inequality follows from
$
\sum_{C \in \cG} L_C \leq  \sum_{C\in \cG} \frac{1}{\tau} \sum_{i\in C} L_i
 = \frac{1}{\tau} \sum_{i=1}^n L_i  = \frac{n}{\tau} \bar{L}. 
$
On the other hand,
\[\bar{L} = \frac{1}{n} \sum_{i=1}^n L_i = \frac{1}{n} \sum_{C\in \cG} \sum_{i\in C} L_i = \frac{1}{|\cG|} \sum_{C\in \cG} \frac{1}{\tau } \sum_{i\in C} L_i \leq \max_{C\in \cG}\frac{1}{\tau } \sum_{i\in C} L_i.\]
\hfill \qed
\end{proof}

\newpage
\section{Duality of Sketch-and-Project and Constrain-and-Approximate} 

\begin{lemma}\label{lem:sol}
Let $\mJ^k,\Jac \in \R^{d \times n}$ and $\mS \in \R^{n \times \tau}.$
The sketch-and-project problem
\begin{equation}\label{eq:jacupdateFW} \mJ^{k+1}= \arg\min_{\mJ\in \R^{d \times n}} \frac{1}{2}\norm{\mJ-\mJ^{k}}_{\mW^{-1}}^2 \quad \mbox{subject to} \quad  \Jac \mS = \mJ \mS,\end{equation}
and the constrain-and-approximate problem
\begin{equation}\label{eq:jacupdatedual} \mJ^{k+1}=  \underset{\mJ \in \R^{d \times n}}{\arg} \underset{\mY\in \R^{d \times \tau}}{\min}  \frac{1}{2}\norm{\mJ-\Jac}_{\mW^{-1}}^2 \quad \mbox{subject to} \quad  \mJ = \mJ^{k}+ \mY \mS^\top \mW, \end{equation}
have the same solution, given by:
\begin{equation}\label{eq:jacobsolW}
\mJ^{k+1} = \mJ^{k} -(\mJ^{k}-\Jac) \mS (\mS^\top \mW \mS)^{\dagger} \mS^\top \mW.
\end{equation}
\end{lemma}

\begin{proof}
Let $\mZ = (\mJ-\mJ^{k}) \mW^{-1/2}$ so that~\eqref{eq:jacupdateFW} becomes
\begin{equation}\label{eq:aj8j8sdj} \mJ^{k+1}= \arg\min_{\mJ\in \R^{d \times n}} \frac{1}{2}\norm{\mZ}_{\mI}^2 \quad \mbox{subject to} \quad  \mZ \mW^{1/2} \mS = (\Jac   -\mJ^{k}) \mS.\end{equation}
It follows from one of the properties of pseudoinverse\footnote{The least norm solution to $\mA \mX =\mB$ is given by $\mX = \mA^\dagger \mB$.} that the least norm solution of the above is given by
$\mZ = (\Jac   -\mJ^{k}) \mS (\mW^{1/2} \mS)^\dagger.$
Substituting $\mZ = (\mJ -\mJ^{k}) \mW^{-1/2}$, multiplying on the right by $\mW^{1/2}$ gives
\begin{equation} \label{eq:japseisd}
 \mJ  = \mJ^{k} + (\Jac   -\mJ^{k}) \mS (\mW^{1/2} \mS)^\dagger \mW^{1/2}.
\end{equation}
Now it remains to use  another  pseudoinverse property:  $\mA^\dagger = (\mA^\top \mA)^\dagger \mA^\top$. We use it in \eqref{eq:japseisd}  with $ \mA = \mW^{1/2} \mS$,  which gives~\eqref{eq:jacobsolW}. Next we show using duality that~\eqref{eq:jacupdatedual} is equivalent to~\eqref{eq:jacupdateFW}. Consider the Lagrangian of~\eqref{eq:jacupdateFW}, namely
\begin{equation}
\label{eq:Lagjacsdfs}
L(\mJ,\mY) \eqdef \frac{1}{2}\norm{\mJ-\mJ^{k}}_{\mW^{-1}}^2 +\dotprod{ \mY, (\Jac -\mJ) \mS }
 = \frac{1}{2}\norm{\mJ-\mJ^{k}}_{\mW^{-1}}^2 +\dotprod{ \mY \mS^\top , \Jac - \mJ  }.
\end{equation}
By strong duality we know that 
$\eqref{eq:jacupdateFW} =\min_{\mJ \in \R^{d \times n}} \max_{\mY \in \R^{d \times \tau}} L(\mJ, \mY) =  
\max_{\mY \in \R^{d \times \tau}} \min_{\mJ \in \R^{d \times n}}  L(\mJ, \mY).
$ Now we will show that $\max_{\mY \in \R^{d \times \tau}} \min_{\mJ \in \R^{d \times n}}  L(\mJ, \mY) = \eqref{eq:jacupdatedual}.$
Differentiating $L(\mJ, \mY)$ in $\mJ$ and setting it to zero,
\begin{equation}\label{eq:Lagjacdifffs3x} 
\mY \mS^\top  = (\mJ - \mJ^{k}) \mW^{-1}. \end{equation}
Substituting~\eqref{eq:Lagjacdifffs3x} into~\eqref{eq:Lagjacsdfs} gives
\begin{eqnarray*}
L(\mJ, \mY) & =& \frac{1}{2}\norm{\mJ-\mJ^{k}}_{\mW^{-1}}^2 +\dotprod{ (\mJ-\mJ^{k}) \mW^{-1}, \Jac + \mJ^{k} - \mJ^k - \mJ  } \\
&= & -\frac{1}{2}\norm{\mJ-\mJ^{k}}_{\mW^{-1}}^2 -\dotprod{ (\mJ-\mJ^{k}) \mW^{-1}, \mJ^{k}- \Jac}.
\end{eqnarray*}
Adding and subtracting to the right hand side $\frac{1}{2}\norm{\mJ^{k} -\Jac}_{\mW^{-1}}^2$ and completing the square gives
\[L(\mJ, \mY)  = -\frac{1}{2}\norm{\mJ-\mJ^{k}+(\mJ^{k} -\Jac) }_{\mW^{-1}}^2 +\frac{1}{2}\norm{\mJ^{k} -\Jac}_{\mW^{-1}}^2\\
=  -\frac{1}{2}\norm{\mJ -\Jac }_{\mW^{-1}}^2 +\frac{1}{2}\norm{\mJ^{k} -\Jac}_{\mW^{-1}}^2.
\] Keeping in mind the constraint~\eqref{eq:Lagjacdifffs3x},  maximizing the above over $\mY$  gives~\eqref{eq:jacupdatedual}. \hfill \qed
\end{proof}

\section{Proof of Theorem~\ref{thm:rhobnbcollect}} \label{app:rho}

First we will establish that
	\begin{equation} \label{eq:asdo81n7}
	\frac{|\cal G|}{c_1^2}\sum_{C \in \cG} e_C e_C^\top \mW
	=\frac{|\cG|c_2}{c_1^2}\left(
	\begin{array}{@{}ccccc@{}}
	\frac{c_1}{c_2} w_1 & w_2 & \cdots & \frac{}{}w_{n-1} & w_{n}\\
	w_{1} & \frac{c_1}{c_2}w_2 & \cdots &  w_{n-1} &  w_{n} \\
	\vdots & & \ddots &  & \vdots \\
	w_{1} & \cdots &   & \frac{c_1}{c_2}w_{n-1} &  w_{n}  \\
	w_{1} &  w_{2} & \cdots & w_{n-1} & \frac{c_1}{c_2}w_n
	\end{array}
	\right).
	\end{equation}
	Indeed,  for every $i $ we have that
	$e_i^\top \frac{|\cal G|}{c_1^2} \left(\sum_{C \in \cG}  e_C e_C^\top \mW \right)e_i= w_i \frac{|\cal G|}{c_1^2} \sum_{C \in \cG\, : \, i \in C} 1 = w_i \frac{|\cal G|}{c_1},$
	and for every $i \neq j$ we have 
	$e_i^\top\frac{|\cal G|}{c_1^2} \left(\sum_{C \in \cG}  e_C e_C^\top \mW \right) e_j= w_j \frac{|\cal G|}{c_1^2}\sum_{C \in \cG\, : \, i,j \in C} 1 = w_j \frac{| \cG | c_2}{c_1^2}.$	Using~\eqref{eq:asdo81n7},~\eqref{eq:theorhoXX2} and the Gershgorin circle theorem to bound $\rho$ from above we get
$
	\rho \leq  \max_i \left\{ \left( \frac{|\cal G|}{c_1}  -1\right) w_i+\sum_{i\neq j} w_j \left|\frac{| \cG | c_2}{c_1^2}-1\right|\right\},
$
as claimed.	When $\mW=\mI$ we can get tighter results by using that $\left(\frac{|\cal G|}{c_1^2}\sum_{C \in \cG} e_C e_C^\top -\ones \ones^\top\right)$ is a 
	\emph{circulant matrix} with associated vector 
$
	v =  \left(\frac{| \cG | }{c_1}-1 ,\frac{| \cG | c_2}{c_1^2}-1, \ldots, \frac{| \cG | c_2}{c_1^2}-1\right) \in \R^n.
$ There is an elegant formula for calculating eigenvalues $\lambda_j$ of circulant matrices~\cite{varga1954} using $v$, given by
	\begin{equation}\label{eq:radeeigs}
	\lambda_{j} = v_1+ \sum_{k=1}^{n-1} \omega_j^{k} v_{n-k+1} = \frac{| \cG | }{c_1}-1 +\left( \frac{| \cG | c_2}{c_1^2}-1\right)\sum_{k=1}^{n-1} \omega_j^{k}, \quad \mbox{for }j=0,\ldots, n-1,
	\end{equation}
	where $\omega_j =e^{\frac{2 \pi \iu j}{n} }$ are the $n$-th roots of unity and $\iu$ is the imaginary number. From~\eqref{eq:radeeigs} we see that there are only two distinct eigenvalues. Namely, for  $j =0$ we have
	\begin{equation*}
	\label{eq:lambda0circ} \lambda_0  \overset{\eqref{eq:radeeigs}}{=} \frac{| \cG | }{c_1}-1 +\left( \frac{ | \cG | c_2}{c_1^2}-1\right)(n-1) = \frac{| \cG | }{c_1}\left(1+ (n-1) \frac{c_2}{c_1}\right) -n.
	\end{equation*}
	The other eigenvalue is given by any $j \neq 0$ since
	\begin{equation*}
	\label{eq:lambdajcirc} 
	\lambda_j  \overset{\eqref{eq:radeeigs}}{=} \frac{| \cG | }{c_1}-1  -\left( \frac{| \cG | c_2}{c_1^2}-1\right)+ \left( \frac{| \cG | c_2}{c_1^2}-1\right)\underbrace{\sum_{k=0}^{n-1} \omega_j^{k}}_{=0} =\frac{| \cG | }{c_1}\left(1  - \frac{c_2}{c_1}\right).
	\end{equation*}

 \clearpage
\section{Notation Glossary} \label{sec:notation_glossary}

{\footnotesize

\begin{table}[!h]
\begin{center}
\begin{tabular}{|c|l|c|}
 \hline
 $f(x)$ & $\tfrac{1}{n}\sum_{i=1}^n f_i(x)$ (convex loss function $f:\R^d\to \R$) & \eqref{eq:prob}\\
  $x^*$ & minimizer of $f$ & \eqref{eq:prob}\\ 
  $\mu$ & strong convexity constant of $f$ & Tab~\ref{tbl:complexity_summary} \& Assum~\ref{ass:str_convex_at_solution} \& Thm~\ref{theo:convpart}\\
  $\alpha$ & stepsize & \eqref{eq:xupdate} \\
$g^k$ & stochastic estimator of $\nabla f(x^k)$  & \eqref{eq:xupdate}, \eqref{eq:g^k=intro}, \eqref{eq:g^k-SAGA}, \eqref{eq:unbiasedgrad} \\
 $[n]$ & $\{1,2,\dots,n\}$ & \\
 $F(x)$ & $(f_1(x),\ldots, f_n(x))^\top\in \R^n$ (function $F:\R^d\to \R^n$) & \eqref{eq:F(x)} \\
  $\Jac(x)$ & $[\nabla f_1(x),\ldots, \nabla f_n(x)] \in \R^{d\times n}$ (Jacobian of $F$ at $x$) & \eqref{eq:jac_def}\\
  $\ones$ & $(1,1,\dots,1)^\top \in \R^n$  (vector of all ones) &
\eqref{eq:ubif98gf8}   \\  
  $f^*$ / $f^k$ & shorthand for $f(x^*)$ / $f(x^k)$ & \\
  $\mW$ & $n\times n$ symmetric positive definite ``weight'' matrix &  \eqref{eq:fro_norm}, \eqref{eq:PSdef} \\
$\norm{\mX}_{\mW^{-1}}$ & $(\Tr{\mX \mW^{-1} \mX^\top})^{1/2}$  (weighted Frobenius norm) & \eqref{eq:fro_norm}\\
  $\mS$ & a random (sketching) $n\times \tau$ matrix picked from $\cD$ & \\
  $\Proj_\mS$ & $\mS (\mS^\top \mW \mS)^{\dagger} \mS^\top  \mW$  (stochastic projection matrix) & \\
  $\theta_{\mS}$ & bias-correcting random variable &  \eqref{eq:bias-corr-eq-intro} \& Assum~\ref{def:sketch}\\
$\ED{\cdot}$ & $ \EE{\mS \sim \cD}{\cdot}$ (expectation over $\mS\sim \cD$) &  \\
$S$ or $S_k$  & sampling (a random subset of $[n]$) & \\
$\tau$ & $\E{|S|}$  (minibatch size) & \\
$C$  & subset of $[n]$ & \\
$e_C$ & $\sum_{i\in C} e_i$ ($e_i$ is the $i$th unit coordinate vector in $\R^d$) & \\
$p_C$ / $p_i$   & $\Prb{S=C}$ / $\Prb{i\in S}$ & Sec \ref{sec:SAGA-intro}, \ref{sec:minibatch_sketches}\\
$\mI_C$ & column submatrix of $\mI$ with columns indexed by $C$  &  Sec~\ref{sec:minibatch_sketches} \& Thm~\ref{theo:convpart} \\
$\cG=\support(S)$ & $\{C \subseteq [n] \;:\; p_C>0\}$  (support of sampling $S$) & Sec~\ref{sec:minibatch_sketches}\\
$f_C$ & $\frac{1}{|C|}\sum_{i \in C} f_i$  (subsampled loss function) & Sec~\ref{sec:minibatch_sketches} \& Thm~\ref{theo:convpart}, \ref{lem:two_ineqXXX} \\
  $L_C$ & smoothness constant of $f_C$ & Sec~\ref{sec:intro-summary-of-results}, \ref{sec:98hs98hs8gh} \& Thm~\ref{theo:convpart}, \ref{lem:two_ineqXXX}\\
  $L_i$ & smoothness constant of $f_i$ & Sec~\ref{sec:intro-summary-of-results}, \ref{sec:98hs98hs8gh}\\
 $L_{\max}$ & $\max_i L_i$ & Sec~\ref{sec:intro-summary-of-results}, \ref{sec:98hs98hs8gh} \& Thm~\ref{lem:two_ineqXXX}\\
$L$ & smoothness constant of $f = \frac{1}{n}\sum_i f_i$ & Sec~\ref{sec:intro-summary-of-results}, \ref{sec:98hs98hs8gh} \& Thm~\ref{lem:two_ineqXXX} \\
$\bar{L}$ & $  \tfrac{1}{n}\sum_{i=1}L_i$  & Sec~\ref{sec:intro-summary-of-results}, \ref{sec:98hs98hs8gh} \& Thm~\ref{lem:two_ineqXXX}\\
 $\cL_1$ & Expected smoothness constant of the stochastic gradient & Assum~\ref{ass:ES1} \& Thm~\ref{theo:convgen} \\
 $\cL_2$ &  Expected smoothness constant of the Jacobian & Assum~\ref{ass:ES2} \& Thm~\ref{theo:convgen} \\ 
 $\LGi$ & $ \tfrac{1}{c_1}\sum_{C \; :\; C \in \cG, \, i \in C} L_C$ & \\  
  $\LGconst$ & $\max_{i} \LGi$  (= $\cL_1$ for  $\tau$--uniform $S$ with $c_1$--uniform support) & Sec~\ref{sec:intro-summary-of-results}, \ref{sec:98hs98hs8gh}  \& Thm~\ref{lem:Lfhatuni}, \ref{lem:two_ineqXXX}\\  
$\kappa$ & Stochastic condition number &  Sec~\ref{sec:kappa} \& Lem~\ref{lem:stochcondition} \& Thm~\ref{theo:convgen} \\
$\rho$ & Sketch residual & \eqref{eq:rhointro} \& Thm~\ref{theo:convgen} \& Lem~\ref{lem:gradient_bounddeltaXX} \\
$\Psi^k$ / $\Psi_S^k$ & Lyapunov function / stochastic Lyapunov function & \eqref{eq:lyapgen} /  \eqref{eq:stochlyap} \\  
  $c_1 $ & $|\{C \, : \, C \in \support(S), \, 1 \in C\}|$ & Def~\ref{def:unif_support} \\
  $c_2$ &  $|\{C \, : \, C \in  \support(S), \, 1\in C; 2\in C\}|$ & \eqref{eq:c_2_xxx} \\
  \hline
\end{tabular}

\end{center}

\caption{Frequently used notation.}
\label{tbl:notation}
\end{table}

}

\end{document}